\documentclass[10pt,11pt]{amsart}%
\usepackage[dvips]{graphicx}
\usepackage{amsmath}
\usepackage{color}
\usepackage{tikz}
\usepackage{amsfonts}
\usepackage{comment}
\usepackage{amssymb}
\usepackage{hyperref}%
\usepackage{dynkin-diagrams}
\usepackage{textcomp}

\usepackage{slashbox}
\usepackage{url}
\usepackage{xcolor}
\usepackage{pgf,tikz}
\usepackage{placeins}
\usepackage{ytableau}

\usepackage{multicol}
\usepackage{mathtools}
\usepackage{hyperref}
\newtheorem{teo}{Theorem}[section]
\newtheorem{cor}[teo]{Corollary}
\newtheorem{prop}[teo]{Proposition}
\newtheorem{lema}[teo]{Lemma}

\theoremstyle{definition}
\newtheorem{alg}[teo]{Algorithm}
\newtheorem{defin}[teo]{Definition}
\newtheorem{obs}[teo]{Remark}
\theoremstyle{remark}
\newtheorem{ex}[teo]{Example}
\newcommand{\oo}{\color{magenta}}

\newcommand{\N}{\mathbb N}
\newcommand{\Z}{\mathbb Z}

\newcommand{\x}{\boldsymbol{x}}
\newcommand{\y}{\boldsymbol{y}}
\newcommand{\YT}[3]{
	\vcenter{\hbox{
			\begin{tikzpicture}[x={(0in,-#1)},y={(#1,0in)}] 
				\foreach \rowi [count=\i] in {#3} {
					\foreach \e [count=\j] in \rowi {
						\draw (\i,\j) rectangle +(-1,-1);
						\draw (\i-0.5,\j-0.5) node {$#2\e$};
					}
				}
			\end{tikzpicture}
	}}
}
\newcommand{\thiYT}[3]{
	\vcenter{\hbox{
			\begin{tikzpicture}[x={(0in,-#1)},y={(#1,0in)}] 
				\foreach \rowi [count=\i] in {#3} {
					\foreach \e [count=\j] in \rowi {
						\draw[ultra thick] (\i,\j) rectangle +(-1,-1);
						\draw (\i-0.5,\j-0.5) node {$#2\e$};
					}
				}
			\end{tikzpicture}
	}}
}

\newcommand{\SYT}[3]{
	\vcenter{\hbox{
			\begin{tikzpicture}[x={(0in,-#1)},y={(#1,0in)}] 
				\foreach \rowi [count=\i] in {#3} {
					\foreach \e [count=\j] in \rowi {
						\draw (\i,-\j) rectangle +(-1,-1);
						\draw (\i-0.5,-\j-0.5) node {$#2\e$};
					}
				}
			\end{tikzpicture}
	}}
}

\usepackage[backend=bibtex]{biblatex}
\begin{document}
	\title[Cocrystals, symplectic keys and virtual keys]{Cocrystals of symplectic Kashiwara-Nakashima tableaux, symplectic Willis like direct way, virtual keys and applications}
	\author{Olga Azenhas and Jo\~ao Miguel Santos}
\address{University of Coimbra, Department of Mathematics, CMUC, 3000-143 Coimbra, Portugal}\email{oazenhas@mat.uc.pt}
 \email{jmsantos@mat.uc.pt}

	\begin{abstract}We attach a  $\mathfrak{sl}_2$ crystal, called cocrystal, to a symplectic Kashiwara-Nakashima (KN) tableau,  whose vertices are skew KN tableaux connected via the Lecouvey-Sheats symplectic \emph{jeu de taquin}. These cocrystals contain all the needed information to compute right and left keys of a symplectic KN tableau. Motivated by Willis' direct way of computing type $A$ right and left keys, we also give a way of computing symplectic, right and left, keys without the use of the symplectic \emph{jeu de taquin}. On the other hand, we  prove that Baker virtualization by folding $A_{2n-1}$ into $C_n$ commutes with dilatation of crystals. Thus we may alternatively utilize this Baker virtualization
to  embed  a type $C_n$ Demazure crystal, its opposite and atoms into  $A_{2n-1}$ ones. The right, respectively left keys of a KN tableau are thereby computed as $A_{2n-1}$ semistandard tableaux  and returned back via reverse Baker embedding to the $C_n$ crystal as its right respectively  left symplectic keys. In particular,  Baker embedding also virtualizes the crystal of Lakshmibai-Seshadri paths as $B_n$-paths into the crystal  of Lakshmibai-Seshadri paths as $\mathfrak{S}_{2n}$-paths.
Lastly, as an application of our explicit symplectic right and left key maps,
thanks to the isomorphism between  Lakshmibai-Seshadri path  and Kashiwara crystals we use, similarly to the ${{Gl}(n,\mathbb{C})}$ case, left and right key maps as  a tool to test whether a symplectic KN tableau is \emph{standard} on a  Schubert or Richardson variety  in the flag variety $Sp(2n,\mathbb{C})/B$, with $B$ a Borel subgroup. 
	\end{abstract}
	
	\maketitle
	
	\section{Introduction}
	Symplectic tableaux \cite{king1976weight, de1979symplectic, Kashiwara1994CrystalGF} provide the monomial weight generators for the characters of the symplectic Lie algebra $\mathfrak{sp}(2n, \mathbb{C})$. Given a partition $\lambda \in \mathbb{Z}_{\ge 0}^n$, symplectic Kashiwara-Nakashima (KN) tableaux of shape $\lambda$, on the alphabet $[\pm n]$ \cite{Kashiwara1994CrystalGF}, a variation of De Concini tableaux { in symplectic \emph{standard} monomial theory} \cite{de1979symplectic}, are endowed with a type $C_n$ Kashiwara crystal structure $\mathfrak{B}(\lambda)$ compatible with a plactic monoid  and sliding algorithms, studied by Lecouvey in terms of crystal isomorphisms \cite{lecouvey2002schensted}.

Let $G=Sp(2n,\mathbb{C})$ be the symplectic group. Fix $H\subseteq B\subseteq G$, $H$ a maximal torus,   $B$ a  Borel subgroup, and let $W$ be  the associated Weyl group  identified  as a   Coxeter group  with $B_n$ (hyperoctahedral group) with longest element $w_0=-Id$. Let the Lie algebras of $G$, $B$ and  $H$ be $\mathfrak{g}=\mathfrak{sp}(2n, \mathbb{C})$,  $\mathfrak{b}$, a Borel subalgebra of $\mathfrak{g}$, and $\mathfrak{h}$, a  Cartan subalgebra of $\mathfrak{g}$, respectively. Let $V(\lambda)$ be the irreducible $G$-module with highest weight $\lambda$.
 The Kashiwara crystal $\mathfrak{B}(\lambda)$ is a combinatorial skeleton for the $G$-module $V(\lambda)$. Another combinatorial skeleton is the Littelmann crystal of Lakshmibai-Seshadri (L-S) paths of shape $\lambda$, denoted $\mathbf{B}(\lambda)$, isomorphic to the Kashiwara crystal $\mathfrak{B}(\lambda)$. For $w\in W$, the Demazure module $V_w(\lambda)\subseteq V(\lambda)$ is the $B$-submodule defined $V_w(\lambda)=\mathcal{U}(\mathfrak{b}).V(\lambda)_{w\lambda}$, where $\mathcal{U}(\mathfrak{b})$ is the enveloping algebra of the Borel subalgebra $\mathfrak{b}$ of $\mathfrak{g}$, and $V(\lambda)_{w\lambda}$ is the one dimensional weight space of $V(\lambda)$ with extremal weight $w\lambda$. The Demazure module $V_e(\lambda)$ is just one-dimensional generated by  the highest weight vector of $V(\lambda)$ and $V_{w_0}(\lambda)=V(\lambda)$.

{ Key polynomials or Demazure characters are the characters of the Demazure modules $V_w(\lambda)$.
Let $W\lambda$ be the orbit of $\lambda$ with the induced Bruhat order, and $u,v\in W\lambda$.  Kashiwara \cite{kashiwara1992crystal} and Littelmann \cite{littelmann1995crystal} have shown that they can be obtained by summing the monomial weights over certain subsets $\mathfrak{B}_v=\mathfrak{B}_w(\lambda)$, $v=w\lambda\in W\lambda$, in the crystal $\mathfrak{B}(\lambda)$, called Demazure crystals. Demazure crystals $\mathfrak{B}_v$ can be partitioned into Demazure crystal atoms, $\overline{\mathfrak{B}}_u$, where $u\in W\lambda$ runs in the Bruhat interval $\lambda\le u\le v$ of $W\lambda$. Therefore the $\dim V_w(\lambda)$ is the cardinality of $\mathfrak{B}_v$, $v=w\lambda$.}

 Considering $G/B$ the (full) flag variety, $G$ is a semi-simple algebraic group over $k$ closed algebraic field, where we fix $H\subseteq B\subseteq G$ and $W$ its Weyl group equipped with the (strong) Bruhat order, with $w_0$ is the long element, we have also for any $w\in W$ the corresponding Schubert variety $X_w=\bigsqcup_{\tau\le w}B\tau B/B\subseteq G/B$ where $G=\bigsqcup_{ \tau\in W}B\tau B$ is the Bruhat decomposition of $G$.
The Borel-Weil theorem provides a geometric  interpretation of  Demazure and opposite Demazure modules showing that they are  in natural correspondence with Schubert respectively opposite Schubert varieties also compatible with restrictions and intersections. For any $\lambda\in \mathbb{Z}_{\ge 0}^n$ a partition, we have $L_\lambda$  the line bundle on $G/B$ and its restriction (denoted by the same symbol) to $X_w$.
There is a $G$ (resp. $B$)-module isomorphism between the dual module $V^*(\lambda)$ (resp. dual submodule $V^*_w(\lambda)$) of $V(\lambda)$  and the space of global sections (also called the zero degree cohomology of the sheaf of sections of a line bundle) $H^0(G/B,L_\lambda)$ (resp. $ H^0(X_w,L_\lambda)$):  $V^*(\lambda)=V(-w_0\lambda)\simeq H^0(G/B,L_\lambda)$, and  
$V^*_w(\lambda)=V(-w_0\lambda)\simeq H^0(X_w,L_\lambda)$. When $G=Sp(2n,\mathbb{C})$ it holds $V^*(\lambda)=V(w_0\lambda)\simeq H^0(G/B,L_\lambda)$, and  
$V^*_w(\lambda)=V(\lambda)\simeq H^0(X_w,L_\lambda)$.  In particular, for $w=w_0$, respectively $w=e$, one has $X_{w_0}=G/B$,  respectively  $X_e=B$,
and  $V_{w_0}(\lambda)=V(\lambda)$, respectively $V_{e}(\lambda)=\{b_\lambda\}$.
For any $\tau\le w$ in $W$, $X_\tau\subseteq X_w$ and the restriction map $H^0(G/B,L_\lambda)\longrightarrow H^0(X_\tau,L_\lambda)$ is surjective. Hence $V_\tau(\lambda)\subseteq V_w(\lambda)$. (We refer to \cite{laklitt,lakrag} and references therein.)

Kashiwara  has  constructed a specific $\mathbb{C}$-basis for the irreducible highest weight $\mathfrak{g}$-module $V(\lambda)$ via the quantized enveloping  Lie algebra.   More precisely, the specialization $q=1$ in the Kashiwara  lower global basis (= Lusztig canonical basis \cite{lusztig})
$\{G_\lambda(b): b \in \mathfrak{B}(\lambda)\}$ \cite{kashiwaraglobal} gives that aforesaid basis for $V(\lambda)$.
We then may also conclude  from  \cite[Proposition 3.2.3 (i), (4.1)]{kashiwara1992crystal}
that the Kashiwara  lower global basis  at $ q=1$ restricts to Demazure and opposite Demazure modules of $V(\lambda)$.
 More precisely, given $w$ in the Weyl group $W$
$\{G_\lambda(b):b \in \mathfrak{B}_w(\lambda)\}$ at $q=1$ gives a basis for the Demazure module $V_w(\lambda)$
and
$\{G_\lambda(b):b \in \mathfrak{B}^w(\lambda)\}$ at $q=1$ gives a basis for the opposite Demazure module $V^w(\lambda)$.
More generally, the restriction of the Global/Canonical basis to vectors labelled
 by vertices in a Demazure crystal of a highest weight crystal gives
 the Global/canonical basis of the corresponding Demazure module.

From a geometric vein, De Concini \cite{de1979symplectic} constructed in the symplectic case  bases indexed by \emph{standard} symplectic tableaux, that is symplectic tableaux in $\mathfrak{B}(\lambda)$ indexing bases for the homogeneous components of the coordinate rings of
the isotropic flag varieties $Sp(2n,\mathbb{C})/B$ and thanks to the Borel-Weil theorem they are bases for the irreducible representations
of $Sp(2n,\mathbb{C})$. The symplectic tableaux in $\mathfrak{B}_w(\lambda)$ ($\mathfrak{B}^w(\lambda)$) index a basis for the coordinate ring of the Schubert variety $X_w$ and, by the Weil-Borel theorem, a basis for the corresponding Demazure modules \cite[Theorem A.4.0.1, Corollary A.4.0.2]{laklitt,lakrag}. These bases are called \emph{standard} monomial bases. For the general case, $G$ a semisimple algebraic group, there are  constructions using the crystal of L-S paths, and we refer to \cite{laklitt} (also here \S \ref{standardlaklitt}). Despite of these bases being indexed by the same combinatorial objects, the relation between Global basis at $q=1$ and \emph{standard} monomial bases is not yet well understood in full generality.

	Keys  in type $A_{n-1}$ have its origin in the $GL(n,\mathbb{C})$ \emph{standard} bases to detect the semistandard tableaux which are \emph{standard} on a Schubert variety. (See \cite{reinershimoz} for an application of keys within \emph{standard} monomial theory.) In type $A_{n-1}$, Lascoux and Sch\"{u}tzenberger characterized key tableaux as semistandard Young tableaux (SSYT) with nested columns \cite[Definition 2.9]{lascoux1990keys}, and have used the \emph{jeu de taquin} to define the \emph{right key respectively left key maps} which sends a SSYT $T$ to a key tableau pair $(K^+(T),K^-(T))$, called the \emph{right key} respectively \emph{left key} of $T$. A tableau is a key if and only if it is equal to its right and left key. In each Demazure crystal atom there exists exactly one key tableau and the right key map detects the Demazure crystal atom that contains a given SSYT \cite[Theorem $3.8$]{lascoux1990keys}. 
	By direct inspection of a Young tableau, Willis \cite{Willis2011ADW} has  given an alternative algorithm to compute the right (respect. left) key of a semistandard tableau that does not require the use of \emph{jeu de taquin}.
	Other methods to compute the type $A$ right or left key maps includes the alcove path model \cite{lenart2015generalization}, semi-skyline augmented fillings \cite{Mason2007AnEC},  coloured vertex models \cite{BRUBAKER2021105354} or \cite{krv23}. For a complete overview in type $A$, see \cite{BRUBAKER2021105354} and  the references therein.
	
	In type $C_n$,  symplectic key tableaux are characterized in \cite{azenhas2012key, santos2019symplectic, santos2020SLCsymplectic,jacon2019keys}. They are the unique tableaux in $\mathfrak{B}(\lambda)$ whose weight is in $W\lambda$, and, for each one, there is exactly one Demazure crystal atom indexed by the corresponding weight. Using the Lecouvey-Sheats symplectic \emph{jeu de taquin}, a right key map is given, in \cite{santos2019symplectic, santos2020SLCsymplectic,jacon2019keys}, to send a Kashiwara-Nakashima tableau $T$ to its right key tableau $K_+(T)$, that detects the Demazure crystal atom which contains $T$.
	They are also computed in the type $C_n$ alcove path model \cite{lenart2015generalization}, and in the coloured  five vertex model \cite{buciumas2021quasi}. In \cite{buciumas2021quasi}, it is also computed the right key for reverse King tableaux \cite{king1976weight}.
	Henceforth, the symplectic Demazure character $\kappa_v(x)$ is expressed in terms of right keys \cite{santos2019symplectic} $$\kappa_v(x)=\sum_{\substack{T\in \mathfrak{B}(\lambda) \\ K_+(T)\leq K(v)}} x^{\text{wt} T},$$ where $K(v)$ is the key tableau of shape $\lambda$ and weight $v$, $x^{\text{wt} T}$ is the weight monomial corresponding to $T$, with  ${\text{wt} T}\in\Z^n$ the weight of $T$, and  $K_+(T)\leq K(v)$ means entrywise comparison.
	
	The  \emph{left key map} is  dual to the  right key map. It detects the tableaux which go in each \emph{opposite} Demazure crystal \cite{choi2018lakshmibai}. Given $v\in W\lambda$, the opposite Demazure crystal $\mathfrak{B}^{-v}$ is the image of Demazure crystal $\mathfrak{B}_v\subseteq \mathfrak{B}(\lambda)$ by the
	the Schützenberger-Lusztig involution on $\mathfrak{B}(\lambda)$ \cite[Proposition 64]{santos2019symplectic}. The crystal  $\mathfrak{B}(\lambda)$ can also be partitioned into opposite Demazure crystal atoms, $\overline{\mathfrak{B}}^u$, where $u\in W\lambda$ runs in the Bruhat interval $-v\le u\le -\lambda$.  Given a Demazure crystal and its opposite, for each tableau weight in the Demazure crystal there is a symmetric tableau weight in the opposite Demazure crystal.
	
	Motivated by Lascoux' double crystal graph in type $A$ \cite{lascoux2003double}, 
	where Sch\"utzenberger \emph{jeu de taquin} slides are used as crystal operators
	, we can attach a type $A$ \emph{cocrystal} to each vertex of the type $C_n$ crystal of Kashiwara-Nakashima tableaux $\mathfrak{B}^\lambda$, in which the crystal operators are given by symplectic \emph{jeu de taquin} slides on consecutive columns.
	Given a Kashiwara-Nakashima tableau $T$ in $\mathfrak{B}(\lambda)$, the vertices of its cocrystal are Kashiwara-Nakashima skew tableaux connected to $T$ via symplectic \emph{jeu de taquin}. Our construction builds on Heo-Kwon work \cite[Lemma 2.3, Lemma 2.4]{kwonheococrystal2020} and
	uses the dual RSK correspondence \cite{fulton1997young}. 
	These cocrystals are type $A$ crystals whose elements are type $C_n$ Kashiwara-Nakashima skew tableaux, and contain all the needed information to compute the right and left key maps. Actually, the key skew tableaux of the cocrystal (Definition \ref{Defcokey}) provide all the needed information to compute left and right key maps of a Kashiwara-Nakashima tableau. The cocrystals also allow us to generalize the Proposition 7 \cite[Appendix A.5]{fulton1997young} from semistandard Young tableaux to type $C_n$ Kashiwara-Nakashima tableaux.
	This is an analogue of LS-paths which carry explicitly right and left keys as initial and final directions.
	
	Jacon and Lecouvey have suggested, in \cite{jacon2019keys}, that Willis' method \cite{Willis2011ADW}  to compute right and left keys in type $A_{n-1}$ should be adaptable to type $C_n$.   Motivated by Willis' direct inspection \cite{Willis2011ADW}, we create an alternative algorithm, based on a Kashiwara-Nakashima tableau, for the symplectic right key map, and for the symplectic left key map, that does not use the symplectic \emph{jeu de taquin}.

	Due to the added technicality of the symplectic \emph{jeu de taquin} compared to the one for SSYT, Willis' \emph{earliest weakly increasing subsequence} will fail to predict what gets slid during the Lecouvey-Sheats symplectic \emph{jeu de taquin}. Instead we need a way to calculate, without the use of \emph{jeu de taquin}, what would appear in each column if we were to swap its length with the previous column length via \emph{jeu de taquin}. The role of Willis' sequences will be replaced by our matchings (see Section \ref{rightkeydw}). In type $A$, these kind of matches were used earlier \cite{azenhas2006schur, lenart2004matching} for jeu de taquin on two columns.

	Lastly we  prove that Baker virtualization \cite{ba00a} by folding $A_{2n-1}$ into $C_n$ commutes with dilatation of crystals \cite{kash2}. Dilatation of crystals provide a constructive bijection between Kashiwara crystals and Littelmann crystals of L-S paths. Thus we may alternatively utilize the Baker virtualization
to  embed  a type $C_n$ Demazure crystal, its opposite and atoms into  $A_{2n-1}$ ones and compute symplectic right and left keys via type A methods.

	The paper is organized in eight sections as follows. In Section \ref{tabjdt}, we discuss the  type $C$ Kashiwara-Nakashima tableaux and the symplectic \emph{jeu de taquin}. Section \ref{SecDemCrys} recalls the combinatorics of Kashiwara crystals, symplectic key tableaux,  right and left key maps in terms of a Demazure crystal and opposite Demazure crystal and their parallel with the model of L-S paths and with Schubert varieties on a flag variety. In Section \ref{coc}  we  attach a type $A$ \emph{cocrystal} to each vertex of the type $C_n$ crystal of Kashiwara-Nakashima tableaux $\mathfrak{B}(\lambda)$.
	Section \ref{rightkeyjdt} 
recalls the right and left key maps via symplectic \emph{jeu de taquin} in \cite{santos2019symplectic}, as a preparation for the alternative method without jeu de taquin. In Section \ref{sec:directway}, we give an algorithm for computing  the symplectic right and left key maps that does not require the \emph{jeu de taquin}, and prove that it returns the same object as the previous method.
We end this Section, \ref{SecExample},  with an illustrative example of our new algorithm and the one based on the Lecouvey-Sheats \emph{jeu de taquin}. In Section \ref{sec:virtual} we show that Baker virtualization \cite{ba00a} commutes with dilatation of crystals \cite{kash2} and virtualize the symplectic right and left keys maps. That is, symplectic keys can be computed via type A methods and returned back to the symplectic setting via reverse Baker embedding. We end this section by illustrating   this method  for symplectic Kashiwars crystals. As for the vitualization of the crystal of Lakshmibai-Seshadri paths for the Weyl group in type $C$, we explain the recipe as the size of the needed dilatation is in general big.
Section \ref{sec:final} makes a remark on a another possible method to compute symplectic keys.
	
\bigskip

	\emph{An extended abstract \cite{santos2021fpsac} of a part of this paper, here \S \ref{rightkeydw},  by the second author (JMS),  was accepted in the Proceedings of
		the 33rd Conference on Formal Power Series and Algebraic Combinatorics,
		2021. Also  the  content of \S \ref{sec:directway}   and  partially of \S \ref{coc} appear in the PhD thesis \cite{santos2022} by the second author (JMS)}.
	
	\section{Type $C$ Kashiwara-Nakashima tableaux and \emph{jeu de taquin}}\label{tabjdt}
	
	We recall the symplectic tableaux introduced by Kashiwara and Nakashima to label the vertices of the type $C_n$ crystal graphs \cite{Kashiwara1994CrystalGF}.
	Fix $n\in \mathbb{N}_{>0}$. Define the sets $[n]=\{1, \dots, n\}$ and $[\pm n]=\{1,\dots, n, \overline{n}, \dots, \overline{1}\}$ where $\overline{i}$ is just another way of writing $-i$, hence $\overline{\overline{i}}=i$. In the second set we will consider the following order of its elements: $1<\dots< n< \overline{n}< \dots< \overline{1}$ instead of the usual order.
	A vector $\lambda=(\lambda_1,\dots, \lambda_n)\in\mathbb{Z}^n$ is a partition of $|\lambda|=\sum\limits_{i=1}^n \lambda_i$ with at most $n$ parts if $\lambda_1\geq \lambda_2\geq\dots\geq\lambda_n\geq 0$. Let $\mathcal{P}_n$ be the set of partitions $\lambda=(\lambda_1,\dots, \lambda_n)$ with at most $n$ parts.
	The \emph{Young diagram} of shape $\lambda$, in English notation, is an array of boxes (or cells), left justified, in which the $i$-th row, from top to bottom, has $\lambda_i$ boxes. We identify a partition with its Young diagram.
	For example, the Young diagram of shape $\lambda=(2,2,1)$ is $\YT{0.17in}{}{
		{{},{}},
		{{},{}},
		{{}}}$.
	Given $\mu$ and $\nu$ two partitions with $\nu\leq \mu$ entrywise, we write $\nu\subseteq \mu$. The Young diagram of shape $\mu/\nu$ is obtained after removing the boxes of the Young diagram of $\nu$ from the Young diagram of $\mu$.
	For example, the Young diagram of shape $\mu/\nu=(2,2,1)/(1,0,0)$ is $\begin{tikzpicture}[scale=.4, baseline={([yshift=-.8ex]current bounding box.center)}]
		\draw (1,0) rectangle +(-1,-1);
		\draw (1,-1) rectangle +(-1,-1);
		\draw (0,-2) rectangle +(-1,-1);
		\draw (0,-1) rectangle +(-1,-1);
	\end{tikzpicture}$.
	
	Let $\nu\subseteq \mu$ be two partitions and $A$ a completely ordered alphabet. A \emph{semistandard Young tableau} (SSYT) of skew shape $\mu/\nu$, on the alphabet $A$, is a filling of the diagram $\mu/\nu$ with letters from $A$, such that the entries are strictly increasing, from top to bottom, in each column and weakly increasing, from left to right, in each row. When $|\nu|=0$ then we obtain a semistandard Young tableau of straight shape $\mu$.
	Denote by $\mathcal{SSYT}(\mu/\nu, A)$
	the set of all skew SSYT's $T$ of shape $\mu/\nu$, with entries in $A$. In particular, when $\left|v\right|=0$ we write $\mathcal{SSYT}(\mu, A)$ and when $A=[n]$ we write $\mathcal{SSYT}(\mu/\nu, n)$.
	
	When considering tableaux with entries in $[\pm n]$, it is usual to have some extra conditions besides being semistandard. We will use a family of tableaux known as \emph{Kashiwara-Nakashima} tableaux.
	From now on we consider tableaux on the alphabet $[\pm n]$.
	
	A \emph{column} is a strictly increasing sequence of numbers (or letters) in $[\pm n]$ and it is usually displayed vertically. The height of a column is the number of letters in it.
	A column is said to be \emph{admissible} if the following \emph{one column condition} (1CC) holds for that column:
	
	\begin{defin}[1CC]\label{1CC}
		Let $C$ be a column. The $1CC$ holds for $C$ if for all pairs $i$ and $\overline{i}$ in $C$, where $i$ is in the $a$-th row counting from the top of the column, and $\overline{i}$ in the $b$-th row counting from the bottom, we have $a+b\leq i$. Equivalently, for all pairs $i$ and $\overline{i}$ in $C$, the number $N(i)$ of letters $x$ in
		$C$ such that $x\leq i$ or $x \geq \overline{i}$ satisfies $N(i) \leq i$.
	\end{defin}
	If a column $C$ satisfies the $1CC$ then $C$ has at most $n$ letters.
	If $1CC$ doesn't hold for $C$ we say that $C$ \emph{breaks the $1CC$ at $z$}, where $z$ is the minimal positive integer such that $z$ and $\overline{z}$ exist in $C$ and there are more than $z$ numbers in $C$ with absolute value less or equal than $z$.
	\begin{ex}
		The column $\YT{0.17in}{}{
			{{1}},
			{{2}},
			{{\overline{1}}}}$ breaks the $1CC$ at $1$, and $\YT{0.17in}{}{
			{{2}},
			{{3}},
			{{\overline{3}}}}$ is an admissible column.
	\end{ex}
	
	
	The following definition states conditions to when $C$ can be \emph{split}:
	\begin{defin}\label{Defsplit}
		Let $C$ be a column and let $I = \{z_1 > \dots > z_r\}$ be the
		set of unbarred letters $z$ such that the pair $(z, \overline{z})$ occurs in $C$. The column
		$C$ can be split when there exists a set of $r$ unbarred
		letters $J = \{t_1 > \dots > t_r\} \subseteq [n]$ such that:
		
		{\boldmath 1.} $t_1$ is the greatest letter of $[n]$ satisfying $t_1 < z_1$,  $t_1 \not\in C$, and $\overline{t_1}\not\in C$,
		
		{\boldmath{2.}} for $i=2, \dots, r$, we have that  $t_i$ is the greatest letter of $[n]$ satisfying $t_i < \min(t_{i-1},   z_i)$,  $t_i \not\in C$, and $\overline{t_i} \not\in C$.
	\end{defin}
	
	The $1CC$ holds for a column $C$ (or $C$ is admissible) if and only if $C$ can be split  \cite[Lemma 3.1]{sheats1999symplectic}.
	If $C$ can be split then we define \emph{right column} of $C$, $\text{r}C$, and the \emph{left column} of $C$, $\ell C$, as follows:
	
	{\boldmath{1.}} $rC$ is the column obtained by changing in $C$,  $\overline{z_i}$ into $\overline{t_i}$ for each letter $z_i \in I $ and by reordering if necessary,
	
	{\boldmath{2.}} $\ell C$ is the column obtained by changing in $C$, $z_i$ into $t_i$ for each letter $z_i \in I $ and by reordering if necessary.

	If $C$ is admissible then $\ell C\leq C \leq rC$ by entrywise comparison, where $\ell C$ has the same barred part as $C$ and $rC$ the same unbarred part. If $C$ doesn't have symmetric entries, then $C$ is admissible and  $\ell C= C =rC$.
	In the next definition we give conditions for a column $C$ to be \emph{coadmissible}.
	\begin{defin}
		We say that a column $C$ is coadmissible if for every pair $i$ and $\overline{i}$ on $C$, where $i$ is on the $a$-th row counting from the top of the column, and $\overline{i}$ on the $b$-th row counting from the top, then $b-a\leq n-i$.  Equivalently, for every pair $i$ and $\overline{i}$ on $C$, the number $N^\ast (i)$ of letters $x$ in
		$C$ such that $i \leq x \leq \overline{i}$ satisfies $N^\ast (i) \leq n-i + 1$.
	\end{defin}
	\noindent	Unlike in Definition \ref{1CC}, in the last definition $b$ is counted from the top of the column.
	
	\begin{defin}
		Let $C$ be a column and let $I = \{z_1 > \dots > z_r\}$ be the
		set of unbarred letters $z$ such that the pair $(z, \overline{z})$ occurs in $C$. The column
		$C$ is coadmissible if and only if there exists a set of $r$ unbarred
		letters $H = \{h_1 > \dots > h_r\} \subseteq [n]$ such that:
		
		{\boldmath 1.} $h_r$ is the smallest letter of $[n]$ satisfying $h_r > z_r$,  $h_r \not\in C$, and $\overline{h_r}\not\in C$,
		
		{\boldmath 2.} for $i=r-1, \dots, 1$, we have that  $h_i$ is the smallest letter of $[n]$ satisfying $h_i > \max(h_{i+1},   z_i)$,  $h_i \not\in C$, and $\overline{h_i} \not\in C$.
	\end{defin}
	
	Given an admissible column $C$, consider the map $$\Phi:C\mapsto C^\ast$$ that sends $C$ to the column $C^\ast$ of the same size in which the unbarred entries are taken from $\ell C$ and the barred entries are taken from $rC$.
	
	\begin{lema}
		Let $C$ be an admissible column on the alphabet $[\pm n]$, and $I$ and $J$ the sets in Definition \ref{Defsplit}. The entries $x$ (barred or unbarred) of $\Phi(C)$ are such that
		\begin{enumerate}
			\item  $x \in \Phi(C)$ and $\overline{x} \notin \Phi(C)$ if and only if $x \in C$ and $\overline{x} \notin C$.
			\item $x, \overline{x} \in \Phi(C)$ if and only if  $x \in J$ or $\overline{x}\in J$.
		\end{enumerate}
		Equivalently, the set of entries in $\Phi(C)$ is $(J \cup \overline{J} \cup C) \setminus (I \cup \overline{I})$.\end{lema}
	Henceforth, $\Phi(C) = C$ if and only if $I = \emptyset$ (hence $J = \emptyset)$, that is, $C$ does not have symmetric entries.
	
	The \emph{column $\Phi(C)$} is a coadmissible column and the algorithm to form $\Phi(C)$ from $C$ is reversible \cite[Section 2.2]{lecouvey2002schensted}. In particular, every column on the alphabet $[n]$ is simultaneously admissible and coadmissible. The map $\Phi$ is a bijection between admissible and coadmissible columns of the same height on the alphabet $[\pm n]$.
	
	\begin{ex}Let
		$C=\!\YT{0.17in}{}{
			{{2}},
			{{4}},
			{{\overline{2}}}}$ be an admissible column, so it can be split. Then $\ell C=\YT{0.17in}{}{
			{{1}},
			{{4}},
			{{\overline{2}}}}$ and $rC=\YT{0.17in}{}{
			{{2}},
			{{4}},
			{{\overline{1}}}}$. So $\Phi(C)=\YT{0.17in}{}{
			{{1}},
			{{4}},
			{{\overline{1}}}}$ is coadmissible. $C$ is also coadmissible and $\Phi^{-1}(C)=\YT{0.17in}{}{
			{{3}},
			{{4}},
			{{\overline{3}}}}$.
	\end{ex}
	
	Let $T$ be a skew tableau with all of its columns admissible. The \emph{split form} of a skew tableau $T$, $spl(T)$, is the skew tableau obtained after replacing each column $C$ of $T$ by the two columns $\ell C\,rC$. The tableau $spl(T)$ has double the amount of columns of $T$.

	A semistandard skew tableau $T$ is a \emph{Kashiwara-Nakashima (KN) skew tableau} if its split form is a semistandard skew tableau. We define $\mathcal{KN}(\mu/\nu, n)$ to be the set of all KN tableaux of shape $\mu/\nu$ in the alphabet $[\pm n]$.  When $\nu=0$, we obtain $\mathcal{KN}(\mu, n)$.
{ If $T$ is a skew KN tableau, the \emph{column reading} of $T$, $cr(T)$, is
the word read in $T$ in the Chinese/Japanese way, column reading top to bottom and right to left.
The \emph{length} of ${w}$ is the total number of letters in ${w}$. The weight of a KN tableau $T$ is the vector $\text{wt}( T):=\text{wt}(cr(T))$ in $\Z^n$ whose $i$-th entry is the number of $i$'s minus the number of $\overline{i}$ for $i\in [n]$ (see \S \ref{sec:tensor}).
}

	\begin{ex} Let $n=3$. The split of the tableau
		$T=\YT{0.17in}{}{
			{{2},{2}},
			{{3},{3}},
			{{\overline{3}}}}$ is the tableau $spl(T)=\YT{0.17in}{}{
			{{1}, {2},{2},{2}},
			{{2},{3},{3},{3}},
			{{\overline{3}},{{\overline{1}}}}}$. Hence $T\in \mathcal{KN}((2,2,1),3)$, $cr(T)=2323\overline{3}$ and weight $\text{wt}( T)=(0,2,1)$.
	\end{ex}

	If $T$ is a tableau without symmetric entries in any of its columns, i.e., for all $i\in[n]$ and for all columns $C$ in $T$, $i$ and $\overline{i}$ do not appear simultaneously in the entries of $C$, then in order to check whether $T$ is a KN tableau it is enough to check whether $T$ is semistandard in the alphabet $\left[\pm n\right]$. In particular $\mathcal{SSYT}(\mu/\nu, n)\subseteq \mathcal{KN}(\mu/\nu, n)$.
	
	\subsection{Symplectic \textit{jeu de taquin}}
	Lecouvey-Sheats symplectic \textit{jeu de taquin} (SJDT) \cite{lecouvey2002schensted, sheats1999symplectic} is a procedure on KN skew tableaux, compatible with \emph{Knuth equivalence} (or plactic equivalence on words over the alphabet $[\pm n]$) \cite{lecouvey2002schensted}, that allows us to change the shape of a tableau and to rectify it.
	To explain how the SJDT behaves, we need to look how it works on $2$-column $C_1C_2$ KN skew tableaux. A skew tableau is \emph{punctured} if one of its box contains
	the symbol $\ast$ called the \emph{puncture}.
	A punctured column
	is admissible if the column is admissible when ignoring the puncture. A punctured skew tableau is
	admissible if its columns are admissible and the rows of its split form are weakly increasing ignoring
	the puncture. Let $T$ be a punctured skew tableau with two columns $C_1$ and $C_2$ with the puncture
	in $C_1$. In that
	case, the puncture splits into two punctures in $spl(T)$, and ignoring the punctures, $spl(T)$ must be semistandard.
	Let $\alpha$ be the entry under the puncture of $rC_1$, and $\beta$ the entry to the right of the puncture of $rC_1$.
	$$spl(T)=\ell C_1 rC_1\ell C_2rC_2=\YT{0.22in}{}{
		{{\dots},{\dots}, {\dots}, {\dots}},
		{{\ast},{\ast},{\beta},{\dots}},
		{{\dots},{\alpha},{\dots},{\dots}},
		{{\dots},{\dots}}},$$
	where $\alpha$ or $\beta$ may not exist. The elementary steps of SJDT are the following:
	
	\textbf{A.} If $\alpha\leq \beta$ or $\beta$ does not exist,  then the puncture of $T$ will change its position with the cell beneath it. This is a vertical slide.

	\textbf{B.}	If the slide is not vertical, then it is horizontal. So we have $\alpha> \beta$ or $\alpha$ does not exist. Let $C_1'$ and $C_2'$ be the columns obtained after the slide. We have two subcases, depending on the sign of $\beta$:
	
	{\boldmath 1.} If $\beta$ is barred, we are moving a barred letter, $\beta$, from $\ell C_2$ to the punctured box of $rC_1$, and the puncture will occupy $\beta$'s place in $\ell C_2$. Note that $\ell C_2$ has the same barred part as $C_2$ and that $rC_1$ has the same barred part as $\Phi(C_1)$. Looking at $T$, we will have an horizontal slide of the puncture, getting $C_2'=C_2\setminus \{\beta\}\sqcup\{\ast\}$ and $C_1'=\Phi^{-1}(\Phi(C_1)\setminus\{\ast\}\sqcup \{\beta\})$. In a sense, $\beta$ went from $C_2$ to $\Phi(C_1)$.
	
	{\boldmath 2.} If $\beta$ is unbarred, we have a similar story, but this time $\beta$ will go from $\Phi(C_2)$ to $C_1$, hence $C_1'=C_1\setminus\{\ast\}\cup \{\beta\}$ and $C_2'=\Phi^{-1}(\Phi(C_2)\setminus \{\beta\}\sqcup\{\ast\})$. Although in this case it may happen that $C_1'$ is no longer admissible. In this situation, if the 1CC breaks at $i$, we erase both $i$ and $\overline{i}$ from the column and remove a cell from the bottom and from the top column, and place all the remaining cells orderly with respect to their entries.
	
	Applying successively elementary SJDT slides, eventually, the puncture will be a cell such that $\alpha$ and $\beta$ do not exist. In this case we redefine the shape to not include this cell and the \textit{jeu de taquin} ends.
	
	Given an admissible tableau $T$ of shape $\mu/\nu$, a box of the diagram of shape $\nu$ such that boxes under it and to the right are not in that shape is called an inner corner of $\mu/\nu$. An outside corner is a box of $\mu$ such that boxes under it and to the right are not in the shape $\mu$. The rectification of $T$ consists in playing the SJDT until we get a tableau of shape $\lambda$, for some partition $\lambda$. More precisely, apply successively elementary SJDT steps
	to $T$ until each cell of $\nu$ becomes an outside corner. At the end, we obtain a KN tableau for some
	shape $\lambda$. The rectification is independent of the order in which the inner corners of $\nu$ are filled \cite[Corollary 6.3.9]{lecouvey2002schensted}.
	
	\begin{ex} Consider the KN skew tableau
		$T=\SYT{0.145in}{}{
			{{2}},
			{{3},{1}},
			{{\overline{1}},{2}}}$.
		Let $C_1$ and $C_2$ be the first and second columns of $T$.	To rectify $T$ via symplectic \textit{jeu taquin}, one creates a puncture in the inner corner of $T$ and, by splitting, one obtains $\YT{0.17in}{}{
			{{\ast}, {\ast},{2},{2}},
			{{1},{1},{3},{3}},
			{{2},{2},{\overline{1}},{\overline{1}}}}$. So, the first two slides are vertical, obtaining $\YT{0.17in}{}{
			{{1}, {1},{2},{2}},
			{{2},{2},{3},{3}},
			{{\ast},{\ast},{\overline{1}},{\overline{1}}}}$. Finally, we do an horizontal slide, of type $\mathbf{B.1}$, in which we take $\overline{1}$ from $C_2$, and add it to the coadmissible column $\Phi(C_1)$. That is, $C_2'=(C_2\cup \{\ast\})\setminus{\overline{1}}$ and $C_1'=\Phi^{-1}((\Phi(C_1)\setminus\{\ast\})
		\cup {\overline{1}})$, obtaining the tableau
		$\YT{0.145in}{}{
			{{2},{2}},
			{{3},{3}},
			{{\overline{3}}}}$.
	\end{ex}
	
	
	Let $T$ be a {\ KN} skew tableau of shape $\mu/\nu$ { ($\nu$ possibly empty)}. Consider a punctured box that can be added to $\mu$, so that $\mu\cup \{\ast\}$ is a valid shape.
	The SJDT is reversible, meaning that we can move $\ast$, the empty cell outside of $\mu$, to the inner shape $\nu$ of the skew tableau $T$, simultaneously increasing both the inner and outer shapes of $T$ by one cell. The slides work similarly to the previous case: the vertical slide means that an empty cell is going up and an horizontal slide means that an entry goes from $\Phi(C_1)$ to $C_2$ or from $C_1$ to $\Phi(C_2)$, depending on whether the slid entry is barred or not, respectively.
	We will also call the \emph{reverse jeu de taquin} as SJDT. In the next sections we will be mostly dealing with the \emph{reverse jeu de taquin}.
	Consider the following examples, each containing a tableau and a punctured box that will be slid to its inner shape:
	$\YT{0.17in}{}{
		{{},{},{\ast}},
		{{1},{\overline{1}}},
		{{2}}} \mapsto \YT{0.17in}{}{
		{{},{},{}},
		{{1},{\overline{1}}},
		{{2}}}$;\quad
	$\YT{0.17in}{}{
		{{},{}},
		{{1},{\overline{1}}},
		{{2},{\ast}}}\mapsto \YT{0.17in}{}{
		{{},{}},
		{{},{2}},
		{{2},{\overline{2}}}}$.
	\medskip
	
	\begin{obs}\label{equaljdt}If a tableau with columns $C_1$ and $C_2$ does not have symmetric entries then the SJDT applied to $C_1C_2$ coincides with the \textit{jeu de taquin} known for SSYT's.
\end{obs}
	In sections \ref{rightkeyjdt} and \ref{leftkeyjdt}, we use SJDT to swap lengths of consecutive columns in a skew tableau,  to obtain skew tableaux Knuth related to a straight tableau, which is minimal for the number of cells within its Knuth class. Recall that in the elementary step $B.2$ it is possible to lose cells. If we do a reverse elementary step $B.2$ that results in having two more cells in the skew tableau, we would have to start by adding two symmetric entries to an admissible column, making it non admissible \cite[Lemma 3.2.3]{lecouvey2002schensted}, and then slide an unbarred cell to the column to its right. For instance, consider the following reverse elementary step $B.2$ ($\equiv$ denotes type $C_n$ Knuth equivalence \cite[Definition $3.2.1$]{lecouvey2002schensted}):
	$$ \YT{0.17in}{}{
		{{1},{\ast}},
		{{2}}} \equiv
	\YT{0.17in}{}{
		{{1},{\ast}},
		{{2}},
		{{3}},
		{\overline{3}}}
	\equiv
	\YT{0.17in}{}{
		{{\ast},{1}},
		{{2}},
		{{3}},
		{\overline{3}}}.$$
	
	The first and last skew tableaux are Knuth equivalent, but the middle tableau is not a KN skew tableau. The three semistandard tableaux are Knuth equivalent column words, via te contractor/dilator Knuth relation \cite[Definition $3.2.1$]{lecouvey2002schensted}.

	Hence, a reverse elementary step $B.2$ that results in having more cells in the skew tableau has to be forced, since we have to start by forcing the existence of a non admissible column. This means that if we start with a minimal skew tableau, that is, a skew-tableau with the number of cells of its rectification, we can play SJDT, or its reverse, without ever incur in a loss/gain of boxes.
	
	\section{Crystals, keys,  Demazure crystals and their opposite 
}\label{SecDemCrys}
	In this section we review Demazure and opposite Demazure crystals and their atoms of  a crystal in types $A_{n-1}$ respectively $C_n$ crystals $\mathfrak{B}(\lambda)$,  where $\lambda\in \mathcal{P}_n$, and { their detection}  with right and left key maps in \cite{lascoux1990keys,santos2019symplectic}. We also recall how right and left keys surface in appropriate dilatation of crystals \cite{kash2, Kash95, Kashsimilar96}, and apply it to the  characterization of the intersection of Demazure and opposite Demazure crystals \cite{kashiwara1992crystal}. In addition Demazure and opposite Demazure crystals and their intersections are in natural correspondence correspondence with Schubert, opposite Schubert varieties and Richardson varieties, as explained by the classical Borel-Weil theorem, which in turn amounts to the relevance of key maps in standard monomial theory as originally considered by Lascoux-Sch\"utzenberger in \cite{lascoux1990keys}.

	\subsection{Kashiwara crystal}\label{sec:crystal}
	Let $V$ be an Euclidean space with inner product $\langle\cdot,\cdot\rangle$.
	Fix a root system $\Phi$ with simple roots $\{\alpha_i \mid i\in I\}$ where $I$ is an indexing set and a weight lattice $\Lambda\supseteq\Z\text{-span}\{\alpha_i \mid i\in I\}$. A \emph{Kashiwara crystal} of type $\Phi$ is a nonempty set $\mathfrak{B}$ together with maps \cite{bump2017crystal}:
	\begin{align*}
		e_i, f_i:\mathfrak{B}\rightarrow \mathfrak{B}\sqcup\{0\}\quad
		\varepsilon_i, \varphi_i:\mathfrak{B}\rightarrow \Z\sqcup\{-\infty\}\quad
		\text{\text{w}t}: \mathfrak{B}\rightarrow \Lambda
	\end{align*}
	where $i\in I$ and $0\notin \mathfrak{B}$ is an auxiliary element, satisfying the following conditions:
	\begin{enumerate}
		\item [(a)] if $a, b \in \mathfrak{B}$ then $e_i(a)=b\Leftrightarrow f_i(b)=a$. In this case, we also have \begin{align}\label{weights}\text{wt}(b)=\text{wt}(a)+\alpha_i,\mbox{ and }\; \varepsilon_i(b)=\varepsilon_i(a)-1,\;\; \varphi_i(b)=\varphi_i(a)+1;\end{align}
		\item [(b)] for all $a \in \mathfrak{B}$, we have
\begin{align}\label{lengths}\varphi_i(a)=\langle \textrm{wt}(a),\alpha_i^\vee\rangle+\varepsilon_i(a) \mbox{  with $\alpha_i^\vee= \frac{2\alpha_i}{\langle \alpha_i, \alpha_i\rangle}$}.
\end{align}
	\end{enumerate}
	 Let $\Lambda^+$ denote the set of dominant weights, that is, those $\lambda \in \Lambda$ such that $\langle\lambda, \alpha_i^\vee\rangle\ge 0$, for all $i\in I$. The root systems under consideration in this paper are of types  $A_{n-1}$ or $C_n$, thus $\Lambda=\mathbb{Z}^n$.
	 For type $A_{n-1}$, $I=[n-1]$,  and type $C_n$, $I=[n]$, and one has $\alpha_i^\vee= \alpha_i=\textbf{e}_i-\textbf{e}_{i+1}$, for $1\le i<n$, and,  for $i=n$, in type $C_n$, $\alpha_n=2 \textbf{e}_n$ and $\alpha_n^\vee=\textbf{e}_n$ with $\textbf{e}_i,\,i\in I$, the standard basis of $\mathbb{R}^n$. Thus $\Lambda^+=\mathcal{P}_n$  in type $C_n$, and $\Lambda^+=\{\lambda=(\lambda_1,\dots,\lambda_n)\in \mathbb{Z}^n: \lambda_1\ge\cdots\ge\lambda_n\ge 0\}$ in type $A_{n-1}$.

	The crystals we deal with are seminormal \cite{kash2, Kash95, bump2017crystal}, i.e., $\varphi_i(a)=\max\{k\in \Z_{\geq 0}\mid f_i^k(a)\neq 0\}$ and  $\varepsilon_i(a)=\max\{k\in \Z_{\geq 0}\mid e_i^k(a)\neq 0\}$.
	An element $u\in \mathfrak{B}$ such that $e_i(u)=0$ for all $i\in I $ is called  a \emph{highest weight element}. A \emph{lowest weight element} is an element $u\in \mathfrak{B}$ such that $f_i(u)=0$ for all $i\in I $.
	We associate with $\mathfrak{B}$ a coloured oriented graph with vertices in $\mathfrak{B}$ and edges labeled by $i\in I$: $b\overset{i}{\rightarrow} b'$ if and only if $b'=f_i(b)$, $i\in I$, $b, b'\in\mathfrak{B}$. This is the \emph{crystal graph} of $\mathfrak{B}$.

For types $A_{n-1}$ and $C_n$, the  Weyl groups are  $W=\mathfrak{S}_n$ respectively $W=B_n$. { The Weyl group $W$ of type $C_n$, known as hyperocthaedral group, is the Coxeter group $B_n$ ($2^nn!$ elements) generated by the involutions $s_1, \dots, s_n$ (simple reflections) subject to relations $(s_is_{i+1})^3 = 1, 1 \le i \le n - 2;\;(s_{n-1}s_n)^4=1;\;(s_is_j)^2 = 1, \,1 \le i < j \le n,\, |i - j| > 1$.
 The subgroup generated by the simple reflections $s_1, \dots, s_{n-1}$ is the symmetric group $\mathfrak{S}_n\subseteq B_n$. The elements of $B_n$ can also be seen as the bijective maps $\sigma$ on $[\pm n]$ such that $\sigma(-i)=$ $-\sigma(i)$ and then $B_n$ can be identified with the group of signed permutations with generators $s_i=(i\,i+1)(\bar i,\overline{i+1}),$ $1\le i<n$, $s_n=(n,\bar n)$.
 The elements of the symmetric group can be identified with the permutation matrices, and if we allow the non-zero entries to be either $1$ or $-1$, we have the elements of $B_n$. The elements of $W$ act on $z = (z_1, \dots , z_n) \in \mathbb{Z}^n$ by $s_i z := (z_1, \dots, z_{i+1}, z_i, \dots, z_n)$, $1 \leq i\leq n-1$, $s_n z = (z_1, \dots,\overline{z_n})$. The long element of $B_n$,  $w_0=(s_1s_2\cdots s_n)^n$, satisfies $w_0z=-z$ in which case we may consider  $w_0=-Id$ to be the negative $n\times n$ identity matrix.}

Let $G=Sp(2n,\mathbb{C})$, $ GL(n,\mathbb{C})$   be the symplectic group and respectively the general linear group of degree $n$ over $\mathbb{C}$. Let
$\mathfrak{g}=\mathfrak{sp}(2n, \mathbb{C})$, $\mathfrak{gl}(n, \mathbb{C})$ be the corresponding Lie algebras.
The finite dimensional irreducible representations of $G$ are parameterized by partitions $\lambda\in \mathcal{P}_{n}$.
For  any $\lambda\in\mathcal{P}_{n}$, we denote by $V(\lambda)$ the corresponding finite dimensional irreducible representation (or $\mathfrak{g}$-module).
To each partition $\lambda\in\mathcal{P}_{n}$ corresponds a (connected) crystal graph
$\mathfrak{B}(\lambda)$ which can be seen as the combinatorial skeleton of the simple
module $V(\lambda)$. In particular, its vertices label a distinguished basis
of $V(\lambda)$. The crystal graph $\mathfrak{B}(\lambda)$  has various combinatorial realizations, that is, vertex
labeling,  as KN (resp. SSYT) tableaux, see Figure \ref{cristal21only}, \textbf{or} Littelmann's paths, in particular, Lakshmibai-Seshadri paths.

For any $i\in I$, the crystal $\mathfrak{B}(\lambda)$ can be decomposed
into its {$i$-chains (or $i$-strings) which are obtained just by keeping the $i$-arrows.  There is a unique vertex $b_{\lambda}$ in $\mathfrak{B}(\lambda)$
such that ${e}_{i}(b_{\lambda})=0$ for any $i\in I$, that is,
$b_{\lambda}$ is the source vertex of each $i$-chain containing $b_{\lambda}$,
called the \textsf{highest weight vertex} of $\mathfrak{B}(\lambda)$ in which case we have $\mathrm{wt}%
(b_{\lambda})=\lambda$.\
\ For any $b\in \mathfrak{B}(\lambda)$, there is a path
$b={f}_{i_{1}}\cdots{f}_{i_{r}}(b_{\lambda})$ from $b_{\lambda}$
to $b$,  for some $i_1,\dots,i_r\in I$. The weight function $\mathrm{wt}$ satisfies
\begin{align}
\mathrm{wt}(b)=\lambda-\sum_{k=1}^{r}\alpha_{i_{k}}.\label{weight}
\end{align}
There is a unique vertex $b_{w_0\lambda}$ in $\mathfrak{B}(\lambda)$
such that ${f}_{i}(b_{w_0\lambda})=0$ for any $i\in I$, that is,
$b_{w_0\lambda}$ is the sink vertex of each $i$-chain containing $b_{\lambda}$,
called the \textsf{lowest weight vertex} of $\mathfrak{B}(\lambda)$ in which case we have $\mathrm{wt}%
(b_{w_0\lambda})=w_0\lambda$.\ This means that $\mathfrak{B}(\lambda)$ can be generated by applying  all the sequences of lowering (resp. raising) Kashiwara operators $f_i$ (resp. $e_i$) to $b_\lambda$ (resp. $b_{w_0\lambda}$), as long they do not annihilate

For any $i\in I$, the crystal $\mathfrak{B}(\lambda)$ can be decomposed
into its \textsf{$i$-chains (or strings)} which are obtained just by keeping the $i$-arrows. This means that, for each vertex $b$ and each $i$ in $I$ there is only one $i$-string containing $b$.

\begin{align}\label{stringdraw}\begin{tikzpicture}[xscale=1,yscale=1]
\draw[gray, thick] (-2,2) -- (10,2);
\filldraw[black] (-2,2) circle (2pt) ;
\filldraw[black] (2,2) circle (2pt);
\filldraw[black] (1,2) circle (2pt) ;
\filldraw[black] (3,2) circle (2pt) ;
\filldraw[black] (10,2) circle (2pt) ;
\filldraw[black] (6,2) circle (2pt) ;
\draw (1,1.6) node {\scriptsize$e_i(b)$};
\draw (2,1.6) node {\scriptsize$b$};
\draw (3,1.6) node {\scriptsize$f_i(b)$};
\draw (10,1.5) node {\scriptsize$f_i^{\varphi_i(b)}(b)$};
\draw (-2,1.5) node {\scriptsize$\scriptsize e_i^{\scriptsize\varepsilon_i(b)}(\scriptsize b)$};
\draw [decorate, decoration = {brace, amplitude = 5pt}, yshift = 5pt] (-2,2) -- (2,2) node [black, midway, yshift = 10pt] {{\scriptsize
$\varepsilon_i(b)$}};
\draw [decorate, decoration = {brace, amplitude = 5pt}, yshift = 5pt]
(6,2) -- (10,2) node [black, midway, yshift = 10pt] {{\scriptsize
$\varphi_i(b')=\varepsilon_i(b) $}};
\end{tikzpicture}
\end{align}
For $b\in \mathfrak{B}(\lambda)$ and $i\in I$, we may then set $f_i^{\textrm{max}}(b):=f_i^{\varphi_i(b)}(b)$,  { and } $e_i^{\textrm{max}}(b):=e_i^{\varepsilon_i(b)}(b).$

\begin{obs}\label{stringlength} When $b$ is the source of an $i$-string, that is, $e_i(b)=0=\varepsilon_i(b)$, then from \eqref{lengths} $\varphi_i(b)= $ $\langle \textrm{wt}(b),\alpha_i^\vee\rangle$ $= $ $\textrm{wt}_i(b)-\textrm{wt}_{i+1}(b)\ge 0$, $1\le i<n$, and, in type $C_n$, $\varphi_n(b)=\langle \textrm{wt}(b),\alpha_n^\vee\rangle=\textrm{wt}_n(b)\ge 0$, for $i=n$, is the length of the $i$-string, where $\textrm{wt}_i(b)$ indicates the $i$-th component of ${\textrm{wt}}(b)\in \mathbb{Z}^n$.
\end{obs}Henceforth, from \eqref{weights}
\begin{align*}{\textrm{wt}}(f_i^{\varphi_i(b)}(b))&=\textrm{wt}(b)-\varphi_i(b)\alpha_i\\
&=
\begin{cases}{\textrm{wt}}(b)-(\textrm{wt}_i(b)-\textrm{wt}_{i+1}(b))\alpha_i=s_i \textrm{wt}(b)&  1\le i<n,\\
{\textrm{wt}}(b)-\textrm{wt}_n(b)\alpha_n=s_n\textrm{wt}(b)&i=n.
\end{cases}
\end{align*}
In particular $\varphi_i(b_\lambda)=\lambda_i-\lambda_{i+1}$, $1\le i<n$, and $\lambda_n$, for $i=n$, and $f_i^{\lambda_i-\lambda_{i+1}}(b_\lambda)$ with weight $s_i\lambda$ is uniquely determined by its weight. Similarly, when $b$ is the sink of an $i$-string, in which case $\varepsilon_i(b)=-\langle \textrm{wt}(b),\alpha_i^\vee\rangle$ $= $ $-\textrm{wt}_i(b)+\textrm{wt}_{i+1}(b)\ge 0$, $1\le i<n$, and $\varepsilon_n(b)=-\langle \textrm{wt}(b),\alpha_n^\vee\rangle=-\textrm{wt}_n(b)\ge 0$, for $i=n$.

The Weyl group $W$  acts on the vertices of $\mathfrak{B}(\lambda)$ \cite{Kash95}: the action of
the simple reflection $s_{i}$ on $\mathfrak{B}(\lambda)$ sends each vertex $b$ on the
unique vertex $b^{\prime}$ in the $i$-chain of $b$ such that $\varphi
_{i}(b^{\prime})=\varepsilon_{i}(b)$ and $\varepsilon_{i}(b^{\prime}%
)=\varphi_{i}(b)$. Thus  this means that $b$ and $b^{\prime}$ correspond by
the reflection with respect to the center of the $i$-chain containing $b$.
More precisely, from \eqref{lengths} and \eqref{stringdraw},
$$s_i.b=\begin{cases}f_i^{\varphi_i(b)-\varepsilon_i(b)}(b)=f_i^{\langle \textrm{wt}(b),\alpha_i^\vee\rangle}(b)
& \mbox{if } \langle \textrm{wt}(b),\alpha_i^\vee\rangle\ge 0\\
e_i^{\varphi_i(b)-\varepsilon_i(b)}(b)=e_i^{-\langle \textrm{wt}(b),\alpha_i^\vee\rangle}(b)& \mbox{if } \langle \textrm{wt}(b),\alpha_i^\vee\rangle\le 0,
\end{cases}
$$
and $\textrm{wt}(s_i.b)=s_i.\textrm{wt}(b)$, for $i\in I$.


	
\subsection{Tensor product of crystals and signature rule}\label{sec:tensor}
{ If $B$ and $C$ are crystals,
the crystal ${B}\otimes C$ has set of
vertices the cartesian product of the sets of vertices of ${B}$ and
${C}$, denoted $u\otimes v$, $u\in B$ and  $v\in C$,} and crystal structure  given {by } $\mathrm{wt}${$(u\otimes
v)=$}$\mathrm{wt}${$(u)+$}$\mathrm{wt}${$(v)$} and  the following rules where we follow the Kashiwara convention \cite{Kashiwara1994CrystalGF, kash2}%
\begin{equation}
{e}_{i}(u\otimes v)=\left\{
\begin{array}
[c]{l}%
u\otimes {e}_{i}(v)\text{ if }\varepsilon_{i}(v)>\varphi_{i}(u)\\
{e}_{i}(u)\otimes v\text{ if }\varepsilon_{i}(v)\le \varphi_{i}(u)
\end{array}
\right.  \text{ and }{f}_{i}(u\otimes v)=\left\{
\begin{array}
[c]{l}%
{f}_{i}(u)\otimes v\text{ if }\varphi_{i}(u)>\varepsilon_{i}(v)\\
u\otimes{f}_{i}(v)\text{ if }\varphi_{i}(u)\le\varepsilon_{i}(v)
\end{array}
\right.  . \label{tens_crys}%
\end{equation}
We adopt the convention that $u\otimes0=0\otimes v=0$. Given two partitions $\lambda$ and $\mu$ in
$\mathcal{P}_{n}$, the crystal graph of the representation $V(\lambda)\otimes V(\mu)$ is the crystal $\mathfrak{B}(\lambda)\otimes \mathfrak{B}(\mu)$ whose decomposition into connected components as well multiplicity correspond to the decomposition of $V(\lambda)\otimes V(\mu)$ into irreducible representations.

If $B$, $C$  and $D$ are crystals the map $(u\otimes v)\otimes z\mapsto u\otimes(v\otimes z)$ is a crystal isomorphism between $(B\otimes C)\otimes D$ and $B\otimes (C\otimes D)$. We  also have $B\otimes C\simeq C\otimes B$ but the isomorphism is not natural.
Let $b=b_1\otimes\cdots\otimes b_r$. Then $\textrm{wt}(b)=\sum_{k=1}^r\textrm{wt}(b_k)$, and from \eqref{lengths},
\begin{align}
\varphi(b)=\max\{\varphi(b_k)+\sum_{k< u\le r}\left\langle\textrm{wt}(b_u),\alpha^\vee\right\rangle:1\le k\le r\},\label{tensorphi0}\\
\varepsilon(b)=\max\{\varepsilon(b_k)-\sum_{1\le u<k}\left\langle\textrm{wt}(b_u),\alpha^\vee\right\rangle:1\le k\le r\},\nonumber
\end{align}
and
\begin{align}f(b)=b_1\otimes\cdots \otimes(f b_{k_f})\otimes\cdots\otimes b_r,\label{tensorphi1}\\
e(b)=b_1\otimes\cdots \otimes(e b_{k_e})\otimes\cdots\otimes b_r,\nonumber
\end{align}
where $k_f$ ($k_e$) is the biggest (smallest) integer such that
$$\varphi(b)=\varphi(b_k)+\sum_{k< u\le r}\left\langle\textrm{wt}(b_u),\alpha^\vee\right\rangle\quad \left(\varepsilon(b)=\epsilon(b_k)-\sum_{1\le u< k}\left\langle\textrm{wt}(b_u),\alpha^\vee\right\rangle\right).$$

\subsubsection{The signature rule}The tensor product of crystals allows us to define the crystal operators on arbitrary words on the alphabet $[\pm n]$ or KN skew tableaux so that one has a crystal structure in type $C_n$. Let $\Lambda_k=\textbf{e}_1+\cdots+\textbf{e}_k$, $1\le k\le n$, be the fundamental weights in type $C_n$. In type $C_{n}$, the standard crystal is seminormal and has the following  crystal graph:
\begin{align}\label{standardcrystal}
1\xrightarrow{1} 2\xrightarrow{2} \dots \xrightarrow{n-2}n-1\xrightarrow{n-1} n\xrightarrow{n} \overline{n}\xrightarrow{n-1}\overline{n-1}\xrightarrow{n-2}\overline{n-2}\dots \xrightarrow{2}\overline{2} \xrightarrow{1}\overline{1}\end{align}
 with set
	$\mathfrak{B}=[\pm n]$, $ \text{wt}(\boxed{i})=\bf e_i$, $ \text{wt}(\boxed{\overline{i}})=-\bf e_i$ \cite{Kashiwara1994CrystalGF,bump2017crystal}. The highest weight element is the word $1$, and the highest weight $\Lambda_1=\bf e_1$. This is the crystal graph of the simple $\mathfrak{sp}_{2n}$--module, $V(\Lambda_1)$, and we denote it by  $\mathfrak{B}({\Lambda_1})$.
The crystal $\mathfrak{B}{(\Lambda_1)}$ is the crystal on the words of $[\pm n]^\ast$ with a sole letter. The tensor product of crystals allows us to define the crystal
\begin{align}G_n=\bigoplus\limits_{k\ge 0} {\mathfrak{B}{(\Lambda_1)}}^{\otimes k}\label{crystalwords}
\end{align}
 of all words in $[\pm n]^\ast$, where the vertex $w_1\otimes \dots \otimes w_k$ is identified with the word $w_1\dots w_k\in[\pm n]^\ast$ such that the weight $\textrm{wt}(w)=\sum_{i=1}^k\textrm{wt}(w_i)$. The irreducible representation $V(\Lambda_k)$, $1\le k\le n$,  is embedded into $V(\Lambda_1)^{\otimes k}$. Similarly the crystal $\mathfrak{B}(\Lambda_k)$ is embedded into ${\mathfrak{B}{(\Lambda_1)}}^{\otimes k}$ where the admissible column
$\!\YT{0.20in}{}{
			{{u_1}},
			{{\vdots}},
			{{u_k}}}$ of length $k$ is identified with the word $u_1\otimes\cdots\otimes u_k.$ Let $\lambda=\Lambda_{m_1}+\cdots +\Lambda_{m_k}$, $1\le m_1\le\cdots\le m_k\le n$. By the embedding of $V(\lambda)$ into $V(\Lambda_{m_1})\otimes\cdots \otimes V(\Lambda_{m_k})$, $\mathfrak{B}(\lambda)$ is also embedded
into $\mathfrak{B}(\Lambda_{m_1})\otimes\cdots \otimes \mathfrak{B}(\Lambda_{m_k})$ and the KN tableau of shape $\lambda$ and with columns $C_k\cdots C_1$  is identified with the word $C_1\otimes\cdots\otimes C_k$.

 The action of the operators $e_i$ and $f_i$ is easily given by the \emph{signature rule} \cite{Kashiwara1994CrystalGF,lecouvey2002schensted,kash2}. We substitute each letter $w_j$ by $+$ if $w_j\in \{i, \overline{i+1}\}$ or by $-$ if $w_j\in \{i+1, \overline{i}\}$, and erase it in any other case. Then successively erase any pair $+-$ until all the remaining letters form a word that looks like $-^a +^b$. Then $\varphi_i(w)=b$ and $\varepsilon_i(w)=a$, $e_i$ acts on the letter associated to the rightmost unbracketed $-$ (i.e., not erased), whereas $f_i$ acts on the letter $w_j$ associated to the leftmost unbracketed $+$,
\begin{align}\label{signature}f_i(w_j)=\begin{cases}
	i+1\,\,\text{if } w_j=i \text{ and } i\neq n \\
	\overline{i}\,\,\text{if }  w_j=\overline{i+1}\\
	\overline{n}\,\,  \text{if } w_j=i \text{ and } i=n,
	\end{cases}
\end{align} and the other letters of $w$ are unchanged, and $e_i$ is the inverse map. If $b=0$ then $f_i(w)=0$ and if $a=0$ then $e_i(w)=0$.
	
	The set $\mathcal{KN}(\lambda, n)$ (resp. $\mathcal{SSYT}(\lambda, n)$) is endowed with a Kashiwara crystal structure of type $C_n$ (resp. $A_{n-1}$) \cite{kashiwara1992crystal, lecouvey2002schensted}.
	The crystal operators are fully characterized on words in the alphabet $[\pm n]$ (resp. $[n]$) or on KN  skew tableaux (resp. skew SSYT) via the signature rule \eqref{signature}. We identify $\mathfrak{B}(\lambda)$ with $\mathcal{KN}(\lambda, n)$ (resp. $\mathcal{SSYT}(\lambda, n)$)  in type $C_n$ (respectively type $A_{n-1}$).
Observe that $\mathcal{SSYT}(\lambda, n)$ is a subcrystal of $\mathfrak{B}(\lambda)=\mathcal{KN}(\lambda, n)$.  See Figure \ref{cristal21only}.

\begin{figure}[h]
		\begin{center}

			{
				\begin{tikzpicture}
					[scale=.5,auto=left]
					\node (n0) at (0,14.5) {$\YT{0.17 in}{}{{{1},{1}},{{2}}}$};
					\node (n1l) at (-3.5,13.5)  {$\YT{0.17 in}{}{{{1},{2}},{{2}}}$};
					\node (n1r) at (3.5,13.5)  {$\YT{0.17 in}{}{{{1},{1}},{{\overline{2}}}}$};
					\node (n2l) at (-3.5,11)  {$\YT{0.17 in}{}{{{1},{\overline{2}}},{{2}}}$};
					\node (n2r) at (3.5,11) {$\YT{0.17 in}{}{{{1},{2}},{{\overline{2}}}}$};
					\node (n3r) at (3.5,8.5)  {$\YT{0.17 in}{}{{{2},{2}},{{\overline{2}}}}$};
					\node (n4rr) at (5,6)  {$\YT{0.17 in}{}{{{2},{2}},{{\overline{1}}}}$};
					\node (n4r) at (2,6)  {$\YT{0.17 in}{}{{{2},{\overline{2}}},{{\overline{2}}}}$};
					\node (n5r) at (3.5,3.5)  {$\YT{0.17 in}{}{{{2},{\overline{2}}},{{\overline{1}}}}$};
					\node (n6r) at (3.5,1)  {$\YT{0.17 in}{}{{{\overline{2}},{\overline{2}}},{{\overline{1}}}}$};
					\node (n3ll) at (-5,8.5)  {$\YT{0.17 in}{}{{{1},{\overline{2}}},{{\overline{2}}}}$};
					\node (n4l) at (-3.5,6)  {$\YT{0.17 in}{}{{{1},{\overline{1}}},{{\overline{2}}}}$};
					\node (n3l) at (-2,8.5)  {$\YT{0.17 in}{}{{{1},{\overline{1}}},{{2}}}$};
					\node (n5l) at (-3.5,3.5)  {$\YT{0.17 in}{}{{{2},{\overline{1}}},{{\overline{2}}}}$};
					\node (n6l) at (-3.5,1)  {$\YT{0.17 in}{}{{{2},{\overline{1}}},{{\overline{1}}}}$};
					\node (n7) at (0,0)  {$\YT{0.17 in}{}{{{\overline{2}},{\overline{1}}},{{\overline{1}}}}$};
					
					\draw[->] [draw=red,  thick] (-2.5,0.7)--(-1,0.3);
					\draw[->] [draw=blue,  thick] (2.5,0.7)--(1,0.3);
					\draw[->] [draw=blue,  thick] (-3.5,2.5)--(-3.5,2);
					\draw[->] [draw=red,  thick] (3.5,2.5)--(3.5,2);
					\draw[->] [draw=blue,  thick] (-3.5,5)--(-3.5,4.5);
					\draw[->] [draw=blue,  thick] (2.5,5)--(3,4.5);
					\draw[->] [draw=red,  thick] (4.5,5)--(4,4.5);
					\draw[->] [draw=red,  thick] (-2.5,7.5)--(-3,7);
					\draw[->] [draw=blue,  thick] (-4.5,7.5)--(-4,7);
					\draw[<-] [draw=blue,  thick] (4.5,7)--(4,7.5);
					\draw[<-] [draw=red,  thick] (2.5,7)--(3,7.5);
					\draw[->] [draw=blue,  thick] (-3,10)--(-2.5,9.5);
					\draw[->] [draw=red,  thick] (-4,10)--(-4.5,9.5);
					\draw[->] [draw=blue,  thick] (3.5,10)--(3.5,9.5);
					\draw[->] [draw=red,  thick] (-3.5,12.5)--(-3.5,12);
					\draw[->] [draw=blue,  thick] (3.5,12.5)--(3.5,12);
					\draw[->] [draw=blue,  thick] (-1,14.2)--(-2.5,13.8);
					\draw[->] [draw=red,  thick] (1,14.2)--(2.5,13.8);
				\end{tikzpicture}
			}	
\caption{The type $C_2$ crystal graph $\mathcal{KN}((2,1),2)$ containing the $A_1$ crystal $\mathcal{SSYT}((2,1),2)$, consisting of the two top left most tableaux, as a subcrystal. The type $C_2$ lowering crystal operators are $f_1$, $\color{blue}{\rightarrow}$, and $~~~~~~~~f_2$, $~~\color{red}{\rightarrow}$.
\label{cristal21only}}
		\end{center}
\end{figure}

	\subsection{Bruhat order} For an element $w\in W$ as a Coxeter group, the minimum number of simple reflections needed to produce $w$ is $\ell(w)$, the \emph{length} of $w$.
  An expression of the form $s_{i_1} s_{i_2} \cdots s_{i_k}$ representing $ \sigma\in W$ where all the $s_{i_j}$'s are simple reflections and $\ell(\sigma)=k$ is called a {\emph{reduced} decomposition} of $\sigma$.
 The {(\emph{strong}) Bruhat order} $\leq$ on $W$ as a Coxeter group
can be defined by $\sigma^{\prime}\leq\sigma$ in $W$ if and
only if there is a reduced decomposition of $\sigma$ admitting a subexpression (not
necessarily made of consecutive letters) which is a reduced decomposition of
$\sigma^{\prime}$, if and only if every reduced decomposition of $\sigma$
admits a subexpression which is a reduced decomposition of $\sigma^{\prime}$
(see~\cite[Corollary 2.2.3]{bjorner2006combinatorics}).
   We refer the reader to~\cite{bjorner2006combinatorics} for basic statements on  a finite Coxeter group.

 Given any partition $\lambda$ in $\mathcal{P}_{n}$,
we denote by $W_{\lambda}$ its \emph{stabilizer} under the action of
$W$, that is, $W_{\lambda}$ is
parabolic subgroup $W_J$, where $J = \{ 1\le i\le n: \langle \lambda,\alpha^\vee_i\rangle= 0\}$. Each coset in $W/W_{\lambda}$
contains a unique element of minimal length and the set of elements of minimal
length is denoted by $W^{\lambda}$.\ Then each $\sigma
\in W$ admits a unique decomposition of the form $\sigma=\sigma^\lambda\sigma'$
with $\sigma^\lambda\in W^{\lambda}$ and $\sigma'\in W_{\lambda},$  and
$\ell(\sigma)=\ell(\sigma^\lambda)+\ell(\sigma')$. One then has a one-to-one correspondence
between the elements of $W^{\lambda}$ and $W{\lambda}$, the $W%
$-orbit of $\lambda$ \cite{CB,bjorner2006combinatorics}. (Similarly, for the  unique element of maximal  length in each coset in $W/W_{\lambda}$
\cite[Corollary 2.4.5]{bjorner2006combinatorics}.)

\begin{obs} \label{inducedbruhat} The induced (strong) Bruhat order of $W$ on $W\lambda$  coincides with the restriction of the (strong) Bruhat order on $W$ to $W^\lambda$.\begin{itemize}
\item
Consider the set $W \lambda$, which is in bijection with $W^\lambda$ through $w \lambda \mapsto w^\lambda$, where $w^\lambda$ is the representative of minimal length of $w W_\lambda$. Then the transitive closure of the relations \[
	\mu<s_\alpha \mu
	,\;\mbox{if $\langle \mu,\alpha^\vee\rangle>0$,  $s_\alpha$ a reflection in $ W$   and $\mu\in
		W\lambda$, $\alpha\in\Phi^+$}
	\]
yields a partial order on $W \lambda$, which coincides through the aforementioned bijection with the restriction of the (strong) Bruhat order on $W$ to $W^\lambda$. That is, $\mu<s_\alpha \mu$ if and only if $w^\lambda<(s_\alpha w)^\lambda$, where $\mu=w\lambda$  for some $w\in W$. It also amounts to note that, for $\mu<\nu$ in $W\lambda$ in the induced (strong) Bruhat order if and only if $\mu<s_{\alpha_1}\mu<\cdots s_{\alpha_r}\mu$ if and only if $w<(s_{\alpha_1}w)^\lambda<\cdots(s_{\alpha_r}w)^\lambda$ in $W^\lambda$, for some reflections $t_1,\dots,t_r$ in $W$, and where $\mu=w\lambda$, and $\sigma\lambda$ for some $w,\sigma\in W$.
\item  The transpositions are the reflections in the symmetric group, and in $B_n$, seen as a subgroup of $\mathfrak{S}_{2n}$,  the reflections are   $(i\,-i)$, $(i\,j)(-i\,-j)$, $(i\,-j)(-i,j)$, $1\le i<j\le n$.
    \end{itemize}
\end{obs}

 Next we gather some properties of the long element $w_0$ of $W$ useful in the sequel.
 \begin{obs}
Let $S$ be the set of generators of $W$, and  $x, y, w \in W$. Then the long element $w_0$ in $W$ has the following properties:
\begin{itemize}
\item $x \le  y$ if and only
if $w_0y \le  w_0x$ (resp. $yw_0 \le xw_0)$) \cite[Section 3.7]{CB}, and
\item $\ell(w_0w)=\ell(ww_0)=\ell(w_0)-\ell(w)$, for all $w\in W$ \cite[Section 3.7]{CB}.

\item We  may write $w_0w=w_0s_{j_{\ell}}\cdots s_{j_1}=s_{t_q}\cdots s_{t_1}$ for some reduced words $w=s_{j_{\ell}}\cdots s_{j_1}$ and $s_{t_q}\cdots s_{t_1}$ in $W$ where $q=\ell(w_0)-\ell(w)$ \cite[Proposition 3.1.2] {bjorner2006combinatorics}.
\end{itemize}

\end{obs}

\begin{defin}\label{key}A \emph{key tableau} of shape $\lambda$, in type  $C_n$, is a KN tableau in $\mathcal{KN}(\lambda, n)$ in which the set of elements of
	each column 
is contained in the set of elements of the previous column, left to right, if any, and the letters $i$ and $\overline{i}$ do not appear simultaneously as entries, for any $i \in[n]$. A \emph{key tableau} of shape $\lambda$, in type  $A_{n-1}$, or in  $\mathcal{SSYT}(\lambda, n)$, is a key tableau of type $C_n$ where all entries are positive.
\end{defin}

 Given a permutation  $\sigma=[\sigma_1\ldots\sigma_n]\in B_n$ (respect. $\mathfrak{S}_n$) in window (respect. one-line) notation \cite{santos2019symplectic} and a partition $\lambda$ with at
most $n$ parts, $\sigma$ determines a key tableau $K$ of shape $\lambda$ and weight $\sigma\lambda$, denoted $K=K(\sigma\lambda)$,  as follows: the entries in each column
of $K$,  say with  $j$ boxes, are just $\sigma_1,\ldots,\sigma_j$ arranged in increasing order (top to bottom). (Note that in the window notation for $B_n$ there are no symmetric entries.) In particular, either we consider the minimal or maximal representatives of a coset in $W/W_\lambda$  the corresponding key tableaux are the same. The KN tableau $K$ depends only upon the left coset $\sigma W_\lambda$
of the stabiliser subgroup $W_\lambda$ of $W$, and the map $\sigma\mapsto K$ sets up a bijection between such
cosets and key tableaux. (See   \cite{krv23} for type $A_{n-1}$.)

In fact, either in type $A_{n-1}$ or $C_n$, it is possible to compute   the unique minimal element of the coset $\sigma W_\lambda$ \cite[Prosition 6]{santos2019symplectic} from the key  $K(\sigma\lambda)$, a generalization of Lascoux' method for type $A_{n-1}$ keys \cite{lascoux2003double}.
The  unique maximal element in the coset $\sigma W_\lambda$ can also be computed using the same key and  the same rule in \cite[Proposition 6]{santos2019symplectic} but this time by reading along the  key columns, bottom to top, right  to left, to maximize the number inversions \cite[Proposition 1]{santos2019symplectic} among the elements in the coset.  (See  also \cite{krv23}, a recent method in type $A_{n-1}$ to determine the left and right keys of a semistandard tableau where the unique maximal element representatives of the cosets are also considered.) For instance, consider $W=B_5$ and
$$K=\YT{0.17in}{}{
		{{1},{4},{4},{\overline{5}},{\overline{5}}},
		{{4},{\overline{5}},{\overline{5}}},
		{{\overline{5}},{\overline{3}},{\overline{3}}},
		{\overline{3}}}$$ with shape $\lambda=(5,3,3,1,0)$, weight $v=(1,0,\overline{3},3,\overline{5})\in W(5,3,3,1,0)$ and $\sigma=[\overline{5}\,4\,\overline{3}\,1\,2]=\oo\sigma^\lambda$, in  window notation \cite{santos2019symplectic}, is the minimal representative of the coset $\sigma W_\lambda=\sigma <s_2>$, $W_\lambda=<s_2>=\{1,s_2\}$. The maximal element in $\sigma W_\lambda$ is $[\overline{5}\,\overline{3}\,4\, 1\,2]$.

\begin{prop}\label{keybruhat}\cite{santos2019symplectic} Let $W\lambda$ be equipped with the induced Bruhat order  in $W$ restricted to the set  of  minimal coset representatives in $W^\lambda$. Then,
 for $u,v\in W\lambda$,  $u\le v$ if and only if  $K(u)\le K(v)$ by entry-wise comparison (in the corresponding alphabets).
\end{prop}
 For $\lambda=(\lambda_1>\lambda_2>\cdots>\lambda_n)$ a stair partition, $W_\lambda=\{1\}$ and  $W^\lambda=W$ and write just $K(\sigma)$ for $K(\sigma\lambda)$.
 If $\alpha, \beta\in W$ are written in window (one-line) notation then the Bruhat order in $W$ has an alternative definition $\alpha\le \beta\Leftrightarrow K(\alpha)\le K(\beta)$ \cite{santos2019symplectic} and the latter inequality can be read as $K(\alpha)\le K(\beta)\Leftrightarrow \alpha[i]\le \beta[i]$, for $1\le i\le n$,  where we mean  $\alpha[i]=\{a_1<\cdots<a_i\}\le \beta[i]=\{b_1<b_2<\cdots<b_i\}$ to be $a_k\le b_k$, $1\le k\le i$. (See also \cite{fulton1997young, speyer} in type $A_{n-1}$.)}

\subsubsection{The keys poset } The \emph{highest weight element} of $\mathfrak{B}(\lambda)=\mathcal{KN}(\lambda, n)$ (resp. $\mathcal{SSYT}(\lambda, n)$)  is the key tableau $b_\lambda=K(\lambda)$ of shape and weight $\lambda$.
Recall the Weyl group $W$  acts on the vertices of $\mathfrak{B}(\lambda)$ (see \S \ref{sec:crystal})
such that  the simple reflection $s_{i}$ 
sends each vertex $b$ to the unique vertex $s_ib:=b^{\prime}$ in the $i$-\emph{chain} of $b$   where $b^{\prime}$ is given by the reflection with respect to the center of the $i$-chain containing $b$.
In particular,
any key tableau $K\neq K(\lambda)$ of shape $\lambda$   is the vertex at the end of a $i$-chain of positive length belonging to a path of chains of positive lengths connected to $K(\lambda)$. More precisely, it can be obtained from $K(\lambda)$ by a sequence $s_{i_r},\ldots, s_{i_1}$  of simple reflections, $K=K(s_{i_{r}}\cdots s_{i_1}\lambda)=s_{i_r} K(s_{i_{r-1}}\cdots s_{i_1}\lambda)$  corresponding to the sequence  $(i_1,\dots,i_r)$-strings of positive length
connecting $b_\lambda=K(\lambda)$ to $\sigma b_\lambda:=b_{\sigma\lambda}=K$ with $\sigma=s_{i_{r}}\cdots s_{i_1}$,
where $b_{w_0\lambda}=K(w_0\lambda)$, with $w_0$ the longest element of $W$.
We
 write
\[
O(\lambda)=\{\sigma K({\lambda})=K({\sigma\lambda})\mid\sigma
\in W^{\lambda}\}
\]
for the orbit of the highest weight vertex $b_\lambda=K(\lambda)$ of $\mathfrak{B}(\lambda
)$.\ Observe that $K({\sigma\lambda})$ is then the unique vertex in
$\mathfrak{B}(\lambda)$ of weight $\sigma\lambda$. For each element of $W\lambda$ there is
	exactly one key tableau of shape $\lambda$ with that weight.
{The elements of $O(\lambda)$, called
the \textsf{keys} of $\mathfrak{B}(\lambda)$, }are those vertices of $\mathfrak{B}(\lambda)$ which are
completely characterized by their weight. As expected, one has a direct
correspondence between the keys and the cosets of $W/W_{\lambda}$.

If in the crystal $\mathfrak{B}(\lambda)$ one just keeps the vertices in $O(\lambda)$ and the edges of the $i$-chains connecting them, this directed coloured graph defines a coloured poset for the left weak Bruhat order on the minimal representatives in $W^\lambda$.
Recall the left weak Bruhat order in $W$ can defined by the cover relations $\sigma<s_i\sigma$ which holds whenever $\ell(s_i\sigma)=\ell(\sigma)+1$,
 for  $\sigma \in W$ and the simple reflection $s_i\in W$ \cite{hershlenart}. Consider the  weak Bruhat order in $W$ induced on $W\lambda$ as in Remark \ref{inducedbruhat} replacing the reflection $s_\alpha $ by a simple reflection $s_i$ \cite{hershlenart}.
 For $u=(u_1,\ldots,u_n)\in W\lambda$, it means $u<s_iu$ whenever $\langle u,\alpha_i^\vee\rangle> 0$, $i\in I$, that is, $u_i>u_{i+1}$, $1\le i<n$, and $u_n>0$ if $i=n$.
 For brevity, often we just refer weak Bruhat order since right weak Bruhat order will not be used.
  Therefore,  one has the following assertions.
 \begin{prop}\label{keystring}
  For $u\in W\lambda$, $i\in I$ and $s_i\in W$. The following assertions are equivalent
  \begin{enumerate}
 \item $K(u)<K(s_iu)$.
 \item  $u<s_iu$ which happens when
$\langle u,\alpha_i^\vee\rangle> 0$.
\item  There is a crystal $i$-chain from $K(u)$ to $K(s_iu)$ of length $\langle u,\alpha_i^\vee\rangle> 0$.
\item  $f^{\mathrm{max}}_i(K(u))=K(s_iu)$, with $\mathrm{max}=\varphi_i(K(u))=\langle u,\alpha_i^\vee\rangle>0$.
    \end{enumerate}
\end{prop}

By Proposition \ref{keybruhat}, one has $\lambda<u$ in $W\lambda$ if and only if $K(\lambda)<K(u)$. By the crystal string property this happens when there is a weak saturated chain in $W\lambda$  from $\lambda$ to $u$, that is,
\begin{align}\label{saturated} K(\lambda)<K(u)&\Leftrightarrow K(\lambda)<K(s_{i_1}\lambda)<K(s_{i_2}s_{i_1}\lambda)<\cdots <K(s_{i_r}\cdots s_{i_1}\lambda)=K(u)\nonumber\\
&\Leftrightarrow\lambda<s_{i_1}\lambda<s_{i_2}s_{i_1}\lambda<\cdots <s_{i_r}\cdots s_{i_1}\lambda=u\nonumber\\
&\Leftrightarrow \lambda<u.
\end{align}
for some $(i_r,\dots,i_1)$ coloured sequence of strings (of positive lengths) in $\mathfrak{B}(\lambda)$ connecting $K(\lambda)$ to $K(u)$.
Indeed, \eqref{saturated} means that $\lambda<w\lambda$, $w\in W$, if and only if there is  a chain of reduced words  $ 1<s^\lambda_{i_1}=s_{i_1}<$ $(s_{i_2}s_{i_1})^\lambda$ $\cdots<(s_{i_r}\cdots s_{i_1})^\lambda=w^\lambda$ in $W^\lambda$.
This also means that there is a  $(i_r,\dots,i_1)$ coloured sequence of strings (of positive lengths) such that
\begin{align}\label{string}f_{i_r}^{<s_{i_{r-1}}\cdots s_{i_1}\lambda, \alpha_{i_r}^\vee>}\cdots f_{i_2}^{<s_{i_1}\lambda, \alpha_{i_2}^\vee>} f_{i_1}^{<\lambda, \alpha_{i_1}^\vee>}(K(\lambda))=K(u),
\end{align}
or equivalently
$$f_{i_r}^{\textrm{max}}\cdots f_{i_2}^{\textrm{max}} f_{i_1}^{\textrm{max}}(K(\lambda))=K(u) \; \mbox{  with all $\textrm{max}
$ positive}.$$
Thus we may write
\begin{align}\label{O} O(\lambda)
=\{{f^{\textrm{max}}_{j_r}}\cdots {f^{\textrm{max}}_{j_1}}(K(\lambda))\mid j_1,\dots, j_r\in I, ~r\ge 0\},
\end{align}
noting that $f_j^{\textrm{max}}(K(u)):=f_j^{\varphi_j(K(u))}(K(u))=K(s_ju)$ if $\textrm{max}=$ $\varphi_j(K(u))=$ $\langle \mu,\alpha_j^\vee\rangle>0$ and $K(u)$ if $\langle u,\alpha_j^\vee\rangle=0$.

\begin{ex}Let $n=5$, $\lambda=(2,2,2,1,0)$, and \[
K(\lambda)=\YT{0.19in}{}{
	    	{{1},{1}},
			{{2},{2}},
			{3,3},
			{4}}\overset{f_3=f_3^{\mathrm{max}}}\longrightarrow K(s_3\lambda)=\YT{0.19in}{}{
	    	{{1},{1}},
			{{2},{2}},
			{3,4},
			{4}}
\overset{f^2_4=f_4^{\mathrm{max}}}\longrightarrow K(s_4s_3\lambda)=\YT{0.19in}{}{
	    	{{1},{1}},
			{{2},{2}},
			{3,5},
			{5}}.\]
One then has
\[K(\lambda)<K(s_3\lambda)<K(s_4s_3\lambda) \mbox{ by entrywise comparision}\Leftrightarrow\]
\[\Leftrightarrow\lambda=(2,2,2,1,0)<s_3\lambda=(2,2,1,2,0)<s_4s_3\lambda=(2,2,1,0,2) \mbox{ in $W\lambda$},\]
and from Proposition  \ref{keybruhat}
\[1<s_3=s_3^\lambda=[12435]<(s_4s_3)^\lambda=[12534]\neq s_4s_3=[12453] \mbox{ in $W^\lambda$}.\]
\end{ex}

\subsection{Keys, dilatation of crystals and  Lakshmibai-Seshadri paths}\label{subsecdilatation}
Let $m$ be a positive integer and $\lambda$ a partition in $\mathcal{P}_n$. There exists a
unique embedding of crystals $\psi_{m}:\mathfrak{B}(\lambda)\hookrightarrow \mathfrak{B}(m\lambda)$
such that for any vertex $b\in \mathfrak{B}(\lambda)$ and any path $b={f}_{i_{1}%
}\cdots{f}_{i_{l}}(b_{\lambda})$ in $\mathfrak{B}(\lambda)$, we have \cite{Kashsimilar96,kash2}
\[
\psi_{m}(b)={f}_{i_{1}}^{m}\cdots{f}_{i_{l}}^{m}(b_{m\lambda}).
\]
Since the vertex $b_{\lambda}^{\otimes m}$ is of highest weight $m\lambda$ in
$\mathfrak{B}(\lambda)^{\otimes m}$, one gets a particular realization $\mathfrak{B}(b_{\lambda
}^{\otimes m})$ of $\mathfrak{B}(m\lambda)$ in $\mathfrak{B}(\lambda)^{\otimes m}$ with highest
weight vertex $b_{\lambda}^{\otimes m}$. This  gives a canonical
embedding
\begin{equation}
\theta_{m}:\left\{
\begin{array}
[c]{c}%
\mathfrak{B}(b_{\lambda})\hookrightarrow \mathfrak{B}(b_{\lambda}^{\otimes m})\subset \mathfrak{B}(b_{\lambda
})^{\otimes m}\\
b\longmapsto b_{1}\otimes\cdots\otimes b_{m}%
\end{array}
\right.  \label{embdedd}%
\end{equation}
with important properties given in the  theorem below.

\begin{teo}
\label{Th_Dila}(see \cite{Kashsimilar96, kash2, agl})

\begin{enumerate}
\item Let $\sigma\in W^{\lambda}.$ We have $\theta_{m}%
(b_{\sigma\lambda})=b_{\sigma\lambda}^{\otimes m}$.

\item Let $b\in \mathfrak{B}(\lambda)$. When $m$ has sufficiently many factors, in general, the least common multiple of the maximal $i$-string lengths, there
exist elements $\sigma_{1},\ldots,\sigma_{m}$ in $W^{\lambda}$
such that $\theta_{m}(b)=b_{\sigma_{1}\lambda}\otimes\cdots\otimes
b_{\sigma_{m}\lambda}$. Moreover, in this case

\begin{itemize}
\item the elements $b_{\sigma_{1}\lambda}$ and $b_{\sigma_{m}\lambda}$ in
$\theta_{m}(b)$ do not then depend on $m$,

\item up to repetition, the sequence $(\sigma_{1}\lambda,\ldots,\sigma
_{m}\lambda)$ in $\theta_{m}(b)$ does not depend on the realization of the
crystal $\mathfrak{B}(\lambda)$ and we have $\sigma_{1}\geq\sigma_{2}\geq\cdots\geq
\sigma_{m}$.
\end{itemize}
\end{enumerate}
\end{teo}

 We  define the pair of keys, right and left, of an element in
$\mathfrak{B}(\lambda)$ as follows.

\begin{defin}
\label{DefKeys}Let $b\in \mathfrak{B}(\lambda)$. The keys $K^{+}(b)$, the right key, and $K^{-}(b)$, the left key,
of $b$ are defined as follows:
\[
K^{+}(b)=b_{\sigma_{1}\lambda}\text{ and }K^{-}(b)=b_{\sigma_{m}\lambda}.
\]
{In particular, $K^{+}(b_{\sigma\lambda})=K^{-}(b_{\sigma\lambda}%
)=b_{\sigma\lambda}$ for any $\sigma\in W^{\lambda}$. The orbit
$O(\lambda)$ is simultaneously the set of all left and right keys of $\mathfrak{B}%
(\lambda)$.}
\end{defin}
See Figure \ref{dilatation} for an example in type C, and \cite[Example 2.12 ]{agl} for an example in type A.

\begin{obs}\cite[Chapitre 8]{kash2}, \cite[Section 2.3.2]{agl}
$K^-(b)\le  K^+(b)$ for any $b\in \mathfrak{B}(\lambda)$, and $K^-(b)= K^+(b)$ if and only if $b$ is  in ${O}(\lambda)$.
\end{obs}

\begin{obs}\label{LSkeys}
\begin{enumerate}
\item A Lakshmibai-Seshadri path can  be described
as a $W$-path satisfying certain integrality conditions. We refer the reader to \cite{littelmann1994LS,laklitt} for further details. (See also \cite{choi2018lakshmibai}.)
Let $\lambda\in\Lambda^+$ and consider the Bruhat order on $W/W_\lambda$.  Let $\mathbf{\tau} = (\tau_0 > \cdots>\tau_r )$ be a
strictly decreasing sequence of elements of $W/W_\lambda$ and let $\mathbf{a}=(0 < a_1 < \cdots < a_r < 1)$
be strictly increasing sequence of rational numbers.
The pair $\mathbf{\pi} = (\mathbf{\tau}, \mathbf{a})$ is called a rational $W$-path of shape $\lambda$ or a convex
subset of shape $\lambda$ of the orbit $W\lambda$ \cite{laklitt}. If in addition  the pair $\mathbf{\pi}=(\mathbf{\tau}, \mathbf{a})$ satisfies certain
integrality conditions \cite{littelmann1994LS,laklitt} then is called a Lakshmibai-Seshadri (L-S) path of shape $\lambda$.

Let $\mathbf{B}(\lambda)$ be the crystal of Lakshmibai–Seshadri (L-S) paths of shape $\lambda$ \cite{littelmann1994LS}. The crystal of L-S paths of shape $\lambda$ is isomorphic to the Kashiwara crystal $\mathfrak{B}(\lambda)$ \cite{Kashsimilar96}. The dilatation of crystals was the main tool for proving that Littelmann's crystals obtained from its path model coincide with Kashiwara's ones. It is enough to show it for L-S path model that it is what we got from the dilatation \cite[Chapitre 8]{kash2}.
If $\mathbf{\pi}=(\mathbf{\tau},\mathbf{a})$ is an L–S path of shape $\lambda$, the sequence $\mathbf{\tau} = (\tau_0,\dots,\tau_r )$
{  strictly decreasing in the Bruhat order on $W/W_\lambda$}. We call $i(\pi) = \tau_0$,
the initial direction and $e(\pi) = \tau_r$, the final direction of the path which coincide with the right key respectively left key of the corresponding vertex $b$ in $\mathfrak{B}(\lambda)$. That is, $K^+(b)=K(\tau_0\lambda)=b_{\tau_0\lambda}$ and $K^-(b)=K(\tau_r\lambda)=b_{\tau_r\lambda}$.
\item  Indeed Theorem \ref{Th_Dila}, Assertion (2),   gives a procedure  to compute the right and left key maps on a abstract crystal. However this method is not in general effective when $m$ is big. (In general we take the least common multiple of the maximal string lengths.).
Despite  the terminology "left" and "right" keys as in
the original Lascoux's definition \cite{lascoux1990keys}, based on the tableau model realization,
does not fit with the positions of $b_{\sigma_{1}\lambda}$ and $b_{\sigma_{k}\lambda}$ in $\theta_{k}(b)$ we still keep it.
\end{enumerate}
\end{obs}

\begin{figure}[h]
		\begin{multicols}{2}
			
			${
				\begin{tikzpicture}
					[scale=1.35,auto=left]
					\node (n0) at (0,14.5) {$\thiYT{0.17 in}{}{{{1},{1}},{{2}}}^{\otimes 6}$};
					\node (n1l) at (-3.5,13.5)  {$\thiYT{0.17 in}{}{{{1},{2}},{{2}}}^{\otimes 6}$};
					\node (n1r) at (3.5,13.5)  {$\thiYT{0.17 in}{}{{{1},{1}},{{\overline{2}}}}^{\otimes 6}$};
					\node (n2l) at (-3.5,11)  {$\YT{0.17 in}{}{{{1},{\overline{2}}},{\overline{2}}}^{\otimes 3}\otimes\YT{0.17 in}{}{{{1},{{2}}},{{2}}}^{\otimes 3}$};
					\node (n2r) at (3.5,11) {$\YT{0.17 in}{}{{{2},{2}},{{\overline{1}}}}^{\otimes 2}\otimes\YT{0.17 in}{}{{{1},{{1}}},{{\overline 2}}}^{\otimes 4}$};
					\node (n3r) at (3.5,8.5)  {$\YT{0.17 in}{}{{{2},{2}},{{\overline{1}}}}^{\otimes 4}\otimes\YT{0.17 in}{}{{{1},{{1}}},{{\overline 2}}}^{\otimes 2}$};
					\node (n4rr) at (5,6)  {$\thiYT{0.17 in}{}{{{2},{2}},{{\overline{1}}}}^{\otimes 6}$};
					\node (n4r) at (2,6)  {$\YT{0.17 in}{}{{\overline{2},{\overline{2}}},{{\overline{1}}}}^{\otimes 3}\otimes \YT{0.17 in}{}{{{2},{{2}}},{{\overline{1}}}}\otimes \YT{0.17 in}{}{{{1},{{1}}},{{\overline{2}}}}^{\otimes 2}$};
					\node (n5r) at (3.5,3.5)  {$\YT{0.17 in}{}{{\overline{2},{\overline{2}}},{{\overline{1}}}}^{\otimes 3}\otimes \YT{0.17 in}{}{{{2},{{2}}},{{\overline{1}}}}^{\otimes 3}$};
					\node (n6r) at (3.5,1)  {$\thiYT{0.17 in}{}{{{\overline{2}},{\overline{2}}},{{\overline{1}}}}^{\otimes 6}$};
					\node (n3ll) at (-5,8.5)  {$\thiYT{0.17 in}{}{{{1},{\overline{2}}},{{\overline{2}}}}^{\otimes 6}$};
					\node (n4l) at (-3.5,6)  {$\YT{0.17 in}{}{{{2},{\overline{1}}},{{\overline{1}}}}^{\otimes 2}\otimes\YT{0.17 in}{}{{{1},{\overline{2}}},{{\overline{2}}}}^{\otimes 4}$};
					\node (n3l) at (-2,8.5)  {$\YT{0.17 in}{}{{{2},{\overline{1}}},{\overline{1}}}^{\otimes 2}\otimes\YT{0.17 in}{}{{{1},{\overline{2}}},{{\overline{2}}}}\otimes \YT{0.17 in}{}{{{1},{{2}}},{{2}}}^{\otimes 3}$};
					\node (n5l) at (-3.5,3.5)  {$\YT{0.17 in}{}{{{2},{\overline{1}}},{{\overline{1}}}}^{\otimes 4}\otimes \YT{0.17 in}{}{{{1},{\overline{2}}},{{\overline{2}}}}^{\otimes 2}$};
					\node (n6l) at (-3.5,1)  {$\thiYT{0.17 in}{}{{{2},{\overline{1}}},{{\overline{1}}}}^{\otimes 6}$};
					\node (n7) at (0,0)  {$\thiYT{0.17 in}{}{{{\overline{2}},{\overline{1}}},{{\overline{1}}}}^{\otimes 6}$};
					\draw[->] [draw=red,  thick] (-2.5,0.7)--(-1,0.3);
					\draw[->] [draw=blue,  thick] (2.5,0.7)--(1,0.3);
					\draw[->] [draw=blue,  thick] (-3.5,2.5)--(-3.5,2);
					\draw[->] [draw=red,  thick] (3.5,2.5)--(3.5,2);
					\draw[->] [draw=blue,  thick] (-3.5,5)--(-3.5,4.5);
					\draw[->] [draw=blue,  thick] (2.5,5)--(3,4.5);
					\draw[->] [draw=red,  thick] (4.5,5)--(4,4.5);
					\draw[->] [draw=red,  thick] (-2.5,7.5)--(-3,7);
					\draw[->] [draw=blue,  thick] (-4.5,7.5)--(-4,7);
					\draw[<-] [draw=blue,  thick] (4.5,7)--(4,7.5);
					\draw[<-] [draw=red,  thick] (2.5,7)--(3,7.5);
					\draw[->] [draw=blue,  thick] (-3,10)--(-2.5,9.5);
					\draw[->] [draw=red,  thick] (-4,10)--(-4.5,9.5);
					\draw[->] [draw=blue,  thick] (3.5,10)--(3.5,9.5);
					\draw[->] [draw=red,  thick] (-3.5,12.5)--(-3.5,12);
					\draw[->] [draw=blue,  thick] (3.5,12.5)--(3.5,12);
					\draw[->] [draw=blue,  thick] (-1,14.2)--(-2.5,13.8);
					\draw[->] [draw=red,  thick] (1,14.2)--(2.5,13.8);

				\end{tikzpicture}
			}$	
		\end{multicols}
\caption{The dilatation  of the crystal $\textsf{KN}((2,1),2)$ in Figure \ref{cristal21only}, by $m=6$, the least common multiple of the maximal $i$-string lengths, inside $\textsf{KN}((12,6),2)\simeq \mathfrak{B}(K(2,1)^{\otimes 6},2)$, exhibiting the right and left keys of each vertex of $\textsf{KN}((2,1),2)$ as the leftmost  respectively rightmost factor in each $6$-fold tensor product of keys in $O_{B_2}(2,1)$. The blue (resp. red) arrow means $f_1^6$  (resp. $f_2^6$). \label{dilatation}}
\end{figure}

\begin{figure}[h]
		\begin{multicols}{2}
			
			${
				\begin{tikzpicture}
					[scale=1.35,auto=left]
					\node (n0) at (0,14.5) {$(\textbf{id};\textbf{0},\textbf{1})$};
					\node (n1l) at (-3.5,13.5)  {$(\textbf{s}_1;\textbf{0},\textbf{1})$};
					\node (n1r) at (3.5,13.5)  {$(\textbf{s}_2;\textbf{0},\textbf{1})$};
					\node (n2l) at (-3.5,11)  {$(\textbf{s}_2\textbf{s}_1,\textbf{s}_1;\textbf{0},\textbf{$\frac{3}{6}$},\textbf{1})$};
					\node (n2r) at (3.5,11) {$(\textbf{s}_1\textbf{s}_2,\textbf{s}_2;\textbf{0},\textbf{$\frac{2}{6}$},\textbf{1})$};
					\node (n3r) at (3.5,8.5)  {$(\textbf{s}_1\textbf{s}_2,\textbf{s}_2;\textbf{0},\textbf{$\frac{4}{6}$},\textbf{1})$};
					\node (n4rr) at (5,6)  {$(\textbf{s}_1\textbf{s}_2;\textbf{0},\textbf{1})$};
					\node (n4r) at(2,6) {$(\textbf{s}_2\textbf{s}_1\textbf{s}_2,\textbf{s}_1\textbf{s}_2,\textbf{s}_2;\textbf{0},\textbf{$\frac{3}{6}$},\textbf{$\frac{3}{6}+\frac{1}{6}$},\textbf{1})$};
					\node (n5r) at (3.5,3.5)  {$(\textbf{s}_2\textbf{s}_1\textbf{s}_2,\textbf{s}_1\textbf{s}_2;\textbf{0},\textbf{$\frac{3}{6}$},\textbf{1})$};
					\node (n6r) at (3.5,1)  {$(\textbf{s}_2\textbf{s}_1\textbf{s}_2,;\textbf{0},\textbf{1})$};
					\node (n3ll) at (-5,8.5)  {$(\textbf{s}_2\textbf{s}_1;\textbf{0},\textbf{1})$};
					\node (n4l) at (-3.5,6)  {$(\textbf{s}_1\textbf{s}_2\textbf{s}_1,\textbf{s}_2\textbf{s}_1;\textbf{0},\textbf{$\frac{2}{6}$},\textbf{1})$};
					\node (n3l) at (-2,8.5)  {$(\textbf{s}_1\textbf{s}_2\textbf{s}_1,\textbf{s}_2\textbf{s}_1,\textbf{s}_1;\textbf{0},\textbf{$\frac{2}{6}$},\textbf{$\frac{2}{6}+\frac{1}{6}$},
\textbf{1})$};
					\node (n5l) at (-3.5,3.5)  {$(\textbf{s}_1\textbf{s}_2\textbf{s}_1,\textbf{s}_2\textbf{s}_1;\textbf{0},\textbf{$\frac{4}{6}$},\textbf{1})$};
					\node (n6l) at (-3.5,1)  {$(\textbf{s}_1\textbf{s}_2\textbf{s}_1;\textbf{0},\textbf{1})$};
					\node (n7) at (0,0)  {$(\textbf{w}_0=\textbf{s}_2\textbf{s}_1\textbf{s}_2\textbf{s}_1;\textbf{0},\textbf{1})$};
					\draw[->] [draw=red,  thick] (-2.5,0.7)--(-1,0.3);
					\draw[->] [draw=blue,  thick] (2.5,0.7)--(1,0.3);
					\draw[->] [draw=blue,  thick] (-3.5,2.5)--(-3.5,2);
					\draw[->] [draw=red,  thick] (3.5,2.5)--(3.5,2);
					\draw[->] [draw=blue,  thick] (-3.5,5)--(-3.5,4.5);
					\draw[->] [draw=blue,  thick] (2.5,5)--(3,4.5);
					\draw[->] [draw=red,  thick] (4.5,5)--(4,4.5);
					\draw[->] [draw=red,  thick] (-2.5,7.5)--(-3,7);
					\draw[->] [draw=blue,  thick] (-4.5,7.5)--(-4,7);
					\draw[<-] [draw=blue,  thick] (4.5,7)--(4,7.5);
					\draw[<-] [draw=red,  thick] (2.5,7)--(3,7.5);
					\draw[->] [draw=blue,  thick] (-3,10)--(-2.5,9.5);
					\draw[->] [draw=red,  thick] (-4,10)--(-4.5,9.5);
					\draw[->] [draw=blue,  thick] (3.5,10)--(3.5,9.5);
					\draw[->] [draw=red,  thick] (-3.5,12.5)--(-3.5,12);
					\draw[->] [draw=blue,  thick] (3.5,12.5)--(3.5,12);
					\draw[->] [draw=blue,  thick] (-1,14.2)--(-2.5,13.8);
					\draw[->] [draw=red,  thick] (1,14.2)--(2.5,13.8);

				\end{tikzpicture}
			}$	
		\end{multicols}
\caption{The   crystal $\mathbf{B}(2,1)$ of L-S paths of shape $\lambda=\Lambda_2+\Lambda_1=(2,1)$ obtained from the dilatation of the $C_2$ crystal $\mathfrak{B}(2,1)$. The $C_2$ Weyl group $B_2=<s_1,s_2|s_1^2=s_2^2=1,\,(s_1s_2)^4=1>$ with long element $\textbf{w}_0=s_2s_1s_2s_1$. \label{laks}}
\end{figure}

In the type A tableau crystal model there are currently many ways to compute the right  and left key maps besides the original ones based on frank words and JDT \cite{lascoux1990keys}. By direct inspection of a Young tableau, Willis \cite{Willis2011ADW} has  given an alternative algorithm to compute the right key tableau that does not require the use of \emph{jeu de taquin}.  Other methods to compute the type A  right key map includes      \emph{semi skyline augmented fillings} by Mason \cite{mason2008decomposition}, and the alcove path model by Lenart \cite{lenart2015generalization}, for instance. For new methods and a complete overview  in type A, see \cite{bvbg2021vertex} and the references therein. See also for a recent  new method  given in \cite{krv23} based on Deodhar lifts known  in standard monomial theory developed by Lakshmibai, Musili, and Seshadri \cite{lamuse} and references therein.

In \cite{santos2019symplectic} the second author Santos has defined type C frank words on the alphabet $[\pm n]$ and used them to create the right and left key maps, that send KN tableaux to key tableaux in type C to be recalled in Section \ref{rightkeyjdt}. See also \cite{jacon2019keys}. In addition, in \cite{santos2021fpsac}, motivated by the Willis' direct way for semistandard tableaux, a direct way for computing the right key of a KN tableau was also provided. This direct procedure is here also shown for left keys \cite{santos2022}. Alternatively, in \S \ref{sec:virtual}, using the Baker imbedding, we virtualize the symplectic right and left keys in $A_{2n-1}$. Simultaneously, Baker embedding also virtualizes the L-S path crystal
$\mathbf{B}(\lambda)$.

\subsection{Sch\"utzenberger-Lusztig involution and dual crystal}Crystals corresponding to finite-dimensional (quantum group) $U_q(\mathfrak{g})$-representations belong to a family of crystals called \textit{normal crystals} \cite{bump2017crystal}. In classical types, these crystals may be  realized  by a tableau model  \cite{Kashiwara1994CrystalGF} and have nice combinatorial properties. Normal crystals arise as the crystals associated to the finite-dimensional representations of a quantum group $U_q(\mathfrak{g})$ for some Lie algebra $\mathfrak{g}$ \cite{bump2017crystal}. Let  $I$ be the Dynkin diagram associated to the root system of $\mathfrak{g}$.
	Let us recall the definition of Lusztig involution, also called Sch\"utzenberger involution in type $A_{n-1}$. Consider the Dynkin diagram automorphism, a permutation of its nodes which leaves the diagram invariant,  $\theta:I\rightarrow I$   defined by $\alpha_{\theta(i)} = -w_0\alpha_i$, $\alpha_i$ is the $i$-th simple root at node $i\in I$, where $w_0$ is the longest element of the Weyl group $W$. For type $A_{n-1}$ we have that $w_0$ is the reverse permutation and $\theta(i)=n-i$, and for type $C_n$ we have  $w_0=-\text{Id}$ and $\theta(i)=i$, where $\text{Id}$ is the identity map.
	\begin{defin} Let $\mathfrak{B}$ be a normal crystal.
		The Lusztig involution $\xi:\mathfrak{B}\rightarrow\mathfrak{B}$ is the only set  involution such that for all $i\in I$ ($I=[n-1]$ in type $A_{n-1}$ and $I=[n]$ in type $C_n$):
		\begin{enumerate}
			\item $\text{wt}(\xi(x))=w_0(\text{wt}(x))$,  where $w_0$ is the longest element of the Weyl group;
			\item $e_i(\xi (x))=\xi(f_{\theta(i)}(x))$ and $f_i(\xi (x))=\xi(e_{\theta(i)}(x))$;
			\item $\varepsilon_i(\xi (x))=\varphi_{\theta(i)}(x)$ and $\varphi_i(\xi (x))=\varepsilon_{\theta(i)}(x)$.\end{enumerate}
\end{defin}
	The involution  map $\mathrm{\iota}$ is an involution on $\mathfrak{B}(\lambda)$ reversing
the arrows while flipping the labels $i$, and $\theta(i)$, $i\in I$, and applying $w_0$ to the weight of each vertex. (In type $C_n$  the labels are preserved and the weights change of sign, and in type $A_{n-1}$,  the edge labels are flipped by $i\mapsto n-i$, and and vertex weights are reversed.)
If $b\in \mathfrak{B}(\lambda)$ then $b={f}_{i_{\ell}}^{k_{\ell}}\cdots{f}_{i_{1}%
}^{k_{1}}(b_\lambda)$  for some $i_1,\dots,i_{\ell}\in I$ and $k_{i_1},\dots,k_{i_{\ell}}\ge 0$, and   $\xi(b)={e}_{\theta(i_\ell)}^{k_{\ell}}\cdots{e}_{\theta(i_1)%
}^{k_{1}}(b_{w_0\lambda})$.

Let $\mathfrak{C}$ be a  connected component  in the  crystal $G_n$ \eqref{crystalwords}. The dual crystal $\mathfrak{C}^\vee$  is the crystal obtained from $\mathfrak{C}$ after reversing the direction of all arrows, and  if $x\in\mathfrak{C}$, then for the corresponding $x^\vee$ in $\mathfrak{C}^\vee$, we have $-\text{wt}(x)=\text{wt}(x^\vee)$ \cite{bump2017crystal}. For example, in type $A_{n-1}$,
\begin{align}\label{dualstandardcrystal}
\mathfrak{B}(\Lambda_1)&: \;1\xrightarrow{1} 2\xrightarrow{2} \cdots \xrightarrow{n-2}n-1\xrightarrow{n-1} n, \nonumber\\
\mathfrak{B}^\vee(\Lambda_1)= \mathfrak{B}(-w_0\Lambda_1)&:\;-1\xleftarrow{1}- 2\xleftarrow{2} \cdots \xleftarrow{n-2}-n+1\xleftarrow{n-1} -n.\end{align}
and in type $C_n$, since $w_0=-\text{Id}$, $\mathfrak{B}^\vee(\Lambda_1)= \mathfrak{B}(\Lambda_1)$, and from \eqref{standardcrystal}, we get
\begin{align}\nonumber
\overline 1\xleftarrow{1}\overline 2\xleftarrow{2} \cdots \xleftarrow{n-2}\overline {n-1}\xleftarrow{n-1} \overline n\xleftarrow{n}n \xleftarrow{n-1}{n-1}\xleftarrow{n-2}{n-2}\xleftarrow{n-3} \; \cdots \;\xleftarrow{2}{2} \xleftarrow{1}{1}.\end{align}

In type $C_n$,   it then follows from the definition that  $\mathfrak{C}$ and $\mathfrak{C}^\vee$, as crystals in $G_n$, have the same highest weight, and, therefore, they are isomorphic.  Since $\theta_i(i)=i$, the Lusztig involution  is  a realization of the dual crystal and the crystal $\mathfrak{B}(\lambda)$  is self-dual, that is $\mathfrak{B}^\vee(\lambda)=\mathfrak{B}(\lambda)$.
In type $A_{n-1}$, we may identify $\mathfrak{B}^\vee(\lambda)=\mathfrak{B}(-w_0\lambda)$ a $\mathfrak{gl}_n$-crystal with highest weight $-w_0\lambda$ realized by the Stembridge rational tableaux \cite{stembridge}
. See \cite{kwon} for details.
\begin{obs} Let $G =Sp(2n,\mathbb{C})$ be the symplectic group or $G=GL(n,\mathbb{C})$ the general linear group. Let $V(\lambda)$ be the irreducible $G$-module with highest weight $\lambda$ and let $\mathfrak{B}(\lambda)$ be the associated kashiwara crystal.
Let $V^*(\lambda)$ be the corresponding dual $G$-module of $V(\lambda)$. Then $V^*(\lambda)\simeq V(-w_0\lambda)$ with dual crystal $\mathfrak{B}^\vee(\lambda)=\mathfrak{B}(-w_0\lambda)$. When $G =Sp(2n,\mathbb{C})$, then $V^*(\lambda)\simeq V(\lambda)$ with crystal $\mathfrak{B}^\vee(\lambda)=\mathfrak{B}(\lambda)$.
\end{obs}
	
	If $\mathfrak{B}(\lambda)$ is the type $C_n$ (respect. $A_{n-1}$) crystal then the Sch\"utzenberger-Lusztig involution $\xi(T)$, for $T \in\mathfrak{B}(\lambda)$ with  tableau realization  can be computed via the   Schützenberger evacuation, $\textsf{evac}$ which consists in $\pi$-rotating $T$, swapping all of its entries $i$ by $w_0(i)$ for all $i$, and finally rectifying it via symplectic jeu de taquin, obtaining $\xi(T)=\textsf{evac}(T)$ \cite[Section 5.1,Algorithm 59]{santos2019symplectic}.

\begin{obs}The right key of a tableau is the  evacuation  of the left key of the evacuation of the same tableau   $K^+(b)=\textsf{evac}K^-(\textsf{evac}\,b )$ \cite[Proposition 64]{santos2019symplectic}.
\end{obs}

	\subsection{Demazure  crystal, its opposite, their intersection and Schubert varieties}
We scrutinize the structure of Demazure crystals,  opposite Demazure crystals and their intersections and analyse the parallels with Schubert varieties via the Borel-Weil theorem.
\subsubsection{Demazure crystal}
Given a subset $X$ of $\mathfrak{B}(\lambda)$, consider the operator $\mathfrak{D}_i$ on $X$, with $i\in [n]$ defined by
	$\mathfrak{D}_iX=\{x\in \mathfrak{B}(\lambda)\mid e_i^k(x)\in X \,\,\text{for some $k\geq 0$}\}$ \cite{bump2017crystal}. That is, $\mathfrak{D}_iX=\{f_i^k(x):x\in X, k\ge 0\}\setminus\{0\}$  consist of the union of all sections of $i$-strings from $x\in X$ to $f_i^{\textrm{max}}(x)$.
	If $v= \sigma \lambda$ where $\sigma =s_{i_{\ell}}\cdots s_{i_1}\in W$ is a reduced word, we define the \emph{Demazure crystal} $\mathfrak{B}_v$ (also denoted $ \mathfrak{B}_\sigma(\lambda)$) to be
	\begin{align}
		\mathfrak{B}_v&=\mathfrak{D}_{i_{\ell}}\cdots\mathfrak{D}_{i_1}\{K(\lambda)\}\nonumber\\
&=\{f_{i_{\ell}}^{k_{\ell}}\cdots f_{i_1}^{k_{1}}(K(\lambda))\mid(k_{\ell},\ldots,k_{1})\in\mathbb{Z}_{\geq
0}^{\ell}\}\setminus\{0\} \label{demaz0}.
	\end{align} 
	Indeed if $\sigma=e$, $\mathfrak{B}_e(\lambda)=\mathfrak{B}_{\lambda}=\{K(\lambda)\}$, and if $\sigma=w_0$, $ \mathfrak{B}_{w_0}(\lambda)=\mathfrak{B}(\lambda)$.
	This definition is independent of the reduced word for $\sigma$ \cite[Theorem 13.5]{bump2017crystal}. It is also independent of the coset representative of $\sigma W_\lambda$, that is, $\mathfrak{B}_{\sigma\lambda  }=\mathfrak{B}_{\sigma_v\lambda  }$ and $\mathfrak{B}_{\sigma'\lambda  }=\mathfrak{B}_{\lambda  }=\{K(\lambda)\}$ for $\sigma^\lambda\in W^\lambda$ the minimal representative of that coset, respectively $\sigma'\in W_\lambda$ \cite{santos2019symplectic}.

If $\rho\leq\sigma$, for the Bruhat order of $W$, then $u=\rho\lambda\leq v$ in $W\lambda$, equivalently $\rho^\lambda\le \sigma^\lambda$ in $W^\lambda$ \cite[Proposition 2.5.1]{bjorner2006combinatorics}. Since $ f_i^0(x)=x$, if $\rho\leq\sigma$  then  $\mathfrak{B}_{u}\subseteq\mathfrak{B}_{v}$.
	Thus we define the\emph{ Demazure crystal atom} $\overline{\mathfrak{B}}_{ v}$ to be
	\begin{equation*}\overline{\mathfrak{B}}_{ v}=\mathfrak{B}_{v }\setminus\bigsqcup\limits_{ u\in W\lambda,\;u<v}\mathfrak{B}_{u}=\mathfrak{B}_{v }\setminus\bigsqcup\limits_{K(u)<K(v) }\mathfrak{B}_{u}.\end{equation*}

\begin{figure}[h]
		\begin{multicols}{2}
			
			${
				\begin{tikzpicture}
					[scale=.5,auto=left]
					\node (n0) at (0,14.5) {$\thiYT{0.17 in}{}{{{1},{1}},{{2}}}$};
					\node (n1l) at (-3.5,13.5)  {$\thiYT{0.17 in}{}{{{1},{2}},{{2}}}$};
					\node (n1r) at (3.5,13.5)  {$\thiYT{0.17 in}{}{{{1},{1}},{{\overline{2}}}}$};
					\node (n2l) at (-3.5,11)  {$\YT{0.17 in}{}{{{1},{\overline{2}}},{{2}}}$};
					\node (n2r) at (3.5,11) {$\YT{0.17 in}{}{{{1},{2}},{{\overline{2}}}}$};
					\node (n3r) at (3.5,8.5)  {$\YT{0.17 in}{}{{{2},{2}},{{\overline{2}}}}$};
					\node (n4rr) at (5,6)  {$\thiYT{0.17 in}{}{{{2},{2}},{{\overline{1}}}}$};
					\node (n4r) at (2,6)  {$\YT{0.17 in}{}{{{2},{\overline{2}}},{{\overline{2}}}}$};
					\node (n5r) at (3.5,3.5)  {$\YT{0.17 in}{}{{{2},{\overline{2}}},{{\overline{1}}}}$};
					\node (n6r) at (3.5,1)  {$\thiYT{0.17 in}{}{{{\overline{2}},{\overline{2}}},{{\overline{1}}}}$};
					\node (n3ll) at (-5,8.5)  {$\thiYT{0.17 in}{}{{{1},{\overline{2}}},{{\overline{2}}}}$};
					\node (n4l) at (-3.5,6)  {$\YT{0.17 in}{}{{{1},{\overline{1}}},{{\overline{2}}}}$};
					\node (n3l) at (-2,8.5)  {$\YT{0.17 in}{}{{{1},{\overline{1}}},{{2}}}$};
					\node (n5l) at (-3.5,3.5)  {$\YT{0.17 in}{}{{{2},{\overline{1}}},{{\overline{2}}}}$};
					\node (n6l) at (-3.5,1)  {$\thiYT{0.17 in}{}{{{2},{\overline{1}}},{{\overline{1}}}}$};
					\node (n7) at (0,0)  {$\thiYT{0.17 in}{}{{{\overline{2}},{\overline{1}}},{{\overline{1}}}}$};
					\draw (0,12.5)--(-3.5,15.5);
					\draw (0,12.5)--(3.5,15.5);
					\draw (0,12.5)--(5,3.5);
					\draw (0,12.5)--(-5,7);
					\draw (0,12.5)--(-2,-.5);
					\draw (0,12.5)--(2,-.5);
					\draw (0,12.5)--(-5,12.5);
					\draw (0,12.5)--(5,12.5);
					\draw[->] [draw=red,  thick] (-2.5,0.7)--(-1,0.3);
					\draw[->] [draw=blue,  thick] (2.5,0.7)--(1,0.3);
					\draw[->] [draw=blue,  thick] (-3.5,2.5)--(-3.5,2);
					\draw[->] [draw=red,  thick] (3.5,2.5)--(3.5,2);
					\draw[->] [draw=blue,  thick] (-3.5,5)--(-3.5,4.5);
					\draw[->] [draw=blue,  thick] (2.5,5)--(3,4.5);
					\draw[->] [draw=red,  thick] (4.5,5)--(4,4.5);
					\draw[->] [draw=red,  thick] (-2.5,7.5)--(-3,7);
					\draw[->] [draw=blue,  thick] (-4.5,7.5)--(-4,7);
					\draw[<-] [draw=blue,  thick] (4.5,7)--(4,7.5);
					\draw[<-] [draw=red,  thick] (2.5,7)--(3,7.5);
					\draw[->] [draw=blue,  thick] (-3,10)--(-2.5,9.5);
					\draw[->] [draw=red,  thick] (-4,10)--(-4.5,9.5);
					\draw[->] [draw=blue,  thick] (3.5,10)--(3.5,9.5);
					\draw[->] [draw=red,  thick] (-3.5,12.5)--(-3.5,12);
					\draw[->] [draw=blue,  thick] (3.5,12.5)--(3.5,12);
					\draw[->] [draw=blue,  thick] (-1,14.2)--(-2.5,13.8);
					\draw[->] [draw=red,  thick] (1,14.2)--(2.5,13.8);
				\end{tikzpicture}
			}$	

			\noindent	The crystal is split into $\left | B_2(2,1)\right|=8$ parts, the number of elements of the $B_2$-orbit of $(2,1)$. Each part is a Demazure crystal atom and contains exactly one symplectic key tableau in $O(\lambda)$, drawn with  thick lines, so we can identify each part with the weight of that key tableau, which is a vector in the $B_2$-orbit of $(2,1)$.
			
			%
			%
		\end{multicols}
\caption{The partition of the $C_2$ crystal graph $\mathfrak{B}{(2,1)}$ into $\left | B_2(2,1)\right|=8$ Demazure atom crystals \label{cristal21}.}
\end{figure}

	Every Demazure crystal atom contains exactly one key tableau. The right key map $ K^+$, Theorem \ref{Th_Dila}, Assertiom $(2)$, or \cite{lascoux1990keys} and \cite[Theorem 14, Theorem 17]{santos2019symplectic} in types A and C respectively, sends each tableau of $\mathfrak{B}(\lambda)$ to the unique key tableau  living in the Demazure crystal atom that contains the given tableau. See  Figures  \ref{dilatation} and  \ref{cristal21}. Thereby,
\begin{align}
\overline{\mathfrak{B}}_v&=\{b\in \mathfrak{B}(\lambda): K^+(b)=K(v)\},\label{Boverline}\\
{\mathfrak{B}}_{ v}&=\bigsqcup_{v'\in W\lambda,\,v'\le v} \overline{\mathfrak{B}}_{v' }=\{b\in \mathfrak{B}(\lambda): K^+(b)\le K(v)\} \label{Bv}.
\end{align}

\begin{obs}For Cartan type $C_n$,
since $K(-v)=\textsf{evac}~ K(v)$,  and $K^-(b)=\textsf{evac}K^+(\textsf{evac}\,b )$,  it follows from \eqref{Bv},
 \begin{align*}{\mathfrak{B}}_{ -v}=\bigsqcup_{v'\in W\lambda,\,v'\ge v} \overline{\mathfrak{B}}_{-v' }&=\{b\in \mathfrak{B}(\lambda): \textsf{evac}K^+( b)\ge K(v)\}\\
 &=\{b\in \mathfrak{B}(\lambda): K^-(\textsf{evac}\, b)\ge K(v)\}.
\end{align*}
\end{obs}

\begin{ex}For Cartan type $C_2$, $\lambda=(2,1)$, $v=(1,-2)$, one has
\begin{enumerate}
\item $\mathfrak{B}_{v}=\mathfrak{B}_{s_2s_1(2,1)}=\overline{\mathfrak{B}}_{(2,1)}\sqcup \overline{\mathfrak{B}}_{(1,2)}\sqcup \overline{\mathfrak{B}}_{(2,-1)}\sqcup \overline{\mathfrak{B}}_{(1,-2)}$, since $1,s_1,s_2<s_2s_1$ $\Rightarrow \lambda,s_1\lambda, s_2\lambda<s_2s_1\lambda$, and
    \item $\mathfrak{B}_{-v}=\mathfrak{B}_{(-1,2)}=\mathfrak{B}_{w_0(1,-2)}=\mathfrak{B}_{s_1s_2(2,1)}$ where $w_0 s_2s_1=s_1s_2$ in  $W=B_2$. Then
\begin{align*}{\mathfrak{B}}_{ -v}&=\{f_{1}^{k_{1}} f_{2}^{k_{2}}(K(2,1))\mid(k_{2},k_{1})\in\mathbb{Z}_{\geq
0}^{2}\}\setminus\{0\}
\\
&=\overline{\mathfrak{B}}_{(2,1)}\sqcup \overline{\mathfrak{B}}_{(2,-1)}\sqcup\overline{\mathfrak{B}}_{(1,2)}\sqcup\overline{\mathfrak{B}}_{(-1,2)}\\
&=\overline{\mathfrak{B}}_{-(-2,-1)}\sqcup \overline{\mathfrak{B}}_{-(-2,1)}\sqcup\overline{\mathfrak{B}}_{-(-1,-2)}\sqcup\overline{\mathfrak{B}}_{-(1,-2)}.
\end{align*}
\end{enumerate}
\end{ex}

	\subsubsection{Opposite Demazure crystal} 
	Analogously to the previous case, we start by creating an opposite operator
	$\mathfrak{D}^{op}_i$ on $X$, with $i\in [n]$ defined by
	$\mathfrak{D}^{op}_iX=\{x\in \mathfrak{B}(\lambda)\mid f_{\theta(i)}^k(x)\in X \,\,\text{for some $k\geq 0$}\}$. That is,
$\mathfrak{D}^{op}_iX=\{ e_{\theta(i)}^k(x)\mid \, x\in X , \,k\geq 0 \}\setminus\{0\}$.
	If $v= \sigma \lambda$ and $\sigma =s_{i_{\ell}}\cdots s_{i_1}\in W$ is a reduced word, we define the \emph{opposite Demazure crystal} $\mathfrak{B}^{w_0v}=\mathfrak{B}^{w_0\sigma}(\lambda)$ to be
	\begin{align*}
		\mathfrak{B}^{w_0v}&:=\mathfrak{D}^{op}_{i_{\ell}}\cdots\mathfrak{D}^{op}_{i_1}\{K( w_0\lambda)\}\nonumber\\
&={\{{e}_{\theta(i_{\ell})}^{k_{\ell}}\cdots{e}_{\theta(i_{1})%
}^{k_{1}}(K(w_0\lambda))\mid(k_{\ell},\ldots,k_{1})\in
\mathbb{Z}_{\geq0}^{\ell}\}\setminus\{0\}}.
	\end{align*} 

	In other words, the opposite Demazure crystal $\mathfrak{B}^{w_0v}$ is the image of $\mathfrak{B}_v$ \eqref{Bv} by the Sch\"utzenberger-Lusztig $\xi$ involution. Recall,  if $b\in \mathfrak{B}_{v}$ then $b={f}_{i_{\ell}}^{k_{\ell}}\cdots{f}_{i_{1}%
}^{k_{1}}(K(\lambda))$ and
\begin{align*}\xi(b)&={e}_{\theta(i_{\ell})}^{k_{\ell}}\cdots{e}_{\theta(i_{1})%
}^{k_{1}}(K(w_0\lambda))=\textsf{evac}(b).\end{align*}
 Therefore
\begin{align}
		\mathfrak{B}^{w_0v}
&={\{{e}_{\theta(i_{\ell})}^{k_{\ell}}\cdots{e}_{\theta(i_{1})%
}^{k_{1}}(K( w_0\lambda))\mid(k_{\ell},\ldots,k_{1})\in
\mathbb{Z}_{\geq0}^{\ell}\}\setminus\{0\}}\nonumber\\
&=\xi(\mathfrak{B}_{v})\label{oppxi}\\
&=\xi\{b\in \mathfrak{B}(\lambda): K^+(b)\le K(v)\}\nonumber\\
&=\{b\in \mathfrak{B}(\lambda): K^+(\xi b)\le K(v)\}\nonumber\\
&=\{b\in \mathfrak{B}(\lambda):\xi K^-( b)\le K(v)\}\nonumber\\
&=\{b\in \mathfrak{B}(\lambda): K^-( b)\ge \xi K(v)\}\nonumber\\
&=\{b\in \mathfrak{B}(\lambda)\mid K^ -(b)\ge K(w_0v)\}.\nonumber
\end{align}
\mbox{and  }
\begin{align}\mathfrak{B}^{v}&=\xi(\mathfrak{B}_{w_0v})=\{b\in \mathfrak{B}(\lambda): K^-( b)\ge K(v)\}.\label{oppB}
\end{align}
	In particular, since $\mathfrak{B}^{\sigma}(\lambda)=\xi\mathfrak{B}_{w_0\sigma}(\lambda)$ then $\mathfrak{B}^{e}(\lambda)=\xi(\mathfrak{B}_{w_0}(\lambda))$.

\begin{obs} \begin{enumerate}
\item In types $A$ and $C$ the Sch\"utzenberger evacuation algorithm  is a realization of the Sch\"utzenberger-Lusztig involution on $\mathcal{SSYT}(\lambda,n)$ respectively $\mathcal{KN}(\lambda, n)$ \cite[Algorithm 59]{santos2019symplectic}. In type $C_n$ the tableau weights in $\mathfrak{B}_v$ and in $ \mathfrak{B}^{-v}$ are symmetric.
    \item The Lusztig-Sch\"utzenberger involution can be considered, apart the sign of the weights, to take $\mathfrak{B}(\lambda)$ to its dual $\mathfrak{B}(\lambda)^\vee$ as it sends $f_i$ to $ e_{\theta(i)}$, and, in particular, we may view the opposite Demazure crystal $\mathfrak{B}^{w_0\sigma}(\lambda)$ as the dual  of  the Demazure crystal  $\mathfrak{B}_{\sigma}(\lambda)$.
In general, we may view $\mathfrak{B}(\lambda)^\vee=\mathfrak{B}_{w_0}(\lambda)^\vee$ as $\mathfrak{B}(-w_0\lambda)$.
\end{enumerate}
\end{obs}

From \eqref{Bv}, \eqref{oppxi} and because  $\xi$ is an involution, we define the \emph{opposite Demazure  crystal atom} $\overline{\mathfrak{B}}^{w_0 v}$ to be
\begin{align}\label{opBoverline}
\mathfrak{B}^{w_0v}=\bigsqcup_{v'\in W\lambda,\,v'\le v} \xi(\overline{\mathfrak{B}}_{v' })=\bigsqcup_{v'\in W\lambda,\,v'\le v} \overline{\mathfrak{B}}^{w_0v' }=\bigsqcup_{v'\in W\lambda,\,v'\ge w_0v} \overline{\mathfrak{B}}^{v' }.
\end{align}
From \eqref{opBoverline}, \eqref{oppB} and \eqref{Boverline}, one has
\begin{align}
\overline{\mathfrak{B}}^{w_0v}&=\{b\in \mathfrak{B}(\lambda): K^-(b)=K(w_0v)\}=\xi(\overline{\mathfrak{B}}_v)\nonumber\\
\Leftrightarrow
\overline{\mathfrak{B}}^{v}&=\{b\in \mathfrak{B}(\lambda): K^-(b)=K(v)\}=\xi(\overline{\mathfrak{B}}_{w_0v}).\label{opBoverlinatom}
\end{align}

	Every opposite Demazure crystal atom contains exactly one key tableau. The left key map sends each tableau to the key tableau  in the opposite Demazure crystal atom that contains the tableau.

	\subsubsection{  Relations among Demazure crystals}
	\begin{prop}Let $u,v, x,y\in W\lambda$  and $b\in \mathfrak{B}(\lambda)$. Then
\begin{enumerate}
\item \cite[Chapitre 8]{kash2} $K^-(b)\le  K^+(b)$ and $K^-(b)= K^+(b)$  $\Leftrightarrow$ $b$ is a  key tableau in ${O}(\lambda)$.
\item
${\mathfrak{B}}_{x }\subseteq {\mathfrak{B}}_{y }\Leftrightarrow {\mathfrak{B}}^{x }\supseteq {\mathfrak{B}}^{y }\Leftrightarrow x\le y$.
\item  $\overline{\mathfrak{B}}^{u}\cap \overline{\mathfrak{B}}_{v}\neq \emptyset \Leftrightarrow u\le v$ which happens when $$\overline{\mathfrak{B}}^{u }\cap \overline{\mathfrak{B}}_{v }=\{b \in B(\lambda)\mid K(u)=K^-(b)\le K^+(b)=K(v)\}\supseteq \{K(u),K(v)\}.$$
    In particular, $\overline{\mathfrak{B}}^{v }\cap \overline{\mathfrak{B}}_{v }=\{K(v)\}$ and $\overline{\mathfrak{B}}^{v }\cap \overline{\mathfrak{B}}_{w_0v }\neq \emptyset$.
\item \cite[Proposition 4.4]{kashiwara1992crystal} ${\mathfrak{B}}^{u}\cap {\mathfrak{B}}_{v}\neq \emptyset \Leftrightarrow u\le v$ which happens when $${\mathfrak{B}}^{u}\cap {\mathfrak{B}}_{v}=\{b \in B(\lambda)\mid K(u)\le K^-(b)\le K^+(b)\le K(v)\}\supseteq \{K(z)\mid z\in[u,v]\},$$ 
where $[u,v]\subseteq W$ is an interval in the Bruhat order.
In particular, ${\mathfrak{B}}^{v}\cap {\mathfrak{B}}_{v}=\{K(v)\}$.

    \end{enumerate}
\end{prop}
\begin{proof}
$(1)$ From Theorem \ref{Th_Dila}, \cite[Chapitre 8]{kash2},  $K^-(b)= K^+(b)= K(z)$, for some $z\in W\lambda$, then $b=K(z)$.

$(2)$  Assuming $ x\le y\Leftrightarrow {\mathfrak{B}}_{x }\subseteq {\mathfrak{B}}_{y }$ then since $\xi$ is an involution
\begin{align*}{\mathfrak{B}}_{x }\subseteq {\mathfrak{B}}_{y }&\Leftrightarrow x\le y\Leftrightarrow w_0x\ge w_0y\Leftrightarrow  {\mathfrak{B}}_{w_0x }\supseteq {\mathfrak{B}}_{w_0y }\\
&\Leftrightarrow \xi({\mathfrak{B}}_{w_0x })\supseteq \xi({\mathfrak{B}}_{w_0y })
\Leftrightarrow {\mathfrak{B}}^{x }\supseteq {\mathfrak{B}}^{y }.
\end{align*}
It remains to prove ${\mathfrak{B}}_{x }\subseteq {\mathfrak{B}}_{y }\Rightarrow x\le y$. Let ${\mathfrak{B}}_{x }\subseteq {\mathfrak{B}}_{y }$. From \eqref{Boverline},
 $$\overline {\mathfrak{B}}_x=\{b\in \mathfrak{B}(\lambda): K^+(b)=K(x)\}\subseteq {\mathfrak{B}}_{x }\subseteq {\mathfrak{B}}_y.$$
Since, for $b\in  {\mathfrak{B}}(\lambda)$, $K^+(b) $ is uniquely determined, then from \eqref{Bv}, for $b\in \overline{\mathfrak{B}}_x \subseteq {\mathfrak{B}}_y$,  $K^+(b)=K(x)\le K(y)\Leftrightarrow x\le y$.

$(3)$ From \eqref{opBoverlinatom} and  \eqref{Boverline},
$\overline{\mathfrak{B}}^{u}=\{b\in B(\lambda): K^-(b)=K(u)\},\text{ and } \overline{\mathfrak{B}}_{v}=\{b\in B(\lambda): K^+(b)=K(v)\}
$
Therefore, \begin{align*}\overline{\mathfrak{B}}^{u}\cap \overline{\mathfrak{B}}_{v}=\{b \in B(\lambda)\mid K(u)=K^-(b)\le K^+(b)=K(v)\}\neq \emptyset\Leftrightarrow u\le v.
\end{align*}

$(4)$ From \eqref{opBoverlinatom},  \eqref{Boverline} and Assertion 3, \begin{align*}{\mathfrak{B}}^{u}\cap {\mathfrak{B}}_{v}=&\bigsqcup_{u'\in W\lambda,\,u'\ge u} \overline{\mathfrak{B}}^{u' }\bigcap \bigsqcup_{v'\in W\lambda,\,v'\le v} \overline{\mathfrak{B}}_{v' }\\
=&\bigsqcup_{\begin{smallmatrix}u', v'\in W\lambda,\\
u\le u'\le v'\le v
\end{smallmatrix}}\overline{\mathfrak{B}}^{u' }\cap \overline{\mathfrak{B}}_{v' }\\
=&\bigsqcup_{ [u', v']\subset [u, v]}\overline{\mathfrak{B}}^{u' }\cap \overline{\mathfrak{B}}_{v' }\\
=& \{b \in B(\lambda)\mid K(u)\le K^-(b)\le K^+(b)\le K(v)\}.
\end{align*}

Therefore,
${\mathfrak{B}}^{u}\cap {\mathfrak{B}}_{v}\neq \emptyset \Leftrightarrow u\le v$ which happens when ${\mathfrak{B}}^{u}\cap {\mathfrak{B}}_{v}\supseteq \{K(z)\mid z\in[u,v]\}$.
In particular,
\begin{align*}\label{Bintersect}{\mathfrak{B}}^{v}\cap {\mathfrak{B}}_{v}
=&\overline{\mathfrak{B}}^{v }\cap \overline{\mathfrak{B}}_{v }\\
=&\{b\in B(\lambda): K^-(b)=K(v)\}\cap \{b\in B(\lambda): K^+(b)=K(v)\}\\
=& \{b\in B(\lambda): K^-(b)=K(v)= K^+(b)\}, \text{  by  assertion $(1)$,}\\
=&\{K(v)\}.
\end{align*}

\end{proof}

\begin{obs} \begin{enumerate}

\item Note that that assertion $(4)$ is consistent with ${\mathfrak{B}}_{v}$ and ${\mathfrak{B}}^{u}$,
\begin{align*}{\mathfrak{B}}^{\lambda}\cap {\mathfrak{B}}_{v}&=\{b \in B(\lambda)\mid K(\lambda)\le K^-(b)\le K^+(b)\le K(v)\}\\
&=\{b \in B(\lambda)\mid K^+(b)\le K(v)\}={\mathfrak{B}}_{v}.
\end{align*}

\item

\begin{align*}{\mathfrak{B}}^{u}\cap {\mathfrak{B}}_{w_0\lambda}&=\{b \in B(\lambda)\mid K(u)\le K^-(b)\le K^+(b)\le K(-\lambda)\}\\
&=\{b \in B(\lambda)\mid K(u)\le K^-(b)\}={\mathfrak{B}}^{u}.
\end{align*}
\end{enumerate}
\end{obs}

Therefore
\begin{defin} we define ${\mathfrak{B}}^u_v$ the \emph{$u,v$-Demazure crystal} ${\mathfrak{B}}_v$,
\begin{equation}{\mathfrak{B}}^u_v:={\mathfrak{B}}^{u}\cap {\mathfrak{B}}_{v}=\bigsqcup_{ [u', v']\subset [u, v]}\overline{\mathfrak{B}}^{u' }\cap \overline{\mathfrak{B}}_{v' }=\bigsqcup_{ [u', v']\subset [u, v]} \overline{\mathfrak{B}}^{u' }_{v' },\label{demazureinterval}
\end{equation}
where ${\mathfrak{B}}^\lambda_v={\mathfrak{B}}_v$ and ${\mathfrak{B}}^u_{w_0\lambda}={\mathfrak{B}}^u$, that is, when $u=\lambda$ and $v=w_0\lambda$ we get the Demazure ${\mathfrak{B}}_v$ and opposite Demazure ${\mathfrak{B}}^u$ respectively.

Similarly, $\overline{\mathfrak{B}}^u_v$ the \emph{$u$-bounded Demazure atom crystal} ${\mathfrak{B}}_v$,
\begin{align}\overline{\mathfrak{B}}^{u }_v:=\overline{\mathfrak{B}}^{u }\cap \overline{\mathfrak{B}}_{v }.
\end{align}
\end{defin}
\begin{prop}For $u,v,u',v'\in W\lambda$, $$\overline{\mathfrak{B}}^u_v \subseteq \overline{\mathfrak{B}}^{u'}_{v'}\Leftrightarrow v'\le v\le u\le u'.$$

\end{prop}

\subsubsection{Symplectic keys and
  standard monomial theory for Schubert and Richardson varieties}\label{standardlaklitt} Demazure and opposite Demazure crystals and their intersections are in natural correspondence correspondence with Schubert, opposite Schubert varieties and Richardson varieties respectively as illuminated in \cite{fujita22} through  a polyhedral lens.  The similarity  between  Demazure crystals and Schubert, opposite  and Richardson varieties is explained by the classical { Borel-Weil} theorem. (We refer to \cite{brion,fulton1997young,speyer,wooyong23} for the geometric definitions and properties.) Next we collect { bunch of } results for the convenience of the reader.

  Let $G$ be a   semisimple algebraic group over $k$ closed algebraic field,  and $T\subseteq G$ a maximal torus. We also fix  $T\subseteq B \subseteq G$,  $B$ a Borel subgroup of $G$ (a subgroup of $G$ containing a maximal torus).  Let  $B^-$ be the corresponding opposite Borel subgroup, that is, it is the unique Borel subgroup of $G$ with
the property $B\cap B^-=T$.
Consider the Bruhat decomposition \begin{align*} G=\bigsqcup_{w\in W} BwB=\bigsqcup_{w\in W}B^-wB.
\end{align*}The quotient space $G/B$, called full flag variety, inherits a disjoint finite decomposition into cells (respect. opposite cells) $$G/B=\bigsqcup_{w\in W} BwB/B=\bigsqcup_{w\in W}B^-wB/B.$$

Let $C_v:=BvB/B$ and $C^v:=B^-vB/B$. The Schubert variety $X_w$, respectively the opposite Schubert variety $X^w$, in $G/B$ are
$$X_w=\bigsqcup_{v\le w}C_v,\;\; X^w=\bigsqcup_{u\ge w}C^u=w_0X_{w_0w}\subseteq G/B,$$
where $\le $ denotes the (strong) Bruhat order on $W$.

For $\lambda$ a dominant weight of $G$, let $L_\lambda = G\times_B k_{-\lambda}$ be the line bundle on $G/B$ associated to the $B$-character
$(-\lambda)$. That is, we consider $k_{-\lambda}$  as a $B$-module  where
$ b .z = (-\lambda)(b)z$, $z\in k$.

 Let $V (\lambda)$ be the Weyl module of highest
weight $\lambda$.
Let $\mathfrak{B}(\lambda)$ be the Kashiwara crystal of highest weight $\lambda$ and
$\textbf{B}(\lambda)$ be the set of LS paths of shape $\lambda$.
The crystal of L-S paths $\textbf{B}(\lambda)$ is isomorphic to the Kashiwara crystal $\mathfrak{B}(\lambda)$.
If $\mathbf{\pi}=(\mathbf{\tau},\mathbf{a})$ is an LS path of shape $\lambda$, the sequence $\tau = (\tau_0,\dots,\tau_r )$
{ is strictly decreasing in  $W/W_\lambda$}. The initial direction $i(\pi) = \tau_0$,
and the ending direction  $e(\pi) = \tau_r$ of the path $\pi$  coincide with the right key respectively left key of the corresponding vertex in the Kashiwara crystal $\mathfrak{B}(\lambda)$.

Following  mostly \cite{laklitt}, the LS path $\mathbf{\pi}$ in  $\mathbf{B}(\lambda)$ is said to be {\emph{standard} on a Richardson variety $X^\kappa_\tau=X^\kappa\cap X_\tau$, $\kappa,\tau \in W^\lambda$
 if  $\kappa\le e(\pi)\le i(\pi)\le \tau$ in the Bruhat order in $ W^\lambda$. We denote by $\mathbf{B}_\tau^\kappa (\lambda)$ the set of
all LS paths of shape $\lambda$, \emph{standard} on $X^\kappa_\tau$.
If $\tau=w_0$ (respectively $\kappa$) in $W^\lambda$
 (respectively $id$), then $\tau$ (respectively $\kappa$) will be omitted and we will write just
$\mathbf{B}^\kappa(\lambda)=\{\pi \in \mathbf{B}(\lambda):e(\pi)\ge \kappa\}$, called opposite Demazure crystal,  the set of
all L–S paths of shape $\lambda$, \emph{standard} on the opposite Schubert variety $X^\kappa$ (respectively $\mathbf{B}_\tau (\lambda)=\{\pi \in \textbf{B}(\lambda):i(\pi)\le \tau\}$ called the Demazure crystal,  the set of
all L–S paths of shape $\lambda$, \emph{standard} on the Schubert variety $X_\tau$). Therefore $\mathbf{B}_\tau^\kappa (\lambda)=\mathbf{B}_\tau\cap \mathbf{B}^\kappa (\lambda)$.
In the Kashiwara crystal this is equivalent to say that $b\in \mathfrak{B}(\lambda)$ is \emph{standard} on the Richardson variety $X^\kappa_\tau$ if $\kappa\le K^-(b)\le K^+(b)\le \tau
$ which means $b\in \mathfrak{B}^\kappa_\tau(\lambda)=\mathfrak{B}^\kappa\cap \mathfrak{B}_\tau(\lambda)$.

The space of global sections $H^0(G/B,L_\lambda)$ as a $G$-module is (isomorphic to)  the dual  of the module $V(\lambda)$,
        $$H^0(G/B,L_\lambda)\simeq V(\lambda)^*.$$

        Associated to the combinatorial LS path model ${B}(\lambda)$ of LS paths of shape $\lambda$ we have the \emph{path vector basis} $$\mathbb{B}(\lambda)=\{p_\tau: \tau \in \mathbf{B}(\lambda)\}\subseteq H^0(G/B,L_\lambda)=V(\lambda)^*.$$
        It is shown that $\mathbb{B}(\lambda)$ is compatible with Schubert and opposite Schubert varieties $Z\subseteq G/B$, that is, the set
$\{b_{\mid Z} : b \in \mathbb{B}(\lambda), b_{\mid Z} \not\equiv 0\}$ is linearly independent; $\mathbb{B}(\lambda)$ satisfies certain quadratic relations similar to the quadratic straightening.
For $G=GL_n(\mathbb{C}), Sp(2n,\mathbb{C)}$, associated to  the combinatorial tableau model corresponding to the Kashiwara crystal $\mathfrak{B}(\lambda)$, a basis with the same properties has been constructed, in each case. For $GL_n(\mathbb{C})$ see \cite{reinershimoz,speyer} and \cite[Chapter III]{fulton1997young}, for instance.
For $G=Sp(2n,\mathbb{C)}$ one has De Concini's construction \cite{de1979symplectic} for $V(\lambda)$.

For $\tau \in W/W_\lambda$ let $v_\tau \in  V (\lambda)$ be a extremal weight vector of weight $\tau.\lambda$. The $B$-submodule spanned by the orbit $B.v_\tau$ is the
Demazure module $V_\tau (\lambda)$ associated to $\tau$. The  $B^-$ submodule spanned by the orbit $B^-.v_\tau$ is
 the opposite Demazure module $V_\tau (\lambda)$ associated to $\tau$.

   \begin{prop} \cite{laklitt} The path vector basis $\mathbb{B}(\lambda)$ is compatible with the Demazure submodules,
i.e., the restrictions $\{{p_\pi}_{\mid V_\tau (\lambda)^*} | \pi \in B_\tau (\lambda)\}$ form a basis of $V_\tau (\lambda)^*$, and the restrictions
of the other path vectors ($\pi\notin B_\tau(\lambda)$) vanish on the submodule. Similarly, the restrictions $\{{p_\pi}_{\mid V^\tau (\lambda^*} | \pi \in B^\tau (\lambda)\}$ form a basis of $V^\tau(\lambda)^*$ and the restrictions of the other path vectors  ($\pi\notin B^\tau(\lambda)$) vanish.

Further, the restriction of the section $p_\pi\in \mathbb{B}(\lambda)$ to $X_\tau$ vanishes if and only if $i(\pi) \nleqslant \tau$. The set of path vectors $\{p_\pi \in \mathbb{B}(\lambda): i(\pi) \le \tau\}=\{p_\pi: \pi \in B_\tau (\lambda)\}$  forms a basis of $H^0(X_\tau, L_\lambda)=V_\tau(\lambda)^*$. Similarly, the set of path vectors $\{p_\pi \in \mathbb{B}(\lambda): e(\pi) \ge \tau\}$  forms a basis of $H^0(X^\tau, L_\lambda)=V^\tau(\lambda)^*$.

 The set
$$\mathbb{B}_\tau(\lambda)=\{p_\pi: \pi \in B_\tau(\lambda)\}\subseteq H^0(X_\tau,L_\lambda)$$
and
$$\mathbb{B}^\tau(\lambda)=\{p_\pi: \pi \in B(\lambda)^{\tau}\}\subseteq H^0(X^\tau,L_\lambda).$$

The set of path vectors  $\mathbb{B}^\sigma_\kappa(\lambda)=\{p_\pi: \pi \in B^\sigma_\kappa (\lambda)\}$ form a basis for $H^0(X^\sigma_\kappa,L_\lambda)$.
\end{prop}

    \begin{cor}  \cite{laklitt}  Let $\mathbb{B}(\lambda)^*= \{u_\pi \in V (\lambda) | \pi \in \mathbb{B}(\lambda)\}$ be the basis of $V (\lambda)$ dual to the path vector basis
of $H^0(G/B,L_\lambda)$.
 The vectors $\{u_\pi | \pi \in  B_\tau (\lambda)\}$ form a basis of the Demazure module $V_\tau (\lambda)$,
the vectors $\{u_\pi | \pi \in B_\tau (\lambda)\}$ form a basis of the opposite Demazure module $V _\tau (\lambda)$, and
the vectors $\{u_\pi | \pi \in B^\sigma_\tau(\lambda)\}$ form a basis of the intersection $V^\sigma_
\tau (\lambda) \}= V_\tau (\lambda) \cap V_\sigma (\lambda)$.

\end{cor}

Let $\pi$ be a LS path of shape $\lambda$.
The restriction of the section $p_\pi\in \mathbb{B}(\lambda)$ to $X_\tau$ vanishes if and only if $i(\pi) \nleqslant \tau$. Further, the set of path vectors $\{p_\pi \in \mathbb{B}(\lambda): i(\pi) \le \tau\}$ of shape $\lambda$ forms a basis of $H^0(X_\tau, L_\lambda)=V^*(\lambda)_\tau$.

\begin{cor} \begin{enumerate}\item The following are equivalent \begin{enumerate}

\item ${{\textbf{B}}}_{x }(\lambda)\subseteq {{\textbf{B}}}_{y }(\lambda)\Leftrightarrow {{\textbf{B}}}^{x }(\lambda)\supseteq {{\textbf{B}}}^{y }(\lambda)$.
\item $V_{x }(\lambda)^*\subseteq V_{y }(\lambda)^*\Leftrightarrow V^{x }(\lambda)^*\supseteq V^{y }(\lambda)^*$.
\item $x\le y$.
\end{enumerate}
\item The following are equivalent
\begin{enumerate}
\item $\textbf{B}^\sigma_\tau(\lambda)\neq \emptyset$.
\item $V^\sigma_
\tau (\lambda) \}= V_\tau (\lambda) \cap V_\sigma (\lambda)\neq \{0\}$
\item $\sigma\le \tau$
\end{enumerate}
\end{enumerate}
\end{cor}

\medskip

For $w\in W/W_\lambda$, the space of global sections $H^0(X_w,L_\lambda)$ ($H^0(X^w,L_\lambda)$ as a $B^-$-module) is isomorphic to $V_w(\lambda)^*$ ($V^w(\lambda)^*$) the dual of the Demazure module $V_w(\lambda)$ (the opposite Demazure module $V^w(\lambda)$),
        $$H^0(X_w,L_\lambda)\simeq V_w(\lambda)^*, \quad (H^0(X^w,L_\lambda)\simeq V^w(\lambda)^*).$$

        In particular, $X_{w_0}=X^{id}=G/B$ and $V_{w_0}(\lambda)=V^{id}(\lambda)=V(\lambda)$.

     \medskip
        Facts on Schubert varieties
\begin{prop}
\item For $w,w'\in W/W_\lambda$, the following are equivalent
\begin{enumerate}
\item
$X_w\subseteq X_{w'}$.
\item   $X^w\supseteq X^{w'}$
     \item    $X^w \cap X_{w'} \neq\emptyset$
     \item  $w\leq w'$
     \end{enumerate}
     \item For $\alpha,\beta,\alpha',\beta'$, the following are equivalent
     \begin{enumerate}
\item $X_\alpha^\beta\subseteq X_{\alpha'}^{\beta'}$
        \item $ \beta'\le\beta\le\alpha\le\alpha'$
        \item $[\beta,\alpha]\subseteq [\beta',\alpha']$
        \end{enumerate}
\end{prop}

 By the Bore-Weil theorem, similar facts on Demazure crystals

\begin{prop}
\item For $w,w'\in W/W_\lambda$, the following are equivalent
\begin{enumerate}
\item
$\mathfrak{B}_w(\lambda)\subseteq \mathfrak{B}_{w'}(\lambda)$.
\item   $\mathfrak{B}^w(\lambda)\supseteq \mathfrak{B}^{w'}(\lambda)$
     \item    $\mathfrak{B}^w(\lambda)\cap \mathfrak{B}_{w'}(\lambda) \neq\emptyset$
     \item  $w\leq w'$
     \end{enumerate}
     \item For $\alpha,\beta,\alpha',\beta'$, the following are equivalent
     \begin{enumerate}
\item $B_\alpha^\beta(\lambda)\subseteq \mathfrak{B}_{\alpha'}^{\beta'}(\lambda)$
        \item $ \beta'\le\beta\le\alpha\le\alpha'$
        \item $[\beta,\alpha]\subseteq [\beta',\alpha']$
        \end{enumerate}
\end{prop}

\begin{figure}[h]
		\begin{center}

			${
				\begin{tikzpicture}
					[scale=.5,auto=left]
					\node (n0) at (0,14.5) {$\thiYT{0.17 in}{}{{{1},{1}},{{2}}}$};
					\node (n1l) at (-3.5,13.5)  {$\thiYT{0.17 in}{}{{{1},{2}},{{2}}}$};
					\node (n1r) at (3.5,13.5)  {$\thiYT{0.17 in}{}{{{1},{1}},{{\overline{2}}}}$};
					\node (n2l) at (-3.5,11)  {$\YT{0.17 in}{}{{{1},{\overline{2}}},{{2}}}$};
					\node (n2r) at (3.5,11) {$\YT{0.17 in}{}{{{1},{2}},{{\overline{2}}}}$};
					\node (n3r) at (3.5,8.5)  {$\YT{0.17 in}{}{{{2},{2}},{{\overline{2}}}}$};
					\node (n4rr) at (5,6)  {$\thiYT{0.17 in}{}{{{2},{2}},{{\overline{1}}}}$};
					\node (n4r) at (2,6)  {$\YT{0.17 in}{}{{{2},{\overline{2}}},{{\overline{2}}}}$};
					\node (n5r) at (3.5,3.5)  {$\YT{0.17 in}{}{{{2},{\overline{2}}},{{\overline{1}}}}$};
					\node (n6r) at (3.5,1)  {$\thiYT{0.17 in}{}{{{\overline{2}},{\overline{2}}},{{\overline{1}}}}$};
					\node (n3ll) at (-5,8.5)  {$\thiYT{0.17 in}{}{{{1},{\overline{2}}},{{\overline{2}}}}$};
					\node (n4l) at (-3.5,6)  {$\YT{0.17 in}{}{{{1},{\overline{1}}},{{\overline{2}}}}$};
					\node (n3l) at (-2,8.5)  {$\YT{0.17 in}{}{{{1},{\overline{1}}},{{2}}}$};
					\node (n5l) at (-3.5,3.5)  {$\YT{0.17 in}{}{{{2},{\overline{1}}},{{\overline{2}}}}$};
					\node (n6l) at (-3.5,1)  {$\thiYT{0.17 in}{}{{{2},{\overline{1}}},{{\overline{1}}}}$};
					\node (n7) at (0,0)  {$\thiYT{0.17 in}{}{{{\overline{2}},{\overline{1}}},{{\overline{1}}}}$};
					
					\draw (0,2.5)--(-5,2);
					\draw (0,2.5)--(5,2);
					\draw (0,2.5)--(-2,15.5);
					\draw (0,2.5)--(2,15.5);
					\draw (0,2.5)--(6,9);
					\draw (0,2.5)--(-5.5,12.5);
					\draw (0,2.5)--(-2,-.5);
					\draw (0,2.5)--(2,-.5);
					
					\draw[->] [draw=red,  thick] (-2.5,0.7)--(-1,0.3);
					\draw[->] [draw=blue,  thick] (2.5,0.7)--(1,0.3);
					\draw[->] [draw=blue,  thick] (-3.5,2.5)--(-3.5,2);
					\draw[->] [draw=red,  thick] (3.5,2.5)--(3.5,2);
					\draw[->] [draw=blue,  thick] (-3.5,5)--(-3.5,4.5);
					\draw[->] [draw=blue,  thick] (2.5,5)--(3,4.5);
					\draw[->] [draw=red,  thick] (4.5,5)--(4,4.5);
					\draw[->] [draw=red,  thick] (-2.5,7.5)--(-3,7);
					\draw[->] [draw=blue,  thick] (-4.5,7.5)--(-4,7);
					\draw[<-] [draw=blue,  thick] (4.5,7)--(4,7.5);
					\draw[<-] [draw=red,  thick] (2.5,7)--(3,7.5);
					\draw[->] [draw=blue,  thick] (-3,10)--(-2.5,9.5);
					\draw[->] [draw=red,  thick] (-4,10)--(-4.5,9.5);
					\draw[->] [draw=blue,  thick] (3.5,10)--(3.5,9.5);
					\draw[->] [draw=red,  thick] (-3.5,12.5)--(-3.5,12);
					\draw[->] [draw=blue,  thick] (3.5,12.5)--(3.5,12);
					\draw[->] [draw=blue,  thick] (-1,14.2)--(-2.5,13.8);
					\draw[->] [draw=red,  thick] (1,14.2)--(2.5,13.8);
				\end{tikzpicture}
			}$	

		\end{center}
		
	\caption{\label{cristal21op}
	The $C_2$ crystal graph $\mathfrak{B}(2,1)$  split into opposite Demazure crystal atoms.
}
\end{figure}

	\subsection{Demazure characters and opposite Demazure characters and interval Demazure characters}
	The Demazure characters, or key polinomials, and Demazure atoms can be seen as generating functions of the tableaux in Demazure crystals. Given $v\in B_n\lambda=[\lambda,-\lambda]$:
	$$\kappa_v(x_1, \dots, x_n)=\sum_{T\in\mathfrak{B}_v}x^{\text{wt} T},\;\;
	\overline{\kappa}_v(x_1, \dots, x_n)=\sum_{T\in\overline{\mathfrak{B}}_v}x^{\text{wt} T},$$
	and we can define, analogously, opposite Demazure characters and opposite Demazure atoms:
	$$\kappa^{-v}(x_1, \dots, x_n)=\sum_{T\in\mathfrak{B}^{-v}}x^{\text{wt} T},\;\;
	\overline{\kappa}^{-v}(x_1, \dots, x_n)=\sum_{T\in\overline{\mathfrak{B}}^{-v}}x^{\text{wt} T}.$$
	
	Since the tableau weights in $\mathfrak{B}_v$ and in $\mathfrak{B}^{-v}$ \eqref{oppxi} are symmetric in $\mathbb{Z}^n$,
$$\kappa^{-v}(x_1, \dots, x_n)=\sum_{T\in\mathfrak{B}^{-v}}x^{\text{wt} T}=\sum_{T\in\mathfrak{B}_{v}}x^{-\text{wt} T}=\sum_{T\in\mathfrak{B}_{v}}(x^{-1})^{
\text{wt} T}=\kappa_v(x_1^{-1}, \dots, x_n^{-1})$$
 we have the following result:
	\begin{cor}
		$$\kappa_v(x_1, \dots, x_n)=\kappa^{-v}(x_1^{-1}, \dots, x_n^{-1})$$
	\end{cor}
	As a consequence, for instance, the type $C_n$ Fu-Lascoux non-symmetric Cauchy kernel, given in \cite{fu2009non}, can be written as:
	
	\begin{align*}
		\frac{\prod_{1\leq i<j\leq n}(1-x_ix_j)}{\prod_{i,j=1}^{n}(1-x_iy_j)\prod_{i,j=1}^{n}(1-x_i/y_j)}&=\sum_{v\in \N^n}\overline{\kappa}_v(x_1, \dots, x_n) \kappa_{-v}(y_1,\dots, y_n)\\
		&=\sum_{v\in \N^n}\overline{\kappa}_v(x_1, \dots, x_n) \kappa^{\nu}(y_1^{-1},\dots, y_n^{-1})
	\end{align*}
equivalently, putting $\x=(x_1, \dots, x_n)$ and $\y^{-1}=(y_1^{-1}, \dots, y_n^{-1})$,	

\begin{equation}\label{C}\displaystyle{\prod_{1\le i,j\le n}\displaystyle(1-x_iy_j)^{-1}\prod_{1\le i\le j\le n}(1-{x_i}{y_j}^{-1})^{-1}}=
\sum_{\nu\in\mathbb{N}^n}\overline\kappa_\nu(\x){\kappa^{\nu}}^C(\y)(\sum_{\beta'\; even}s_\beta(\x)),
\end{equation}

	See also \cite{choi2018lakshmibai,agl} for  type $A$ non-symmetric Cauchy kernels \cite{agl}.

The Demazure character interval
$$\kappa^u_v(x_1, \dots, x_n)=\sum_{T\in\mathfrak{B}^u_v}x^{\text{wt} T}=\sum_{\begin{smallmatrix}T\\
 K(u)\le K^-(T)\le K^+(T)\le K(v)
\end{smallmatrix}}x^{\text{wt} T},\quad \overline{\kappa}^u_v(x_1, \dots, x_n)=\sum_{T\in\overline{\mathfrak{B}}^u_v}x^{\text{wt} T}$$

	\section{Cocrystal of a KN tableau}\label{coc}

Motivated by Lascoux’s double crystal graph construction in
	type $A$ \cite{lascoux2003double}, and by Heo-Kwon work in \cite{kwonheococrystal2020} where Sch\"utzenberger \emph{jeu de
		taquin} slides are used as crystal operators for $\mathfrak{sl}_2$, the cocrystal of each KN
	tableau in the type $C_n$ crystal $\mathfrak{B}(\lambda)$ is introduced. These cocrystals contain all the needed information
	to compute right and left keys of a tableau in the type $C_n$ crystal $\mathfrak{B}(\lambda)$ and refine
	previous second author's construction of  symplectic key maps in \cite{santos2019symplectic} based on the symplectic \emph{jeu de taquin}.
	
	\subsection{Dual RSK correspondence}
	In this subsection we work with semistandard Young tableaux. Given a tableau $T\in \mathcal{SSYT}(\lambda, n)$ with column decomposition $T=C_1C_2\cdots C_k$, its \emph{column reading word}, $\text{cr}(T)$, is the word obtained after concatenating all of its column { readings} $\text{cr}(C_i)$ from right to left: $\text{cr}(T):=\text{cr}(C_k)\cdots\text{cr}(C_2)\text{cr}(C_1)$ where $\text{cr}(C_i)$ is the word obtained reading the column $C_i$ top to bottom.
	\begin{ex}\label{exrsk}
		Given $T=\YT{0.14in}{}{
			{{1},{2},{2}},
			{{2},{3}},
			{{4},4}}\in \mathcal{SSYT}((3,2,2,0),4)$,  the column reading of $T$ is  $\text{cr}(T)=2\, 234\, 124$.
	\end{ex}	
	Given a column reading word $w$, we can recover the original tableau via \textit{column insertion}  \cite{fulton1997young}:
	Let $w=w_1\cdots w_\ell$. We start with $i := 1$, $T=\emptyset$ , the empty tableau, and $p=1$.
	\begin{enumerate}
		\item If $w_i$ is bigger than all entries of $C_p$, just add a cell to the { bottom} column $C_p$ with entry $w_i$. Else find $\alpha \in C_p$ the smallest entry of $C_p$ bigger or equal than $w_i$. Then replace $\alpha$ by $w_i$ in $C_p$ and redefine $w_i:=\alpha$, $p:=p+1$ and go to $(2)$ (this is called a \textit{bumping}).
		\item If $i\neq \ell$, then $i:=i+1$ and go to $(1)$. Else the algorithm ends.
	\end{enumerate}
	
	Given $r\geq 1$, let $E_n^r$ be the set of bi-words, two rowed arrays, $w=\begin{pmatrix}
		u_1 \dots  u_r\\ v_1 \dots v_r

		\end{pmatrix}$   without repeated bi-letters, in lexicographic order,
$$\begin{pmatrix}
		u_i\\ v_i
		
	\end{pmatrix}\leq \begin{pmatrix}
		u_{i+1}\\ v_{i+1}
		
	\end{pmatrix} \text{ if $u_i<u_{i+1}$ or if $u_i=u_{i+1}$ and $v_i\leq v_{i+1}$,}$$ where the  bottom word $v_1\cdots v_r$ of $w$ is in the alphabet $[n]$, and the  top word $u_1\cdots u_r$  of $w$ is in the alphabet $[r]$.
	
The set $E_n^r$ can also be thought as
	the set of sequences of $r$ columns, possibly some of them empty, on the alphabet $[n]$, where each pair of consecutive columns has maximum overlapping, and, in the case of two non-empty columns  { whose intermediate columns} are empty, the top edge of the left column and the bottom edge of the right column are aligned. That is, in this representation a bi-word $w=\begin{pmatrix}
		u\\ v

		\end{pmatrix}$ in $E_n^r$ is a  semistandard skew tableau $T$ in the alphabet $[n]$, whose \emph{column reading word is the bottom word in the bi-word, $cr(T)=v$}, with at most  $r$ columns such that  column $i$, \textit{counted right to left}, is filled with all bottom letters in the bi-letters whose top letter is  $i$, and  each pair of consecutive columns $C_{i+1}$ and $C_{i}$ form a skew semistandard tableau where the number of rows of length two is maximized. This means that the \emph{sequence of column lengths, right to left, is equal to the weight of the top word $u$ in the bi-word}. For instance,
$$\begin{array}{ccccc}\begin{tikzpicture}[scale=.444, baseline={([yshift=-.8ex]current bounding box.center)}]
						\draw (0,0) rectangle +(1,1);
						\draw (0,1) rectangle +(1,1);
						\draw (0,2) rectangle +(1,1);
						\draw (1,2) rectangle +(1,1);
						\draw (2,2) rectangle +(1,1);
						\draw (2,3) rectangle +(1,1);
						\draw (2,4) rectangle +(1,1);
						\node at (.5,.5) {$4$};
						\node at (.5,1.5) {$2$};
						\node at (.5,2.5) {$1$};
						\node at (1.5,2.5) {$2$};	
						\node at (2.5,2.5) {$4$};
						\node at (2.5,3.5) {$3$};
						\node at (2.5,4.5) {$2$};
					\end{tikzpicture}\quad
\text{ and
$\begin{pmatrix} 1&1&1&2&3&3&3\\
2&3&4&2&1&2&4
\end{pmatrix}$ \; are identified in $E_4^3.$}\end{array}$$

 In particular, $E_n^r$ has a subset identified with $\bigsqcup\limits_{\substack{ \ell (\lambda)\leq n\\\ell (\lambda')\leq r}}\mathcal{SSYT}(\lambda, n)$  where $\ell(\lambda')$ is the length of $\lambda'$, the conjugate partition of $\lambda$.
	Given a tableau $T\in \mathcal{SSYT}(\lambda, n)$, we create a bi-word, without repeated bi-letters, whose bottom word is $cr(T)$ and in the top word we register in which column of $T$, counted from the right, was each letter of $\text{cr}(T)$ read. Each biword will be an element of $E_n^r$, where $\ell(\lambda')\leq r$. In Example \ref{exrsk}, the bi-word of $T$ is
	$$w=\begin{pmatrix}
		1 & 2 & 2 & 2 & 3 & 3 & 3\\
		2 & 2 & 3 & 4 & 1 & 2 & 4
	\end{pmatrix}\in E_4^3,\, \text{with }\ell(\lambda')=3.$$
	
	The dual RSK, $RSK^\ast$, is a bijection \cite[Section A.4.3]{fulton1997young} between $E_n^r$ and pairs of SSYT's of conjugate shapes and lengths $\leq n$ and $\leq r$, respectively:{\scriptsize\begin{align*}\scriptsize
			RSK^\ast: E_n^r&\rightarrow \bigsqcup\limits_{\substack{ \ell (\lambda)\leq n\\\ell (\lambda')\leq r}} \mathcal{SSYT}(\lambda,n)\times \mathcal{SSYT}(\lambda',r)=\bigsqcup\limits_{\substack{ \ell (\lambda)\leq n\\\ell (\lambda')\leq r\\P\in \mathcal{SSYT}(\lambda,n)}} \{P\}\times \mathcal{SSYT}(\lambda',r)\\
			w&\mapsto (P,Q).
	\end{align*}}
	
	The bijection $RSK^*$ can be calculated in the following way:
	
	Let $w=\begin{pmatrix}
		x_1 & x_2 & \dots & x_m\\
		y_1 & y_2 & \dots & y_m
	\end{pmatrix}$. Then start with $i=1$ and $P=Q$  empty tableaux.
	\begin{enumerate}
		\item Column insert $y_i$ into $P$.
		\item Add one cell to $Q$ whose entry is $x_i$, in a position such that $P$ and $Q$, with this new cell, have conjugate shapes.
		\item If $i\neq m$, then $i:=i+1$ and return to $(1)$. Else the algorithm is finished.
	\end{enumerate}
	
Given a bi-word $w$, the first and second components of $RSK^\ast(w)$ are called the $P$- symbol and the $Q$-symbol of $w$ respectively.
	
	The $RSK^\ast$ maps  the bi-word of a skew SSYT $\tilde T$ in the alphabet $ [n]$ with  at most $r$  columns, to $(rect(\tilde T), Q)$ where  $\text{rect}(\tilde{T})$ is the rectification of $\tilde{T}$ via SJDT, and the weight of $Q$  is  the sequence of column lengths of $\tilde{T}$ from right to left.
\begin{ex}\label{Ttilde} Let  $\,\tilde{T}=\begin{tikzpicture}[scale=.444, baseline={([yshift=-.8ex]current bounding box.center)}]
		\draw (0,0) rectangle +(1,1);
		\draw (0,1) rectangle +(1,1);
		\draw (0,2) rectangle +(1,1);
		\draw (1,1) rectangle +(1,1);
		\draw (1,2) rectangle +(1,1);
		\draw (2,2) rectangle +(1,1);
		\draw (2,3) rectangle +(1,1);
		\node at (.5,.5) {$4$};
		\node at (.5,1.5) {$2$};
		\node at (.5,2.5) {$1$};
		\node at (1.5,1.5) {$4$};
		\node at (1.5,2.5) {$2$};
		\node at (2.5,2.5) {$3$};
		\node at (2.5,3.5) {$2$};
	\end{tikzpicture}$ be a  SSYT. Its biword is $\tilde{w}=\begin{pmatrix}
		1 & 1 & 2 & 2 & 3 & 3 & 3\\
		2 & 3 & 2 & 4 & 1 & 2 & 4
	\end{pmatrix}$, and $$RSK^\ast(\tilde{w})=\left(\text{rect}(\tilde{T}), \tilde{Q}=\YT{0.17in}{}{{1,1,2},{2,3,3},{3}}\right).$$
	
 \noindent  The weight of $\tilde{Q}$, $\mathrm{wt}(\tilde Q)=(2,2,3)$ is the sequence of column lengths of $\tilde{T}$ from right to left.
\end{ex}

	\begin{ex}\label{exrsk2} The bi-word $w=\begin{pmatrix}
		1 & 2 & 2 & 2 & 3 & 3 & 3\\
		2 & 2 & 3 & 4 & 1 & 2 & 4
	\end{pmatrix}$ of $T\in \mathcal{SSYT}((3,3,2), 4)$, in Example \ref{exrsk}, is mapped to the pair $$\left(T=\YT{0.17in}{}{{1,2,2},{2,3},{4,4}},K(rev\,(3,2,2)')= \YT{0.17in}{}{{1,2,2},{2,3,3},{3}}\right)
$$
where  $\lambda'=(3,3,1)$, $rev\,\lambda'=(1,3,3)$.
\end{ex}	
	More generally, given $T\in \mathcal{SSYT}(\lambda, n)$ with $\ell(\lambda')\leq r$, the dual RSK maps the bi-word presentation   of $T$, to the pair $(T, K(rev(\lambda')))$, where $rev(\lambda')$ is the  vector $\lambda'$ written backwards. Note that in this case $Q$ is of shape $\lambda'$ and its weight records  the column lengths of $T$, from right to left. That is, $Q$ is of shape $\lambda'$ and its weight, $wt(Q)=rev\,\lambda'$, henceforth $Q=K(rev\,\lambda')$.
	
	\subsection{Cocrystal of SSYT's}
	Given $T\in \mathcal{SSYT}(\lambda, n)$ with $\ell(\lambda')\leq r$, we define the \emph{cocrystal} of $T$, $\mathfrak{CB}^{\lambda'}(T)$, to be the $\mathfrak{gl}_r$-crystal,
	\begin{equation}
		\label{isomorph}
		\mathfrak{CB}^{\lambda'}(T)= (\text{RSK}^\ast)^{-1}(\{T\}\times \mathcal{SSYT}(\lambda', r))\subseteq E_n^r,
	\end{equation}
	where the  lowering $\mathcal{F}_i$ and raising $\mathcal{E}_i$ crystal operators, are SJDT slides on consecutive columns $i$, $i+1$ on each element of $\mathfrak{CB}^{\lambda'}(T)$, for $i=1, \dots, r-1$, right to left. More precisely, $\mathcal{F}_i$ ($\mathcal{E}_i$), when defined,  sends a cell from the $i$ ($i+1$)-th column to the $i+1$ ($i$)-th column.
The weight map in $\mathfrak{CB}^{\lambda'}(T)$  maps each element $\tilde T$  to the sequence of its  column lengths from right to left, $\rm wt(\tilde T)=\text{ reverse sequence of  column lengths of $\tilde T$} $; and for $X\in \mathfrak{CB}^{\lambda'}(T)$ if  $\mathcal{F}_i(X)=Y$ we also have $\text{wt}(Y)=\text{wt}(X)-\alpha_i$, with $\alpha_i=e_i-e_{i+1}$, $i=1,\dots, r-1$, the simple roots of type $A_{r-1}$.
The cocrystal of $T$ is a highest weight crystal with highest weight $rev\lambda'$, and highest element  the anti-rectification of $T$,  $\textsf{arect}(T)$, that is, the rectification on $T$ is performed south-eastward, \begin{align*}\mathfrak{CB}^{\lambda'}(T)&=\{\mathcal{F}^{m_{i_k}}_{i_k}\cdots \mathcal{F}^{m_{i_1}}_{i_1}(\textsf{arect}(T))\mid (m_{i_k},\dots,m_{i_1})\in \mathbb{Z}_{\ge 0}^k,\; k\ge 0\}\setminus\{0\}\\
&=\{\mathcal{E}^{m_{i_k}}_{i_k}\cdots \mathcal{E}^{m_{i_1}}_{i_1}(T)\mid (m_{i_k},\dots,m_{i_1})\in \mathbb{Z}_{\ge 0}^k,\; k\ge 0\}\setminus\{0\},
\end{align*}
where $T$ is the lowest weight element of $\mathfrak{CB}^{\lambda'}(T)$.
The type $A_{r-1}$ crystals $\mathcal{SSYT}(\lambda', r)$ and $\mathfrak{CB}^{\lambda'}(T)$ are isomorphic { via the map
\begin{align}\label{crystalisomorphism}\mathfrak{CB}^{\lambda'}(T)\rightarrow \mathcal{SSYT}(\lambda', r),\;\tilde T\mapsto Q(\tilde T)\end{align}
\noindent such that $RSK^*(\tilde T)=(T, Q(\tilde T)) $ and $\textrm{wt}(\tilde T)=\textrm{wt}(Q(\tilde T))$.  This crystal isomorphism relies on the following proposition, a consequence of \cite[Lemma 2.3, Lemma 2.4]{kwonheococrystal2020} by Heo-Kwon:
	
	\begin{prop}\label{corheokwon}
		Let $X$ be a skew SSYT.
		The $Q$-symbol of $\mathcal{F}_i(X)$ is the same as $f_i$ applied to the $Q$-symbol of $X$, $Q(\mathcal{F}_i(X))=f_iQ(X)$ (similarly for $\mathcal{E}_i$ and $e_i$), and the weight of the $Q$-symbol of $X$ records the column lengths of $X$ from right to left.
	\end{prop}
	
}	
	%
	%
	%
	%
	
	\begin{ex}\label{ExcocrystalSSYT}Recall $T $ in Example \ref{exrsk2} and $\tilde{T}$ in Example \ref{Ttilde}. Note that $T=\mathcal{F}_1(\tilde{T})$ and that the $Q$-symbols obtained from both tableaux via the map \eqref{crystalisomorphism} are connected via $f_1$, that is, $Q(T)=K(rev(3,2,2)')=f_1Q(\tilde T)=Q(\mathcal{F}_1(\tilde T))$. See Figure \ref{cocrystal}.
		
	 On the right hand side of Figure \ref{cocrystal}, we have the cocrystal $\mathfrak{CB}^{\lambda'}(T)$, whose vertices are obtained by applying the elementary
		SJDT slides $\mathcal{E}_i$, for $i=1, 2$,  on $T$, the lowest weight element of the cocrystal $\mathfrak{CB}^{\lambda'}(T)$. Namely, $\mathcal{E}_1$
		sends an entry from the second column to the first column, and $\mathcal{E}_2$
		sends an entry from the third column to the second column, where we count columns starting from the right. $\mathcal{F}_1$ and $\mathcal{F}_2$ are the inverse operations.
		
		On the left hand side, we have the type $A_2$ crystal $\mathcal{SSYT}((3,3,1),3)$, formed by the $Q$-symbols of every skew tableau that exists in the type $A_2$ crystal $\mathfrak{CB}^{\lambda'}(T)$ on the right. The type $A_2$ crystal operators on the left are defined by the signature rule on the alphabet $[3]$, whereas, on the right, $\mathcal{F}_1$  and $\mathcal{F}_2$ are type $A_2$ crystal operators defined by SJDT on  adjacent columns.
		
\begin{figure}[h]
\begin{center}	
			\begin{tikzpicture} \node at (0,10.5) {Type $A_2$ crystals $\mathcal{SSYT}((3,3,1),3)$};
				\node at (5,10.5) {$\mathfrak{CB}^{\lambda'}(T)$};
				\node at (0,8.8) {${\scriptsize K(\lambda')}=\YT{0.17in}{}{
						{{1},{1},{1}},
						{{2},{2},2},
						{{3}}}$};
				\node at (0,7.1) {$\YT{0.17in}{}{
						{{1},{1},{1}},
						{{2},{2},3},
						{{3}}}$};
				\node at (-1,5.4) {$\YT{0.17in}{}{
						{{1},{1},{2}},
						{{2},{2},3},
						{{3}}}$};
				\node at (1,5.4) {$\YT{0.17in}{}{
						{{1},{1},{1}},
						{{2},{3},3},
						{{3}}}$};
				\node at (-0.4,3.7) {$\tilde{Q}=\YT{0.17in}{}{
						{{1},{1},{2}},
						{{2},{3},3},
						{{3}}}$};
				\node at (0,2) {$K(rev\lambda')=\YT{0.17in}{}{
						{{1},{2},{2}},
						{{2},{3},3},
						{{3}}}$};
				\node[draw=none,fill=none] at (-3, 7){${\color{blue}\rightarrow}f_1$};
				\node[draw=none,fill=none] at (-3, 6.5){${\color{red}\rightarrow} f_2$};
				\node[draw=none,fill=none] at (8, 7.5){${\color{blue}\rightarrow}\mathcal{F}_1$};
				\node[draw=none,fill=none] at (8, 6.5){${\color{red}\rightarrow} \mathcal{F}_2$};
				\draw[->] [draw=red,  thick] (0, 8.4) -- (0,8);
				\draw[->] [draw=blue,  thick] (0, 3.2) -- (0,2.8);
				\draw[->] [draw=blue,  thick] (-.6, 6.4) -- (-.9,6.1);
				\draw[->] [draw=red,  thick] (.6, 6.4) -- (.9,6.1);
				\draw[->] [draw=red,  thick] (-.8, 4.9) -- (-.5,4.5);
				\draw[->] [draw=blue,  thick] (1.1, 4.9) -- (.8,4.5);
				
				\node at (4,1.6) {$T=\YT{0.17in}{}{
						{{1},{2},{2}},
						{{2},{3}},
						{{4},4}}$};
				\node at (3.6,3.4) {$\tilde{T}=\begin{tikzpicture}[scale=.444, baseline={([yshift=-.8ex]current bounding box.center)}]
						\draw (0,0) rectangle +(1,1);
						\draw (0,1) rectangle +(1,1);
						\draw (0,2) rectangle +(1,1);
						\draw (1,1) rectangle +(1,1);
						\draw (1,2) rectangle +(1,1);
						\draw (2,2) rectangle +(1,1);
						\draw (2,3) rectangle +(1,1);
						\node at (.5,.5) {$4$};
						\node at (.5,1.5) {$2$};
						\node at (.5,2.5) {$1$};
						\node at (1.5,1.5) {$4$};
						\node at (1.5,2.5) {$2$};
						\node at (2.5,2.5) {$3$};
						\node at (2.5,3.5) {$2$};
					\end{tikzpicture}$};
				\node at (2.8,4.9) {$\begin{tikzpicture}[scale=.444, baseline={([yshift=-.8ex]current bounding box.center)}]
						\draw (0,0) rectangle +(1,1);
						\draw (0,1) rectangle +(1,1);
						\draw (1,0) rectangle +(1,1);
						\draw (1,1) rectangle +(1,1);
						\draw (1,2) rectangle +(1,1);
						\draw (2,1) rectangle +(1,1);
						\draw (2,2) rectangle +(1,1);
						\node at (.5,.5) {$4$};
						\node at (.5,1.5) {$2$};
						\node at (1.5,0.5) {$4$};
						\node at (1.5,1.5) {$2$};
						\node at (1.5,2.5) {$1$};
						\node at (2.5,1.5) {$3$};
						\node at (2.5,2.5) {$2$};
					\end{tikzpicture}$};
				\node at (5.6,5.52) {$$\begin{tikzpicture}[scale=.444, baseline={([yshift=-.8ex]current bounding box.center)}]
						\draw (0,0) rectangle +(1,1);
						\draw (0,1) rectangle +(1,1);
						\draw (0,2) rectangle +(1,1);
						\draw (1,2) rectangle +(1,1);
						\draw (2,2) rectangle +(1,1);
						\draw (2,3) rectangle +(1,1);
						\draw (2,4) rectangle +(1,1);
						\node at (.5,.5) {$4$};
						\node at (.5,1.5) {$2$};
						\node at (.5,2.5) {$1$};
						\node at (1.5,2.5) {$2$};	
						\node at (2.5,2.5) {$4$};
						\node at (2.5,3.5) {$3$};
						\node at (2.5,4.5) {$2$};
					\end{tikzpicture}$$};
				\node at (4,7.05) {$$\begin{tikzpicture}[scale=.444, baseline={([yshift=-.8ex]current bounding box.center)}]
						\draw (0,0) rectangle +(1,1);
						\draw (0,1) rectangle +(1,1);
						\draw (1,1) rectangle +(1,1);
						\draw (1,2) rectangle +(1,1);
						\draw (2,1) rectangle +(1,1);
						\draw (2,2) rectangle +(1,1);
						\draw (2,3) rectangle +(1,1);
						\node at (.5,.5) {$4$};
						\node at (.5,1.5) {$2$};
						\node at (1.5,1.5) {$2$};
						\node at (1.5,2.5) {$1$};	
						\node at (2.5,1.5) {$4$};
						\node at (2.5,2.5) {$3$};
						\node at (2.5,3.5) {$2$};
					\end{tikzpicture}$$};
				\node at (4,9) {$\SYT{0.17in}{}{
						{{2},1},
						{{3},{2}},
						{{4},4,2}}$};
				\draw[->] [draw=red,  thick] (4, 8) -- (4,7.6);
				\draw[->] [draw=blue,  thick] (4, 2.8) -- (4,2.4);
				\draw[->] [draw=blue,  thick] (3.5, 6.) -- (3.2,5.7);
				\draw[->] [draw=red,  thick] (4.5, 6) -- (4.8,5.7);
				\draw[->] [draw=red,  thick] (3.2, 4.5) -- (3.5,4.1);
				\draw[->] [draw=blue,  thick] (4.8, 4.8) -- (4.5,4.4);	
			\end{tikzpicture}
		\end{center}	
\caption{The type $A_2$ crystal operators $f_1$ and $f_2$ are given by the signature rule on the alphabet $[3]$, whereas $\mathcal{F}_1$ and $\mathcal{F}_2$, even though they are also type $A_2$ crystal operators, are defined by elementary SJDT moves.\label{cocrystal}}
	\end{figure}
	\end{ex}

	\subsection{Cocrystal of KN tableaux}
	Let $T\in \mathcal{SSYT}(\lambda, n)$. Note that $T$, the lowest weight of the cocrystal $\mathfrak{CB}^{\lambda'}(T)$, is also in the type $C_n$ crystal $\textsf{KN}(\lambda)$ (recall that $\mathcal{SSYT}(\lambda, n)$ is a subcrystal of $\textsf{KN}(\lambda)$). Fixed an arbitrary tableau $Y$ in the crystal  $\textsf{KN}(\lambda)$, there is a sequence $\mathcal{S}$ of type $C_n$ crystal operators, on $\textsf{KN}(\lambda)$, such that $\mathcal{S}(T)=Y$.  All elements of the cocrystal $\mathfrak{CB}^{\lambda'}(T)$ are SJDT related and we can apply this sequence $\mathcal{S}$ to all skew tableaux in the cocrystal, obtaining, for each skew tableau, a new skew tableau of the same shape. All these skew tableaux, obtained by application of the sequence $\mathcal{S}$ to each element of $\mathfrak{CB}^{\lambda'}(T)$, will be connected via SJDT, because the SJDT and the crystal operators of $\textsf{KN}(\lambda)$ commute \cite[Theorem 6.3.8]{lecouvey2002schensted}, hence they are the elements of a new cocrystal $\mathfrak{CB}^{\lambda'}(\mathcal{S}(T))$ of type $A_{r-1}$, despite  that its vertices are type $C_n$ objects (i.e. KN skew tableaux). Recalling that the weight function of $\mathfrak{CB}^{\lambda'}(T)$ is given by the column lengths of each vertex, from right to left, which is preserved by any sequence $\mathcal{S}$ of crystal operators given by the $C_n$ signature rule in $\textsf{KN}(\lambda)$, the following is a consequence of Proposition \ref{corheokwon}.
	
	\begin{prop}
		Given $T\in \mathcal{KN}(\lambda, n)$, with $\ell(\lambda')\leq r$, the cocrystal $\mathfrak{CB}^{\lambda'}(T)$ with lowest weight element $T$, obtained from $T$ by successive application of elementary SJDT moves, is crystal isomorphic to the $\mathfrak{gl}_r$-crystal $\mathcal{SSYT}(\lambda', r)$.
	\end{prop}
	{ As a consequence, interestingly, these elementary SJDT moves $\mathcal{E}_i$ and $\mathcal{F}_i$ in the cocrystal of a KN tableau do not incur in a loss of boxes, that is, the B2 case in the symplectic jeu de taquin never occurs.
We have thus for $X\in \mathfrak{CB}^{\lambda'}(T)$, $ \mathcal{F}_i(\mathcal{S}(X))=\mathcal{S}(\mathcal{F}_i(X))$ (similarly for $\mathcal{E}_i$),  and therefore $\mathfrak{CB}^{\lambda'}(T)$ and $\mathfrak{CB}^{\lambda'}(\mathcal{S}(T))$ are $A_{r-1}$ isomorphic crystals.  See Figure \ref{Scocrystal}.}
	
	Fulton \cite{fulton1997young} has proved the following result for semistandard tableaux.
	
	\begin{prop} \label{fultonkey}\cite[Proposition 7, Corollary 1, Appendix A.5]{fulton1997young}
		Given $T \in \mathcal{SSYT}(\lambda, n)$ and a  skew shape whose column lengths are a permutation of $\lambda'$, the column lengths of $T$, there is exactly one skew tableau with that shape that rectifies to $T$. Furthermore, the last and first columns only depend on their lengths.
	\end{prop}

	This means that given $T \in \mathcal{SSYT}(\lambda, n)$, the cocrystal $\mathfrak{CB}^{\lambda'}(T)$ attached to $T \in \mathcal{SSYT}(\lambda, n)$  has a distinguished set of skew tableaux whose column lengths are a permutation of $\lambda'$, the column lengths of $T$. The skew shapes of these distinguished vertices are preserved by any sequence $\mathcal{S}$ of type $C_n$ crystal operators of the crystal $\mathfrak{B}^\lambda$. Thus we  obtain another proof of  Proposition 40 and Corollary 41 in second author's work \cite{santos2019symplectic} which is an extension
	of the previous Proposition \ref{fultonkey} to KN tableaux.
	\begin{prop} \label{sahpesjdt}\cite[Proposition 40, Corollary 41]{santos2019symplectic}, \cite{santos2020SLCsymplectic}
		Given $T \in \mathcal{KN}(\lambda, n)$ and a  skew shape whose column lengths are a permutation of $\lambda'$, the column lengths of $T$,  there is exactly one skew tableau with that shape that rectifies to $T$. Furthermore, the last and first columns only depend on their lengths.
	\end{prop}

	A key tableau in the type $A_{r-1}$ crystal $\mathcal{SSYT}(\lambda', r)$  is a tableau of shape $\lambda'$ whose weight is in $\mathfrak{S}_r\lambda'$, the $\mathfrak{S}_r$-orbit of $\lambda'$. For each element of $\mathfrak{S}_r\lambda'$ there is
	exactly one key tableau of shape $\lambda'$ with that weight. More precisely the key tableaux in  $\mathcal{SSYT}(\lambda', r)$ are the distinguished vertices which define the set $\mathfrak{S}_r K(\lambda')$ where  $s_iK(\lambda')=K(s_i\lambda')$ and $s_i$ are the simple transpositions of $\mathfrak{S}_r$,   for $i=1,\dots, r-1$. Thereby it is natural to define keys in a cocrystal (a type A crystal).

	\begin{defin} \label{Defcokey} Given $T\in \mathcal{KN}(\lambda, n)$, with  $\ell(\lambda')\leq r$, and $X\in  \mathfrak{CB}^{\lambda'}(T)$, the skew tableau $X$ is said to be a key  of $\mathfrak{CB}^{\lambda'}(T)$ if its  weight as an element of the said cocrystal, the sequence column lengths of $X$, from right to left,  is a permutation of $rev\,\lambda'$, the weight of $T$ as an element of the same cocrystal.
	\end{defin}
	
	In other words, given $T\in \mathcal{KN}(\lambda, n)$, the keys of  $\mathfrak{CB}^{\lambda'}(T)$ are the image of the keys in  $\mathcal{SSYT}(\lambda', r)$ via the crystal isomorphism \eqref{crystalisomorphism}. We then have an action of the left cosets $\mathfrak{S}_r/\mathfrak{S}_{r\lambda'}$ on the set of the keys in $\mathfrak{CB}^{\lambda'}(T)$. In Figure \ref{Scocrystal} the  weak Bruhat order on the  cosets $\mathfrak{S}_3/\mathfrak{S}_{3(3,3,1)}$ can be identified with the  weak Bruhat order  on the keys of the cocrystal $\mathfrak{CB}^{\lambda'}(T)$
\begin{align}\label{bruhatcocrystal}\mathcal{S}(T)=\YT{0.17in}{}{
						{{1},{2},{2}},
						{{2},{3}},
						{{4},\overline{4}}}
\underset{\mathcal{E}^2_1}\longrightarrow \begin{tikzpicture}[scale=.444, baseline={([yshift=-.8ex]current bounding box.center)}]
						\draw (0,0) rectangle +(1,1);
						\draw (0,1) rectangle +(1,1);
						\draw (0,2) rectangle +(1,1);
						\draw (1,2) rectangle +(1,1);
						\draw (2,2) rectangle +(1,1);
						\draw (2,3) rectangle +(1,1);
						\draw (2,4) rectangle +(1,1);
						\node at (.5,.5) {$4$};
						\node at (.5,1.5) {$2$};
						\node at (.5,2.5) {$1$};
						\node at (1.5,2.5) {$2$};	
						\node at (2.5,2.5) {$\overline 4$};
						\node at (2.5,3.5) {$3$};
						\node at (2.5,4.5) {$2$};
					\end{tikzpicture}
\underset{\mathcal{E}^2_2}\longrightarrow
\SYT{0.17in}{}{
						{{2},1},
						{{3},{2}},
						{{\overline{4}},4,2}}.
\end{align}
	
	\begin{ex}Recall the right hand side crystal in  Figure \ref{cocrystal}. 
$T$ is in the type $C_4$ crystal $\mathcal{KN}((3,2,2,0),4)$. Hence we can apply to each vertex of $\mathfrak{CB}^{\lambda'}(T)$ the sequence of crystal operators $S=f_4$, obtaining a new cocrystal, on the right, whose vertices are KN skew tableaux connected via SJDT. This cocrystal $\mathfrak{CB}^{\lambda'}(f_4(T))$ is a type $A_2$ crystal. See Figure \ref{Scocrystal}.
		
\begin{figure}[h]
\begin{center}	
			\begin{tikzpicture}
				\node at (0,10.5) {$\mathfrak{CB}^{\lambda'}(T)$};
				\node at (5,10.5) {$\mathfrak{CB}^{\lambda'}(f_4(T))=f_4\mathfrak{CB}^{\lambda'}(T)$};
				\node at (2,1.6){$\overset{f_4}\longrightarrow$};
				\node at (2,3.3){$\overset{f_4}\longrightarrow$};
				\node at (-2.75,1.6){$K(\lambda)\rightarrow\cdots\rightarrow$};
				\node at (-3.1,3.3){$\mathcal{E}_1K(\lambda)\rightarrow\cdots\rightarrow$};
				\node at (8,1.6){$\rightarrow\cdots\rightarrow K(-\lambda)$};
				\node at (8,3.3){$\rightarrow\cdots\rightarrow \mathcal{E}_1K(-\lambda)$};
				
				\node at (0,1.6) {$T=\YT{0.17in}{}{
						{{1},{2},{2}},
						{{2},{3}},
						{{4},4}}$};
				\node at (-.4,3.4) {${\tilde T}=\begin{tikzpicture}[scale=.444, baseline={([yshift=-.8ex]current bounding box.center)}]
						\draw (0,0) rectangle +(1,1);
						\draw (0,1) rectangle +(1,1);
						\draw (0,2) rectangle +(1,1);
						\draw (1,1) rectangle +(1,1);
						\draw (1,2) rectangle +(1,1);
						\draw (2,2) rectangle +(1,1);
						\draw (2,3) rectangle +(1,1);
						\node at (.5,.5) {$4$};
						\node at (.5,1.5) {$2$};
						\node at (.5,2.5) {$1$};
						\node at (1.5,1.5) {$4$};
						\node at (1.5,2.5) {$2$};
						\node at (2.5,2.5) {$3$};
						\node at (2.5,3.5) {$2$};
					\end{tikzpicture}$};
				\node at (-1.2,4.9) {$\begin{tikzpicture}[scale=.444, baseline={([yshift=-.8ex]current bounding box.center)}]
						\draw (0,0) rectangle +(1,1);
						\draw (0,1) rectangle +(1,1);
						\draw (1,0) rectangle +(1,1);
						\draw (1,1) rectangle +(1,1);
						\draw (1,2) rectangle +(1,1);
						\draw (2,1) rectangle +(1,1);
						\draw (2,2) rectangle +(1,1);
						\node at (.5,.5) {$4$};
						\node at (.5,1.5) {$2$};
						\node at (1.5,0.5) {$4$};
						\node at (1.5,1.5) {$2$};
						\node at (1.5,2.5) {$1$};
						\node at (2.5,1.5) {$3$};
						\node at (2.5,2.5) {$2$};
					\end{tikzpicture}$};
				\node at (1.6,5.52) {$$\begin{tikzpicture}[scale=.444, baseline={([yshift=-.8ex]current bounding box.center)}]
						\draw (0,0) rectangle +(1,1);
						\draw (0,1) rectangle +(1,1);
						\draw (0,2) rectangle +(1,1);
						\draw (1,2) rectangle +(1,1);
						\draw (2,2) rectangle +(1,1);
						\draw (2,3) rectangle +(1,1);
						\draw (2,4) rectangle +(1,1);
						\node at (.5,.5) {$4$};
						\node at (.5,1.5) {$2$};
						\node at (.5,2.5) {$1$};
						\node at (1.5,2.5) {$2$};	
						\node at (2.5,2.5) {$4$};
						\node at (2.5,3.5) {$3$};
						\node at (2.5,4.5) {$2$};
					\end{tikzpicture}$$};
				\node at (0,7.05) {$$\begin{tikzpicture}[scale=.444, baseline={([yshift=-.8ex]current bounding box.center)}]
						\draw (0,0) rectangle +(1,1);
						\draw (0,1) rectangle +(1,1);
						\draw (1,1) rectangle +(1,1);
						\draw (1,2) rectangle +(1,1);
						\draw (2,1) rectangle +(1,1);
						\draw (2,2) rectangle +(1,1);
						\draw (2,3) rectangle +(1,1);
						\node at (.5,.5) {$4$};
						\node at (.5,1.5) {$2$};
						\node at (1.5,1.5) {$2$};
						\node at (1.5,2.5) {$1$};	
						\node at (2.5,1.5) {$4$};
						\node at (2.5,2.5) {$3$};
						\node at (2.5,3.5) {$2$};
					\end{tikzpicture}$$};
				\node at (0,9) {$\SYT{0.17in}{}{
						{{2},1},
						{{3},{2}},
						{{4},4,2}}$};
				\draw[->] [draw=red,  thick] (0, 8) -- (0,7.6);
				\draw[->] [draw=blue,  thick] (0, 2.8) -- (0,2.4);
				\draw[->] [draw=blue,  thick] (-.5, 6.) -- (-.8,5.7);
				\draw[->] [draw=red,  thick] (.5, 6) -- (.8,5.7);
				\draw[->] [draw=red,  thick] (-.8, 4.5) -- (-.5,4.1);
				\draw[->] [draw=blue,  thick] (.8, 4.8) -- (.5,4.4);	
				\node at (5,1.6) {$f_4(T)=\YT{0.17in}{}{
						{{1},{2},{2}},
						{{2},{3}},
						{{4},\overline{4}}}$};
				\node at (4.2,3.4) {${ f_4(\tilde T)}=\begin{tikzpicture}[scale=.444, baseline={([yshift=-.8ex]current bounding box.center)}]	
						\draw (0,0) rectangle +(1,1);
						\draw (0,1) rectangle +(1,1);
						\draw (0,2) rectangle +(1,1);
						\draw (1,1) rectangle +(1,1);
						\draw (1,2) rectangle +(1,1);
						\draw (2,2) rectangle +(1,1);
						\draw (2,3) rectangle +(1,1);
						\node at (.5,.5) {$4$};
						\node at (.5,1.5) {$2$};
						\node at (.5,2.5) {$1$};
						\node at (1.5,1.5) {$\overline{4}$};
						\node at (1.5,2.5) {$2$};
						\node at (2.5,2.5) {$3$};
						\node at (2.5,3.5) {$2$};
					\end{tikzpicture}$};
				\node at (3.8,4.9) {$\begin{tikzpicture}[scale=.444, baseline={([yshift=-.8ex]current bounding box.center)}]
						\draw (0,0) rectangle +(1,1);
						\draw (0,1) rectangle +(1,1);
						\draw (1,0) rectangle +(1,1);
						\draw (1,1) rectangle +(1,1);
						\draw (1,2) rectangle +(1,1);
						\draw (2,1) rectangle +(1,1);
						\draw (2,2) rectangle +(1,1);
						\node at (.5,.5) {$4$};
						\node at (.5,1.5) {$2$};
						\node at (1.5,0.5) {$\overline{4}$};
						\node at (1.5,1.5) {$2$};
						\node at (1.5,2.5) {$1$};
						\node at (2.5,1.5) {$3$};
						\node at (2.5,2.5) {$2$};
					\end{tikzpicture}$};
				\node at (6.6,5.52) {$$\begin{tikzpicture}[scale=.444, baseline={([yshift=-.8ex]current bounding box.center)}]
						\draw (0,0) rectangle +(1,1);
						\draw (0,1) rectangle +(1,1);
						\draw (0,2) rectangle +(1,1);
						\draw (1,2) rectangle +(1,1);
						\draw (2,2) rectangle +(1,1);
						\draw (2,3) rectangle +(1,1);
						\draw (2,4) rectangle +(1,1);
						\node at (.5,.5) {$4$};
						\node at (.5,1.5) {$2$};
						\node at (.5,2.5) {$1$};
						\node at (1.5,2.5) {$2$};	
						\node at (2.5,2.5) {$\overline{4}$};
						\node at (2.5,3.5) {$3$};
						\node at (2.5,4.5) {$2$};
					\end{tikzpicture}$$};
				\node at (5,7.05) {$$\begin{tikzpicture}[scale=.444, baseline={([yshift=-.8ex]current bounding box.center)}]
						\draw (0,0) rectangle +(1,1);
						\draw (0,1) rectangle +(1,1);
						\draw (1,1) rectangle +(1,1);
						\draw (1,2) rectangle +(1,1);
						\draw (2,1) rectangle +(1,1);
						\draw (2,2) rectangle +(1,1);
						\draw (2,3) rectangle +(1,1);
						\node at (.5,.5) {$4$};
						\node at (.5,1.5) {$2$};
						\node at (1.5,1.5) {$2$};
						\node at (1.5,2.5) {$1$};	
						\node at (2.5,1.5) {$\overline{4}$};
						\node at (2.5,2.5) {$3$};
						\node at (2.5,3.5) {$2$};
					\end{tikzpicture}$$};
				\node at (5,9) {$\SYT{0.17in}{}{
						{{2},1},
						{{3},{2}},
						{{\overline{4}},4,2}}$};
				\node[draw=none,fill=none] at (-2.5, 7){${\color{blue}\rightarrow}\mathcal{F}_1$};
				\node[draw=none,fill=none] at (-2.5, 6.5){${\color{red}\rightarrow} \mathcal{F}_2$};
				\node[draw=none,fill=none] at (9, 7.5){${\color{blue}\rightarrow}\mathcal{F}_1$};
				\node[draw=none,fill=none] at (9, 6.5){${\color{red}\rightarrow} \mathcal{F}_2$};
				\draw[->] [draw=red,  thick] (5, 8) -- (5,7.6);
				\draw[->] [draw=blue,  thick] (5, 2.8) -- (5,2.4);
				\draw[->] [draw=blue,  thick] (4.5, 6.) -- (4.2,5.7);
				\draw[->] [draw=red,  thick] (5.5, 6) -- (5.8,5.7);
				\draw[->] [draw=red,  thick] (4.2, 4.5) -- (4.5,4.1);
				\draw[->] [draw=blue,  thick] (5.8, 4.8) -- (5.5,4.4);	
				
			\end{tikzpicture}	
		\end{center}
\caption{\label{Scocrystal}The vertical array of isomorphic cocrystals attached to $T$ and $f_4(T)$ in $ \mathcal{KN}((3,3,2),4)$ respectively. The (horizontal) array of (isomorphic)  $C_4$ symplectic crystals generated by elementary SJDT moves on each vertex of $\mathcal{KN}((3,3,2),4)$.}
\end{figure}
		
		The KN tableaux $T$ and $f_4(T)$ are contained in the type $C_4$ crystal $ \mathcal{KN}((3,3,2),4)$ with highest weight element $K(\lambda)$ and lowest weight element $K(-\lambda)$. The KN skew tableaux in a same horizontal position of the cocrystals define a type $C_4$ crystal isomorphic  to the crystal $\mathcal{KN}((3,3,2),4)$. In fact, their highest weight elements are the Littlewood-Richardson tableaux \cite{fulton1997young} of weight $\lambda$, defining the cocrystal attached to $K(\lambda)$, the Yamanouchi tableau of weight and shape $\lambda$. For instance, the type $C_4$ crystal containing $\tilde T$  and $f_4(\tilde T)$ has highest weight element the Littlewood-Richardson tableau $\mathcal{E}_1(K(\lambda))= \begin{tikzpicture}[scale=.444, baseline={([yshift=-.8ex]current bounding box.center)}]
			\draw (0,0) rectangle +(1,1);
			\draw (0,1) rectangle +(1,1);
			\draw (0,2) rectangle +(1,1);
			\draw (1,1) rectangle +(1,1);
			\draw (1,2) rectangle +(1,1);
			\draw (2,2) rectangle +(1,1);
			\draw (2,3) rectangle +(1,1);
			\node at (.5,.5) {$3$};
			\node at (.5,1.5) {$2$};
			\node at (.5,2.5) {$1$};
			\node at (1.5,1.5) {$3$};
			\node at (1.5,2.5) {$1$};
			\node at (2.5,2.5) {$2$};
			\node at (2.5,3.5) {$1$};
		\end{tikzpicture} $

\noindent and lowest weight element its reversal (in the sense of Lusztig involution \cite{santos2019symplectic}) $\mathcal{E}_1(K(-\lambda))= \begin{tikzpicture}[scale=.444, baseline={([yshift=-.8ex]current bounding box.center)}]
			\draw (0,0) rectangle +(1,1);
			\draw (0,1) rectangle +(1,1);
			\draw (0,2) rectangle +(1,1);
			\draw (1,1) rectangle +(1,1);
			\draw (1,2) rectangle +(1,1);
			\draw (2,2) rectangle +(1,1);
			\draw (2,3) rectangle +(1,1);
			\node at (.5,.5) {$\overline{1}$};
			\node at (.5,1.5) {$\overline{2}$};
			\node at (.5,2.5) {$\overline{3}$};
			\node at (1.5,1.5) {$\overline{1}$};
			\node at (1.5,2.5) {$\overline{2}$};
			\node at (2.5,2.5) {$\overline{1}$};
			\node at (2.5,3.5) {$\overline{3}$};
		\end{tikzpicture}$.
		
	\end{ex}

	\subsection{ The Sch\"utzenberger-Lusztig involution on a cocrystal}

	The next lemma will help us relate the cocrystal $\mathfrak{CB}^{\lambda'}(T)$ and the cocrystal $\mathfrak{CB}^{\lambda'}(T^{Ev})$.

Let $\#$ be the map that given a KN tableau $T$, $\pi$-rotates $T$ and replaces each entry with its  symmetric. Recall that when $T\in \mathcal{KN}(\lambda,n)$, \cite[Algorithm 59]{santos2019symplectic},
\begin{align}evac (T)=rect(T^\#)=(arect (T))^\# \label{evac}.\end{align}
	\begin{lema}
		Let $T$ be a KN tableau with two columns $T_1$, $T_2$, left to right. Let $S$ be a skew tableau, with two columns $S_1$, $S_2$, obtained from $T$ after $\pi$-rotating $T$ and replacing  each entry   by its symmetric.
		Let us play the SJDT in both $T$ and $S$, in $T$ we are sending a cell from $T_1$ to $T_2$ and in $S$ we are sending a cell from $S_2$ to $S_1$, obtaining $T'$ and $S'$, respectively. Then $S'$ is obtained from $T'$ by $\pi$-rotating $T'$ and replacing each entry by its symmetric.
	\end{lema}
	\begin{proof}
		Let $T'_1$, $T'_2$ be the columns of $T'$ and $S'_1$, $S'_2$ be the columns of $S'$.
		Let $\#$ be the map that given $T$, $\pi$-rotates $T$ and replaces each entry with its  symmetric. Then $S=T^\#$, $S_1=T_2^\#$, $S_2=T_1^\#$, and $spl(S)=spl(T)^\#$.

		Note that the entry $\alpha$ that gets slid from $T_1$ to $T_2$  is symmetric to the entry that gets slid from $S_2$ to $S_1$. So we have two cases:
		\begin{itemize}
			\item If $\alpha$ is positive: $T_1'=T_1\setminus \{\alpha\}$,  $T_2'=\Phi^{-1}(\Phi(T_2)\cup \{\alpha\})$, $S_2'=S_2\setminus \{\overline{\alpha}\}$,  $T_1'=\Phi^{-1}(\Phi(T_1)\cup \{\overline{\alpha}\})$. So, $S_1'=T_2'^\#$ and $S_2'=T_1'^\ast$. Hence $S'=T'^\#$.
			
			\item If $\alpha$ is negative:$T_1'=\Phi^{-1}(\Phi(T_1)\setminus \{\alpha\})$, $T_2'=T_2\cup \{\alpha\}$,  $S_2'=\Phi^{-1}(\Phi(S_2)\setminus \{\overline{\alpha}\})$, $S_1'=S_1\cup \{\overline{\alpha}\}$. So, $S_1'=T_2'^\#$ and $S_2'=T_1'^\#$. Hence $S'=T'^\#$.
		\end{itemize}
	\end{proof}
	
	\begin{prop}
		Let $T\in \mathcal{KN}(\lambda,n)$. If we  $\pi$-rotate $\mathfrak{CB}^{\lambda'}(T)$, replace each entry by its symmetric, and for every $i \in [n]$, recolour all arrows of colour $i$ with  the colour $r-i$ and reverse them, we obtain  $\mathfrak{CB}^{\lambda'}(T^{Ev})$.
	\end{prop}
	\begin{proof}
		Note that, after doing all the steps in the statement, all the arrows stay intact. Recall the map $^\#$ from the previous proof. Schützenberger evacuation, $evac\, T=rect(T^\#)$,  implies that $T^\#$ is the highest weight element of $\mathfrak{CB}^{\lambda'}(T^{Ev})$.
		
		Using the previous lemma, we have that, for all $i\in [n]$, $\mathcal{E}_i(T)=\mathcal{F}_{r-i}(T^\#)$. Hence we can prove, recursively, that for every skew tableau $S\in \mathfrak{CB}^{\lambda'}(T)$, we have $S^\#\in \mathfrak{CB}^{\lambda'}(T^{Ev})$ and the position of $S^\#$ in $\mathfrak{CB}^{\lambda'}(T^{Ev})$ can be found if we $\pi$-rotate the whole cocrystal.
	\end{proof}
	\begin{obs}
		The last proposition implies that all first of the skew tableaux in $\mathfrak{CB}^{\lambda'}(T)$ are symmetric to all last of the skew tableaux in $\mathfrak{CB}^{\lambda'}(T^{Ev})$. Hence we have that $\text{wt}(K_+(T))=-\text{wt}(K_-(T^{Ev}))$, and the key tableaux are uniquely described by their weight, we can say that $K_+(T)=K_-(T^{Ev})^{Ev}$. Analogously, we have that $K_-(T)=K_+(T^{Ev})^{Ev}$. This is also proven in \cite[Proposition $64$]{santos2019symplectic}.
	\end{obs}
	
	\section{The right and left key of a KN tableau - \emph{Jeu de taquin} approach and cocrystal}\label{rightkeyjdt}

\subsection{Right and left key maps and the keys of a cocrystal}
	Given $T\in \mathcal{KN}{(\lambda,n)}$, to determine the content of a given column of its right key, $K_+(T)$, we need to compute the right column of a last column, with the same length, of a skew tableau in the cocrystal $\mathfrak{CB}^{\lambda'}(T)$. Analogously, to determine the content of a given column of its left key, $K_-(T)$, we need to compute the left column of a first column, with the same length, of a skew tableau in the cocrystal $\mathfrak{CB}^{\lambda'}(T)$ \cite{santos2019symplectic, santos2020SLCsymplectic}.  Proposition \ref{sahpesjdt} ensures that such computation of the right, or left, key of a KN tableau via SJDT is well-defined. The following is a reformulation of \cite[Theorem 43]{santos2020SLCsymplectic} using the concept of cocrystal.
	
	\begin{teo} \cite{santos2019symplectic},\cite[Theorem 43]{santos2020SLCsymplectic}
		Given a KN tableau $T$, we can replace each column with a column of the same size taken from the right columns of the last columns of all keys in the cocrystal of $T$. This tableau is the right key tableau of $T$, $K_+(T)$. If we replace each column of $T$ with a column of the same size taken from the left columns of all keys of the cocrystal of $T$, we obtain the left key of $T$, $K_-(T)$.
	\end{teo}
The cocrystal keys in \eqref{bruhatcocrystal} give the right and left keys of $\mathcal{S}(T)$, $K^+(S(T))=\YT{0.17in}{}{
						{{2},{2},{2}},
						{{3},{3}},
						{{\overline 4},\overline{4}}}$ and $K_ -(S(T))=\YT{0.17in}{}{
						{{1},{1},{2}},
						{{2},{2}},
						{{ 4},{4}}}$.
	
	Hence, the cocrystal $\mathfrak{CB}^{\lambda'}(T)$ contains all the needed information to compute right and left keys. This is explored again in the next section.

	Given $T\in \mathcal{KN}(\lambda, n)$, we apply the SJDT on consecutive columns to compute the keys of $\mathfrak{CB}^{\lambda'}(T)$, and, henceforth, all skew tableaux in the conditions of Proposition \ref{sahpesjdt}.
	
	\begin{ex}\label{jeudetaquin3}
		
		The tableau $T=\YT{0.17in}{}{
			{{1},{3},{\overline{1}}},
			{{3},{\overline{3}}},
			{{\overline{3}}}}$ gives rise to the cocrystal $\mathfrak{CB}^{\lambda'}(T)$, with $\lambda=(3,2,1)$. The following are the vertices of $\mathfrak{CB}^{\lambda'}(T)$ consisting of the six KN skew tableaux with the same number of columns of each length as $T$, each one corresponding to a permutation of its column lengths.
		
		\begin{tikzpicture}
			\node at (-0.5,0) {$T=\YT{0.17in}{}{
					{{1},{3},{\overline{1}}},
					{{3},{\overline{3}}},
					{{\overline{3}}}}$};	
			\node at (3.5,1) {$\begin{tikzpicture}[scale=.4, baseline={([yshift=-.8ex]current bounding box.center)}]
					\draw (0,0) rectangle +(1,1);
					\draw (0,1) rectangle +(1,1);
					\draw (0,2) rectangle +(1,1);
					\draw (1,2) rectangle +(1,1);
					\draw (2,2) rectangle +(1,1);
					\draw (2,3) rectangle +(1,1);
					\node at (.5,.5) {$\overline{3}$};
					\node at (.5,1.5) {$3$};
					\node at (.5,2.5) {$1$};
					\node at (1.5,2.5) {$\overline{3}$};
					\node at (2.5,2.5) {$\overline{1}$};
					\node at (2.5,3.5) {$3$};
				\end{tikzpicture}$};
			\node at (7,1) {$\begin{tikzpicture}[scale=.4, baseline={([yshift=-.8ex]current bounding box.center)}]
					\draw (0,0) rectangle +(1,1);
					\draw (1,0) rectangle +(1,1);
					\draw (1,1) rectangle +(1,1);
					\draw (1,2) rectangle +(1,1);
					\draw (2,1) rectangle +(1,1);
					\draw (2,2) rectangle +(1,1);
					\node at (.5,.5) {$2$};
					\node at (1.5,.5) {$\overline{2}$};
					\node at (1.5,1.5) {$\overline{3}$};
					\node at (1.5,2.5) {$1$};
					\node at (2.5,1.5) {$\overline{1}$};
					\node at (2.5,2.5) {$3$};
				\end{tikzpicture}$};
			\node at (3.5,-1) {$\begin{tikzpicture}[scale=.4, baseline={([yshift=-.8ex]current bounding box.center)}]
					\draw (0,0) rectangle +(1,1);
					\draw (0,1) rectangle +(1,1);
					\draw (1,0) rectangle +(1,1);
					\draw (1,1) rectangle +(1,1);
					\draw (1,2) rectangle +(1,1);
					\draw (2,2) rectangle +(1,1);
					\node at (.5,.5) {$2$};
					\node at (.5,1.5) {$1$};
					\node at (1.5,.5) {$\overline{2}$};
					\node at (1.5,1.5) {$\overline{3}$};
					\node at (1.5,2.5) {$3$};
					\node at (2.5,2.5) {$\overline{1}$};
				\end{tikzpicture}$};
			\node at (7,-1) {$$\begin{tikzpicture}[scale=.4, baseline={([yshift=-.8ex]current bounding box.center)}]
					\draw (0,0) rectangle +(1,1);
					\draw (0,1) rectangle +(1,1);
					\draw (1,1) rectangle +(1,1);
					\draw (2,1) rectangle +(1,1);
					\draw (2,2) rectangle +(1,1);
					\draw (2,3) rectangle +(1,1);
					\node at (.5,.5) {$2$};
					\node at (.5,1.5) {$1$};
					\node at (1.5,1.5) {$\overline{2}$};
					\node at (2.5,1.5) {$\overline{1}$};
					\node at (2.5,2.5) {$\overline{3}$};
					\node at (2.5,3.5) {$3$};
				\end{tikzpicture}$$};
			\node at (10.5,0) {$\SYT{0.17 in}{}{{{3}},{{\overline{3}},{1}},{{\overline{1}},{\overline{2}},{2}}}$};
			\draw [->] (1,.5) -- (2.5,1);
			\draw [->] (1,-.5) -- (2.5,-1);
			\draw [->] (4.5,1) -- (6,1);
			\draw [->] (4.5,-1) -- (6,-1);
			\draw [->] (8,1) -- (9.5,.5);
			\draw [->] (8,-1) -- (9.5,-.5);	
		\end{tikzpicture}
		
		\noindent The right key tableau of $T$ has columns $r\!\YT{0.17in}{}{
			{{3}},
			{{\overline{3}}},
			{{\overline{1}}}}$, $r\!\YT{0.17in}{}{
			{{3}},
			{{\overline{1}}}}$ and $r\!\YT{0.17in}{}{
			{{\overline{1}}}}$.
		Hence $K_+(T)=\YT{0.17in}{}{
			{{3},{3},{\overline{1}}},
			{{\overline{2}},{\overline{1}}},
			{{\overline{1}}}}$.
The left key tableau of  $T$
 has columns $\ell \!\YT{0.17in}{}{
			{{1}},
			{{3}},
			{{\overline{3}}}}$, $\ell \!\YT{0.17in}{}{
			{{1}},
			{{2}}}$ and $\ell \!\YT{0.17in}{}{
			{{2}}}$.
		Hence $K_-(T)=\YT{0.17in}{}{
			{{1},{1},{2}},
			{{2},{2}},
			{{\overline{3}}}}$.

To obtain the {  weak} Bruhat  graph of the cocrystal keys in $\mathfrak{CB}^{\lambda'}(evac(T))$,
we $\pi$-rotate the whole  Bruhat graph of the cocrystal keys in $\mathfrak{CB}^{\lambda'}(T)$, $\pi$-rotate each co-crystal key and replace each entry by its symmetric. That is, we compute $\xi(\mathfrak{CB}^{\lambda'}(T))$

\begin{tikzpicture}
			\node at (-1,0) {$evac (T)=\YT{0.17in}{}{
					{{1},{2},{\overline{2}}},
					{{3},{\overline{1}}},
					{{\overline{3}}}}$};	
			\node at (3.5,1) {$\begin{tikzpicture}[scale=.4, baseline={([yshift=-.8ex]current bounding box.center)}]
					\draw (0,0) rectangle +(1,1);
					\draw (0,1) rectangle +(1,1);
					\draw (0,2) rectangle +(1,1);
					\draw (1,2) rectangle +(1,1);
					\draw (2,2) rectangle +(1,1);
					\draw (2,3) rectangle +(1,1);
					\node at (.5,.5) {${\overline 3}$};
					\node at (.5,1.5) {$3$};
					\node at (.5,2.5) {$1$};
					\node at (1.5,2.5) {$2$};
					\node at (2.5,2.5) {$\overline{1}$};
					\node at (2.5,3.5) {$\overline 2$};
				\end{tikzpicture}$};
			\node at (7,1) {$\begin{tikzpicture}[scale=.4, baseline={([yshift=-.8ex]current bounding box.center)}]
					\draw (0,0) rectangle +(1,1);
					\draw (1,0) rectangle +(1,1);
					\draw (1,1) rectangle +(1,1);
					\draw (1,2) rectangle +(1,1);
					\draw (2,1) rectangle +(1,1);
					\draw (2,2) rectangle +(1,1);
					\node at (.5,.5) {$1$};
					\node at (1.5,.5) {$\overline{3}$};
					\node at (1.5,1.5) {${3}$};
					\node at (1.5,2.5) {$2$};
					\node at (2.5,1.5) {$\overline{1}$};
					\node at (2.5,2.5) {$\overline 2$};
				\end{tikzpicture}$};
			\node at (3.5,-1) {$\begin{tikzpicture}[scale=.4, baseline={([yshift=-.8ex]current bounding box.center)}]
					\draw (0,0) rectangle +(1,1);
					\draw (0,1) rectangle +(1,1);
					\draw (1,0) rectangle +(1,1);
					\draw (1,1) rectangle +(1,1);
					\draw (1,2) rectangle +(1,1);
					\draw (2,2) rectangle +(1,1);
					\node at (.5,.5) {$\overline{3}$};
					\node at (.5,1.5) {$1$};
					\node at (1.5,.5) {$\overline{1}$};
					\node at (1.5,1.5) {${3}$};
					\node at (1.5,2.5) {$2$};
					\node at (2.5,2.5) {$\overline{2}$};
				\end{tikzpicture}$};
			\node at (7,-1) {$$\begin{tikzpicture}[scale=.4, baseline={([yshift=-.8ex]current bounding box.center)}]
					\draw (0,0) rectangle +(1,1);
					\draw (0,1) rectangle +(1,1);
					\draw (1,1) rectangle +(1,1);
					\draw (2,1) rectangle +(1,1);
					\draw (2,2) rectangle +(1,1);
					\draw (2,3) rectangle +(1,1);
					\node at (.5,.5) {$\overline{3}$};
					\node at (.5,1.5) {$1$};
					\node at (1.5,1.5) {${3}$};
					\node at (2.5,1.5) {$\overline{1}$};
					\node at (2.5,2.5) {$\overline{3}$};
					\node at (2.5,3.5) {$3$};
				\end{tikzpicture}$$};
			\node at (10.5,0) {$\SYT{0.17 in}{}{{3},{\overline 3,{ 3}},{\overline  1,\overline{3}, 1}}$};
			\draw [->] (1,.5) -- (2.5,1);
			\draw [->] (1,-.5) -- (2.5,-1);
			\draw [->] (4.5,1) -- (6,1);
			\draw [->] (4.5,-1) -- (6,-1);
			\draw [->] (8,1) -- (9.5,.5);
			\draw [->] (8,-1) -- (9.5,-.5);	
		\end{tikzpicture}

\noindent The right key tableau of $evac T$ has columns $r\!\YT{0.17in}{}{
			{{3}},
			{{\overline{3}}},
			{{\overline{1}}}}$, $r\!\YT{0.17in}{}{
			{{\overline 2}},
			{{\overline{1}}}}$ and $r\!\YT{0.17in}{}{
			{{\overline{2}}}}$.
		Hence $K_+(evac T)=\YT{0.17in}{}{
			{{3},{\overline 2},{\overline{2}}},
			{{\overline{2}},{\overline{1}}},
			{{{\overline 1}}}}$.
The left key tableau of the tableau $evac T=\YT{0.17in}{}{
			{{1},{3},{\overline{1}}},
			{{3},{\overline{3}}},
			{{\overline{3}}}}$ has columns $\ell \!\YT{0.17in}{}{
			{{1}},
			{{3}},
			{{\overline{3}}}}$, $\ell \!\YT{0.17in}{}{
			{{1}},
			{{\overline 3}}}$ and $\ell \!\YT{0.17in}{}{
			{{1}}}$.
		Hence $K_-(evac T)=\YT{0.17in}{}{
			{{1},{1},{1}},
			{{2},{\overline{3}}},
			{{\overline{3}}}}$.

\end{ex}

	Let $T=C_1C_2\cdots C_k$ be a straight KN tableau with columns $C_1, C_2, \dots,  C_k$. Note that, to compute which entries appear in the $i$-th column of $K_+(T)$ we do not need to look to the first $i-1$ columns of $T$. We only need the last column of a skew tableau obtained by applying the SJDT to the columns $C_i\cdots C_k$ of $T$, so that the last column has the length of $C_i$, because, by Proposition \ref{sahpesjdt}, all last columns of skew tableaux associated to $T$ with the same length
	are equal. Let $K_+^1(T)$ be the map that given a tableau returns the first column of $K_+(T)$.
	This is noticeable in Example \ref{jeudetaquin3} where $K_+(T)=K^1_+(C1C2C3)K^1_+(C2C3)K^1_+(C3)$. In general, $K_+(T)=K_+^1(C_1\cdots C_k)K_+^1(C_2\cdots C_k)\cdots K_+^1(C_k)$.
	Based on this observation and Proposition \ref{sahpesjdt}, next algorithm refines our way to compute $K_1^+(T)$ using SJDT:
	\begin{alg}\label{rightkjdtq}
		Let $T$ be a straight KN tableau:
		\begin{enumerate}
			\item Let  $i=2$.
			\item  If $T$ has exactly one column, return the right column of $T$. Otherwise, let $T_i:=T_2$ be the tableau formed by the first two columns of $T$.
			\item  If the length of the two columns of $T_i$ is the same, put $T_i':=T_i$. Else, play the SJDT on $T_i$ until both column lengths are swapped, obtaining $T_i'$.
			\item  If $T$ has more than $i$ columns, redefine $i:=i+1$, and define $T_i$ to be the two-columned tableau formed with the rightmost column of $T'_{i-1}$ and the $i$-th column of $T$, and go back to $2.$ Else, return the right column of the rightmost column of $T_i'$.\end{enumerate}
	\end{alg}
	
	This algorithm is illustrated on the top path of \eqref{jeudetaquin3}.

	\begin{cor}
		If $T$ is a rectangular tableau, $K_+(T)=rC_k  rC_k\cdots rC_k$ ($k$ times).
	\end{cor}

	Let $T=C_1C_2\cdots C_k$ be a KN tableau with columns $C_1, C_2, \dots,  C_k$. We note that given a tableau $T$, to compute which entries appear in the $i$-th column of $K_-(T)$ we only need to look to the first $i$ columns of $T$. We need the first column of a skew tableau obtained by applying the SJDT to the columns $C_1\cdots C_i$ of $T$, so that the first column has the length of $C_i$. Let $K_-^1(T)$ be the map that given a tableau returns the last column of $K_-(T)$.
	
	In Example \ref{jeudetaquin3} we have $K_-(T)=K^1_-(C1)K^1_-(C1C2)K^1_-(C1C2C3)$.
	In general, $K_-(T)=K_-^1(C_1)\cdots K_-^1(C_1\cdots C_{k-1})K_-^1(C_1\cdots C_k)$.
	Next we present how we compute $K_-^1(T)$ using SJDT:
	\begin{alg}\label{leftkjdtq}
		Let $k$ be the number of columns of $T$ and $i=k-1$.
		\begin{enumerate}
			\item If $T$ has exactly one column, return the left column of $T$. Otherwise, let $T_i:=T_{k-1}$ be the tableau formed by the last two columns of $T$.
			\item If the length of the two columns of $T_i$ is the same, put $T_i':=T_i$. Else, play the SJDT on $T_i$ until both column lengths are swapped, obtaining $T_i'$.
			\item If $i\neq 1$, redefine $i:=i-1$, and define $T_i$ as the two-columned tableau formed with the leftmost column of $T'_{i+1}$ and the $i$-th column of $T$, and go back to $(1)$. Else, return the left column of the leftmost column of $T_i'$.
		\end{enumerate}
	\end{alg}
	
	This algorithm is exemplified on the bottom path of Example \ref{jeudetaquin3}.

	\begin{cor}
		If $T$ is a rectangular tableau, $K_-(T)=\ell C_1  \ell C_1\cdots \ell C_1$ ($k$ times).
	\end{cor}
	%

	Next, we present a way of computing $K_+^1(T)$ that does not require the SJDT. Willis has done this when $T$ is a SSYT \cite{Willis2011ADW}. It is a simplified version of the algorithm presented here.
	%
	
	\section{Direct way: right and left}\label{sec:directway}
\subsection{Right key}\label{rightkeydw}
	Let $T = C_1C_2$ be a straight KN two column tableau and $spl(T)=\ell C_1rC_1$ $\ell C_2rC_2$ a straight semistandard tableau. In particular, $rC_1\ell C_2$ is a semistandard tableau. The \emph{matching between $rC_1$ and $\ell C_2$} is defined as follows:
	
	$\bullet$  Let $\beta_1<\dots<\beta_{m'}$ be the elements of $\ell C_2$. Let $i$ go from $m'$ to $1$, match $\beta _i$ with the biggest, not yet matched, element of $rC_1$ smaller or equal than $\beta_i$.

	\begin{teo}[The direct way algorithm for the right key] \label{mainthm}
		Let $T$ be a straight KN tableau with columns $C_1, C_2, \dots, C_k$, and consider its split form $spl(T)$. For every right column $rC_2, \dots, rC_k$, add empty cells to the bottom in order to have all columns with the same length as $rC_1$. We will fill all of these empty cells recursively, proceeding from left to right. The extra numbers that are written in the column $rC_2$ are found in the following way:
		
		$\bullet$ match $rC_1$ and $\ell C_2$.
		
		$\bullet$ Let $\alpha_1<\dots<\alpha_m$ be the elements of $r C_1$. Let $i$ go from $1$ to $m$. If $\alpha_i$ is not matched with any entry of $\ell C_2$, write in the new empty cells of $rC_2$ the smallest element bigger or equal than $\alpha_i$ such that neither it nor its symmetric exist in $rC_2$ or in its new cells. Let $C_2'$ be the column defined by $rC_2$ together with the filled extra cells, after ordering.
		
		To compute the filling of the extra cells of $rC_3$, we do the same thing, with $C_2'$ and $C_3$. If we do this for all pairs of consecutive columns, we eventually obtain a column $C_k'$, consisting of $rC_k$ together with extra cells, with the same length as $rC_1$. We claim that $C_k'=K_+^1(T)$.
	\end{teo}	
	\begin{ex}
		Let $T=C_1C_2C_3=\YT{0.17in}{}{
			{{1},{3},{\overline{1}}},
			{{3},{\overline{3}}},
			{{\overline{3}}}}$, with split form 	$spl(T)=\YT{0.17in}{}{
			{{1},{1},{2},{3},{\overline{1}},{\overline{1}}},
			{{2},{3},{\overline{3}},{\overline{2}}},
			{{\overline{3}},{\overline{2}}}}$.
		We match $rC_1$ and $\ell C_2$, as indicated by the letters $a$ and $b$:
		$\YT{0.17in}{}{
			{{1},{1^a},{2^a},{3},{\overline{1}},{\overline{1}}},
			{{2},{3^b},{\overline{3}^b},{\overline{2}}},
			{{\overline{3}},{\overline{2}}}}$. Hence $\overline{2}$ creates a $\overline{1}$ in $rC_2$, completing the right column  $rC_2$:
		$\YT{0.17in}{}{
			{{1},{1},{2},{3},{\overline{1}^a},{\overline{1}}},
			{{2},{3},{\overline{3}},{\overline{2}}},
			{{\overline{3}},{\overline{2}},,{{\color{blue}\overline{1}^a}}}}$.
		Now we match $C_2'$ and $\ell C_3$, which is already done, and see what new cells $3$ and $\overline{2}$ we create in $rC_3$, obtaining
		$\YT{0.17in}{}{
			{{1},{1},{2},{3},{\overline{1}},{\overline{1}}},
			{{2},{3},{\overline{3}},{\overline{2}},,{{\color{blue}3}}},
			{{\overline{3}},{\overline{2}},,{{\color{blue}\overline{1}}},,{{\color{blue}\overline{2}}}}}$. Hence $K_+^1(T)=\YT{0.17in}{}{
			{{3}},
			{{\overline{2}}},{{\overline{1}}}}$ is obtained from $C'_3$ after reordering its entries.
	\end{ex}
	
	\subsection{The proof of Theorem \ref{mainthm}}
	It is enough to prove that by the end of this algorithm, the entries in $C_k'$ are the entries on the right column of the rightmost column of $T'_k$ from Algorithm \ref{rightkjdtq}.
	In fact, it is enough to do this for $k=2$. For bigger $k$ note that the entries that are "slid" into $C_k$ come from $rC_{k-1}$, so, to go to the next step on the SJDT algorithm we only need to know the previous right column, which is exactly what we claim to compute this way. The next lemma determines which number is added to $rC_2$ given that we know $\alpha$, the entry that is horizontally slid:

	\begin{lema}
		Suppose that $T=C_1C_2$ is a non-rectangular two-column tableau (if the tableau is rectangular then we have nothing to do). Play the SJDT on this tableau which ends up moving one cell from the first column to the second (some entries may change their values). Then,
		
		$\bullet$ Immediately before the horizontal slide of the SJDT, the entry $\alpha$, on the left of the puncture,  is an unmatched cell of $rC_1$.
		
		$\bullet$ Call $C_1'$ and $C_2'$ to both columns after the horizontal slide on $T$. The new entry in $rC_2'$, compared to $rC_2$, is the smallest element bigger or equal than $\alpha$ such that neither it nor its symmetric exist in $rC_2$.
	\end{lema}
	\begin{ex}
		Let $T=\YT{0.17in}{}{
			{{2},{3}},
			{{3},{4}},
			{{5},{\overline{5}}},
			{{\overline{5}}},
			{{\overline{2}}}}$. After splitting, and just before the first horizontal slide, we have  $T=\YT{0.17in}{}{
			{{1},{2},{3},{3}},
			{{3},{3},{4},{4}},
			{{4},{5},{\overline{5}},{\overline{5}}},
			{{\overline{5}},{\overline{4}},{\ast}, {\ast}},
			{{\overline{2}},{\overline{1}}}}$.
		The new entry in $rC_2$ is $\overline{2}$, as predicted by the lemma:
		$\YT{0.17in}{}{
			{{1},{2},{2},{3}},
			{{3},{3},{3},{4}},
			{{4},{4},{\overline{5}},{\overline{5}}},
			{{\ast},{\ast},{\overline{4}},{\overline{2}}},
			{{\overline{2}},{\overline{1}}}}$.
	\end{ex}
	\begin{proof}
		\textbf{Case 1:} $\alpha$ is barred. Then $C_2'=C_2\cup\{\alpha\}$. If $\overline{\alpha}$  does  not exist neither in $C_2$ nor in $\Phi(C_2)$, then $\alpha$ will exist in both $C_2'$ and $\Phi(C_2')$.
		If $\overline{\alpha}$ does exist in $C_2$, and consequently in $\Phi(C_2)$ (but $\alpha \notin \Phi(C_2)$), then $\alpha$ and $\overline{\alpha}$ will both exist in $C_2'$. Hence, in the construction of the barred part of $\Phi(C_2')$, compared to $\Phi(C_2)$, there will be a new barred number which is the smallest number bigger (or equal, but the equality can not happen) than $\alpha$ such that neither it nor its symmetric exist in the barred part of $\Phi(C_2)$ or the unbarred part of $C_2$ (i.e., $rC_2$). If $\alpha$ existed in $\Phi(C_2)$, then $\overline{\alpha}$ existed in $\Phi(C_2)$. That means that whatever number got sent to $\alpha$ in the construction of $\Phi(C_2)$ will be sent to the next available number, meaning that in $rC_2$ will appear a new number, the smallest number bigger (or equal, but the equality can not happen because $\alpha$ is already there) than $\alpha$ such that neither it nor its symmetric exist in $rC_2$.

		\textbf{Case 2:} $\alpha$ is unbarred. Then $C_2'=\Phi^{-1}(\Phi(C_2)\cup \{\alpha\})$. If $\overline{\alpha}$ does  not exist in $C_2$ nor in
		$\Phi(C_2)$, then $\alpha$ will exist in both $C_2'$ and $\Phi(C_2')$.
		If $\overline{\alpha}$ existed in $\Phi(C_2)$, and consequently in $C_2$, then both $\alpha$ and $\overline{\alpha}$ will exist in $\Phi(C_2')$, hence, if we start in the coadmissible column, in the construction of the unbarred part of $C_2'$, compared to $C_2$, there will be a new unbarred number which is the smallest number bigger than $\alpha$ such that neither it nor its symmetric exist in $rC_2$. Finally, if $\alpha$ existed in $C_2$, then $\overline{\alpha}$ also existed in $C_2$.
		That means that whatever number got sent to $\alpha$ in the construction of $C_2$, from $\Phi(C_2)$, will be sent to the next available number, meaning that in $rC_2$ will appear a new number, the smallest number bigger than $\alpha$ such that neither it nor its symmetric exist in $rC_2$.
	\end{proof}
	
	\begin{proof}[Proof of Theorem \ref{mainthm}:]
		Each SJDT in $T$, a two-column skew tableau, moves a cell from the first to the second column. We will prove that if we apply the direct way algorithm after each SJDT, the output $C_2'$ does not change. The cells on $\ell C_2$ without cells to its left do not get to be matched. When we slide horizontally, the columns $rC_1$ and $\ell C_2$ may change more than the adding/removal of $\alpha$, the horizontally slid entry.	
		Since the horizontal slides happen from top to bottom, we only need to see  what changes happen to bigger entries than the one slid. 
		All entries above $\alpha$ are matched to the entry in the same row in $\ell C_2$.
		
		If $\alpha$ is barred then, the remaining barred entries of $rC_1$ and $\ell C_2$ remain unchanged, and since all entries above $\alpha$, including the unbarred ones, are matched to the entry directly on their right, there is no noteworthy change and everything runs as expected.
		
		If $\alpha$ is unbarred then, the remaining unbarred entries of $rC_1$ and $\ell C_2$ remain unchanged.
		In the barred part of $rC_1$ either nothing happens, or there is an entry bigger than $\overline{\alpha}$, $\overline{x}$, that gets replaced by $\overline{\alpha}$. Note that $\overline{x}$ must be such that for every number between $\overline{x}$ and $\overline{\alpha}$, either it or its symmetric existed in $rC_1$.
		In the barred part of $\ell C_2$, if $\overline{\alpha}\in \ell C_2$, then $\overline{\alpha}$ gets replaced by $\overline{y}$, smaller than $\overline{\alpha}$, such that for every number between $\overline{y}$ and $\overline{\alpha}$, either it or its symmetric existed in $\ell C_2$, and both $y$ and $\overline{y}$ do not exist in $\ell C_2$.
		
		Let's look to $\ell C_2$. Let $\alpha<p_1<p_2<\dots< p_m=y$ be the numbers between $\alpha$ and $y$ that does not  exist in $\ell C_2$, right before the horizontal slide. Then, their symmetric exist in $\ell C_2$. For all numbers in $rC_2$ between $\alpha$ and $y$, there exists, in the same row in $rC_1$, a number between $\alpha$ and $y$. Let $\alpha<p_1'<p_2'<\dots< p_m'= y$ be the missing numbers between $\alpha$ and $y$ in $rC_1$, then $p_i\leq p_i'$.
		Note that $\overline{p_1}>\overline{p_2}>\dots >\overline{p_m}=\overline{y}$ exist in $\ell C_2$ after the horizontal slide and that the biggest numbers between $\overline{\alpha}$ and $\overline{y}$ (not including $\overline{\alpha}$) that can exist in $rC_1$ are $\overline{p_1'}>\overline{p_2'}>\dots> \overline{p_m'}$, and since $\overline{p_i}\geq \overline{p_i'}$, the matching holds for this interval after swapping $\overline{\alpha}$ by $\overline{y}$ in $\ell C_2$.

		Now let's look to $rC_1$. Before the slide, call $\overline{x'}$ to the biggest unmatched number of $rC_1$ smaller or equal than $\overline{x}$. If there is no such $\overline{x'}$, then everything in $rC_1$ between $\overline{\alpha}$ and $\overline{x}$ is matched, hence swapping $\overline{x}$ by $\overline{\alpha}$ will keep all of them matched, meaning that the algorithm works in this scenario.
		Let $x'<q_1<q_2<\dots< q_m<\alpha$ be the numbers between $x'$ and $\alpha$ that does not  exist in $r C_1$, right before the horizontal slide. Then, their symmetric exist in $r C_1$. For all numbers in $r C_1$ between $x'$ and $\alpha$, there exists, in the same row in $\ell C_2$, a number between $x'$ and $\alpha$, because $\alpha$ is unmatched. Let $x'<q_1'<q_2'<\dots< q_m'<\alpha$ be the missing numbers between $x'$ and $\alpha$ in $\ell C_2$, then $q_i\geq q_i'$.
		Note that $\overline{q_1}>\overline{q_2}>\dots >\overline{q_m}>\overline{\alpha}$ exist in $r C_1$ after the horizontal slide and the numbers between $\overline{x'}$ and $\overline{\alpha}$ that can exist in $\ell C_2$ are $\overline{q_1'}>\overline{q_2'}>\dots> \overline{q_m'}$, and since $\overline{q_i}\leq \overline{q_i'}$, these numbers are matching a number bigger or equal then $q_i$  in $r C_1$, meaning that $\alpha$ is unmatched in $rC_1$.
		Ignoring signs, the numbers that appear in either $rC_1$ or $\ell C_2$ are the same. So before playing the SJDT, applying the direct way algorithm we have that the unmatched numbers in $rC_1$ are sent to the not used numbers of $\overline{q_1'}>\overline{q_2'}>\dots> \overline{q_m'}$ in $\ell C_2$ (this is a bijection), and $\overline{x'}$ is sent to the smallest available number, bigger or equal than $\overline{x'}$. Now consider $rC_1$ and $\ell C_2$ after the slide. In $rC_1$ we replace $x'$ by $\overline{\alpha}$ and remove $\alpha$ and in $\ell C_2$ there is $\alpha$ or $\overline{\alpha}$. In the direct algorithm, all unmatched numbers of $\overline{q_1}>\overline{q_2}>\dots >\overline{q_m}>\overline{\alpha}$ are sent to the not used numbers of $\overline{q_1'}>\overline{q_2'}>\dots> \overline{q_m'}$ in $\ell C_2$, but now we have more numbers in the first set than in the second, meaning that $\overline{\alpha}$ will bump the image of the least unmatched number, which will bump the image of the second least unmatched number, and so on, meaning that the image of biggest unmatched will be out of this set. This image will be the smallest number available, which was the image of $x'$ before the horizontal slide.
		
		Hence, the outcome of the direct way does not change due to the changes to the columns when we play the SJDT, meaning that the outcome is what we intend.
	\end{proof}

	\subsection{Left key - a direct way}\label{leftkeydw}
	
	\begin{teo}\label{mainthmleft}
		Let $T$ be a KN tableau with columns $C_1, C_2, \dots, C_k$, and consider its split form $spl(T)$. 
		
		We will now delete entries from the left columns, proceeding from right to left, in such a way that in the end every left column has as many entries as $C_k$. The entries deleted from $\ell C_{k-1}$ are found in the following way:
		
		We start by creating a matching between $rC_{k-1}$ and $\ell C_k$. Let $\beta_1<\dots<\beta_{m}$ be the unmatched elements  of $r C_{k-1}$. For $i$ between $1$ and $m$, let $\alpha_i$ be the entry on $\ell C_{k-1}$ next to $\beta_i$. Let $i$ go from $1$ to $m$. Starting at $\alpha_i$ and going up, delete the first entry of $\ell C_{k-1}$ {bigger than the entry directly Northeast of it}. If there is no entry in this conditions, delete the top entry of $\ell C_{k-1}$. Also delete $\beta_i$ from $r C_{k-1}$. By the end of this procedure we obtain $\ell C'_{k-1}$ with the same number of cells as $C_k$.
		
		To continue the algorithm, we do the same thing with $C_{k-2}$ and $\ell C'_{k-1}$. If we do this for all pairs of consecutive columns, we eventually obtain a column $\ell C_1'$, consisting of $\ell C_1$ with some entries deleted, with the same length as $C_k$. We claim that $\ell C_1'=K_-^1(T)$.
	\end{teo}

	\begin{ex}
		Consider $T=\YT{0.17in}{}{
			{{2},{3},{\overline{3}}},
			{{3},{\overline{3}}},
			{{\overline{3}}}}$ whose split form is	$spl(T)=\YT{0.17in}{}{
			{{1},{2},{2},{3},{\overline{3}},{\overline{3}}},
			{{2},{3},{\overline{3}},{\overline{2}}},
			{{\overline{3}},{\overline{1}}}}$.
		We match $rC_2$ and $\ell C_3$, obtaining:
		$\YT{0.17in}{}{
			{{1},{2},{2},{3^a},{\overline{3}^a},{\overline{3}}},
			{{2},{3},{\overline{3}},{\overline{2}}},
			{{\overline{3}},{\overline{1}}}}$. Hence $\overline{2}$ is unmatched in $rC_2$. So it will get deleted, alongside the $\overline{3}$ in $\ell C_2$.
		Thus we have $\YT{0.17in}{}{
			{{1},{2^a},{2^a},{3},{\overline{3}},{\overline{3}}},
			{{2},{3},{{\color{gray}\overline{3}}},{{\color{gray}\overline{2}}}},
			{{\overline{3}},{\overline{1}}}}$ (the deleted entries are greyed out).
		
		Now we have to create the match between $\ell C_2'$ and $r C_1$, which is already done. The entries $3$ and $\overline{1}$ are unmatched in $rC_1$, hence they will be removed alongside the entries $1$ and $\overline{3}$ in $\ell C_1$, obtaining
		$\YT{0.17in}{}{
			{{{\color{gray}1}},{2},{2},{3},{\overline{3}},{\overline{3}}},
			{{2},{{\color{gray}3}},{{\color{gray}\overline{3}}},{{\color{gray}\overline{2}}}},
			{{{\color{gray}\overline{3}}},{{\color{gray}\overline{1}}}}}$.
		Hence $K^1_-(T)=\YT{0.17in}{}{
			{{2}}}$.
	\end{ex}

	\subsection{Proof of Theorem \ref{mainthmleft}}
	It is enough to prove that by the end of this algorithm, the entries in $\ell C_j'$ are the entries on the left column of the leftmost column of $T'_j$ from Algorithm \ref{leftkjdtq}.
	Just like in the right key case, it is enough to do this for $j=k-1$. For smaller $j$ note that we only need to know what remains in the left column $\ell C_j'$, which is exactly what we claim to compute this way.
	
	So only need to prove this when $T$ is a two-column tableaux.
	
	\begin{lema}\label{leftremove}
		Suppose that $T$ is a non-rectangular two-column tableau (if the tableau is rectangular then we have nothing to do). Play the SJDT on this tableau, which ends up moving one cell from the first column to the second (some entries may change its value). Immediately before the horizontal slide of the SJDT, the entry $\beta$, on the left of the puncture,  is an unmatched cell of $rC_1$.
		Call $C_1'$ and $C_2'$ to both columns after the slide.
		
		Then  $\ell C_1'$ will lose an entry, compared to $\ell C_1$, which
		is the biggest entry of $\ell C_{1}$, in a row not under the row that contains $\beta$, bigger than the entry directly Northeast of it.
	\end{lema}
	\begin{ex}
		Consider the tableau $T=\YT{0.17in}{}{
			{{2},{3}},
			{{4},{4}},
			{{5},{\overline{2}}},
			{{\overline{5}}},
			{{\overline{2}}}}$. After split, and just before the horizontal slide, we have  $T=\YT{0.17in}{}{
			{{1},{2},{3},{3}},
			{{3},{4},{4},{4}},
			{{4},{5},{\ast},{\ast}},
			{{\overline{5}},{\overline{3}},{\overline{2}}, {\overline{2}}},
			{{\overline{2}},{\overline{1}}}}$.
		So $5$ slides from $rC_1$ to $\ell C_2$, obtaining the tableau $\YT{0.17in}{}{
			{{2},{3}},
			{{4},{4}},
			{{\ast},{5}},
			{{\overline{5}},{\overline{2}}},
			{{\overline{2}}}}$, whose split is $\YT{0.17in}{}{
			{{1},{2},{3},{3}},
			{{4},{4},{4},{4}},
			{{\ast},{\ast},{5},{5}},
			{{\overline{5}},{\overline{5}},{\overline{2}},{\overline{2}}},
			{{\overline{2}},{\overline{1}}}}$. The entry removed from $\ell C_1$ is $3$, as predicted by the lemma.
	\end{ex}
	\begin{proof} If $\beta$ is unbarred then look at all numbers $\beta\leq i\leq n$, and count, in $C_1$, count how many of them exist together with its symmetric and it is not matched to a number with bigger than $\beta$ in the coadmissible column. Let $k$ be that count. Now let $i$ go from $\beta-1$ to $1$. If $i$ and $\overline{i}$ exist in $C_1$ then $k:=k+1$, and if neither exist then $k:=k-1$. Since $C_1$ is admissible, eventually $k=0$ and this is the $i$ removed from $\ell C_1$. So, the columns $\ell C_1$ and $rC_1$ have same number of entries with absolute value bigger or equal than $i$, hence the entry $i$ of $\ell C_1$ is bigger than the entry directly Northeast of it.
		
		If $\beta$ is barred then look at all numbers $\beta\leq i\leq \overline{1}$, and count, in $C_1$, count how many of them exist together with its symmetric and it is not matched to a number bigger than $\beta$ in the coadmissible column. Let $k$ be that count. Now let $i$ go from $\beta-1$ to $\overline{n}$. If $i$ and $\overline{i}$ exist in $C_1$ then $k:=k+1$, and if neither exist then $k:=k-1$. Since $\Phi(C_1)$ is coadmissible, eventually $k=0$ and this is the $i$ removed from $\ell C_1$. The columns $\ell C_1$ and $rC_1$ have same number of entries with absolute value smaller or equal than $\overline{i}$, hence the entry $i$ of $\ell C_1$ is bigger than the entry directly Northeast of it (remember that $i$ is negative).
	\end{proof}
	\begin{proof}[Proof of Theorem \ref{mainthmleft}:]
		
		Hence we have determined which entry is removed from $\ell C_1$ given that we know $\beta$, the entry of the cell that is horizontally slid.
		The SJDT on $T$ may change the entries or the matching in $rC_1$. We need to prove that, even with these eventual changes, the entries removed from $\ell C_1$ are the ones that we calculated in the beginning, before doing any SJDT slide.

		If $\beta$ is barred, since we run the unmatched entries of $rC_1$ from smallest to biggest, when removing $\beta$ from $rC_1$ the unbarred part of $rC_1$ remains the same, hence, the remaining entries and matched entries do not change, hence the outcome will be the one predicted.

		If $\beta$ is unbarred then the remaining unbarred entries of $rC_1$ remain unchanged. In the barred part of $rC_1$ either nothing happens, or there is an entry bigger than $\overline{\beta}$, $\overline{x}$, that gets replaced by $\overline{\beta}$. Note that $\overline{x}$ must be such that for every number between $\overline{x}$ and $\overline{\beta}$, either it or its symmetric existed in $rC_1$. This can only happen if $k$, from the proof of Lemma \ref{leftremove} starts being bigger than $0$.
		
		Since for all numbers between $\overline{x}$ and $\overline{\beta}$ either it or its symmetric exist in $rC_1$, all unmatched entries here will remove from $\ell C_1$ an entry smaller or equal than $\overline{x}$. In fact, the way of constructing $\overline{x}$ and $i$, from the proof of Lemma \ref{leftremove}, is effectively the same. Since, after the slide of $\beta$,  we may have different matches in the numbers between  $\overline{x}$ and $\overline{\beta}$, and the number of unmatched entries remains the same after the slide. Since all unmatched entries in here will remove something smaller or equal than $\overline{\beta}$ from $\ell C_1$, the outcome of the algorithm is the same as if we apply it to $\ell C_1, \, rC_1$ before or after the horizontal slide. Hence we do not need to do any SJDT in order to know the entries of $\ell C_1$ after the SJDT.
	\end{proof}

	\subsection{SJDT and direct keys example}\label{SecExample}
	In this section we illustrate the computation of the right and left keys of a KN tableau via SJDT and using the \emph{direct way}.
	
	Let $n=5$ and let  be  $T=\YT{0.17in}{}{
		{{2},{3},{3},{4}},
		{{4},{\overline{4}},{\overline{4}},{\overline{4}}},
		{{\overline{5}},{\overline{3}},{\overline{2}}},
		{{\overline{4}}},
		{{\overline{3}}}}$ the KN tableau
	with split form $$spl(T)=\YT{0.17in}{}{
		{{1},{2},{2},{3},{3},{3},{3},{4}},
		{{\color{magenta} 2},{4},{\overline{4}},{\overline{4}},{\overline{4}},{\overline{4}},{\overline{4}},{\overline{3}}},
		{{\overline{5}},{\overline{5}},{\overline{3}},{\overline{2}},{\overline{2}},{\overline{2}}},
		{{\overline{4}},{\overline{3}}},
		{{\overline{3}},{\overline{1}}}}.$$
	
	In order to find the right (resp. left) key of $T$, we play the SJDT to swap heights of consecutive columns, and find skew tableaux, Knuth related to $T$, such that for every column height there is a skew tableau whose last column (resp. first) has that height.

	\begin{tikzpicture}
		\node at (0,0) {$\YT{0.17in}{}{
				{{2},{3},{3},{4}},
				{{4},{\overline{4}},{\overline{4}},{\overline{4}}},
				{{\overline{5}},{\overline{3}},{\overline{2}}},
				{{\overline{4}}},
				{{\overline{3}}}}$};	
		\node at (2,0) {$\begin{tikzpicture}[scale=.4, baseline={([yshift=-.8ex]current bounding box.center)}]
				\draw (0,0) rectangle +(1,1);
				\draw (0,1) rectangle +(1,1);
				\draw (0,2) rectangle +(1,1);
				\draw (1,0) rectangle +(1,1);
				\draw (1,1) rectangle +(1,1);
				\draw (1,2) rectangle +(1,1);
				\draw (1,3) rectangle +(1,1);
				\draw (1,4) rectangle +(1,1);
				\draw (2,2) rectangle +(1,1);
				\draw (2,3) rectangle +(1,1);
				\draw (2,4) rectangle +(1,1);
				\draw (3,3) rectangle +(1,1);
				\draw (3,4) rectangle +(1,1);
				\node at (.5,.5) {$\overline{3}$};
				\node at (.5,1.5) {$\overline{5}$};
				\node at (.5,2.5) {$2$};
				\node at (1.5,0.5) {$\overline{3}$};
				\node at (1.5,1.5) {$\overline{4}$};
				\node at (1.5,2.5) {$\overline{5}$};
				\node at (1.5,3.5) {$5$};
				\node at (1.5,4.5) {$3$};
				\node at (2.5,2.5) {$\overline{2}$};
				\node at (2.5,3.5) {$\overline{4}$};
				\node at (2.5,4.5) {$3$};
				\node at (3.5,3.5) {$\overline{4}$};
				\node at (3.5,4.5) {$4$};
			\end{tikzpicture}$};
		\node at (4,0) {$\begin{tikzpicture}[scale=.4, baseline={([yshift=-.8ex]current bounding box.center)}]
				\draw (0,0) rectangle +(1,1);
				\draw (0,1) rectangle +(1,1);
				\draw (0,2) rectangle +(1,1);
				\draw (1,0) rectangle +(1,1);
				\draw (1,1) rectangle +(1,1);
				\draw (1,2) rectangle +(1,1);
				\draw (1,3) rectangle +(1,1);
				\draw (1,4) rectangle +(1,1);
				\draw (3,5) rectangle +(1,1);
				\draw (2,3) rectangle +(1,1);
				\draw (2,4) rectangle +(1,1);
				\draw (3,3) rectangle +(1,1);
				\draw (3,4) rectangle +(1,1);
				\node at (.5,.5) {$\overline{3}$};
				\node at (.5,1.5) {$\overline{5}$};
				\node at (.5,2.5) {$2$};
				\node at (1.5,0.5) {$\overline{3}$};
				\node at (1.5,1.5) {$\overline{4}$};
				\node at (1.5,2.5) {$\overline{5}$};
				\node at (1.5,3.5) {$5$};
				\node at (1.5,4.5) {$3$};
				\node at (2.5,3.5) {$\overline{4}$};
				\node at (2.5,4.5) {$3$};
				\node at (3.5,3.5) {$\overline{2}$};
				\node at (3.5,4.5) {$\overline{4}$};
				\node at (3.5,5.5) {$4$};
			\end{tikzpicture}$};
		\node at (6,0) {$\begin{tikzpicture}[scale=.4, baseline={([yshift=-.8ex]current bounding box.center)}]
				\draw (0,0) rectangle +(1,1);
				\draw (0,1) rectangle +(1,1);
				\draw (0,2) rectangle +(1,1);
				\draw (1,1) rectangle +(1,1);
				\draw (1,2) rectangle +(1,1);
				\draw (2,1) rectangle +(1,1);
				\draw (2,2) rectangle +(1,1);
				\draw (2,3) rectangle +(1,1);
				\draw (2,4) rectangle +(1,1);
				\draw (2,5) rectangle +(1,1);
				\draw (3,3) rectangle +(1,1);
				\draw (3,4) rectangle +(1,1);
				\draw (3,5) rectangle +(1,1);
				\node at (.5,.5) {$\overline{3}$};
				\node at (.5,1.5) {$\overline{5}$};
				\node at (.5,2.5) {$2$};
				\node at (1.5,1.5) {$\overline{4}$};
				\node at (1.5,2.5) {$2$};
				\node at (2.5,1.5) {$\overline{2}$};
				\node at (2.5,2.5) {$\overline{4}$};
				\node at (2.5,3.5) {$\overline{5}$};
				\node at (2.5,4.5) {$5$};	
				\node at (2.5,5.5) {$3$};
				\node at (3.5,3.5) {$\overline{2}$};
				\node at (3.5,4.5) {$\overline{4}$};
				\node at (3.5,5.5) {$4$};
			\end{tikzpicture}$};
		\node at (8,0) {$\begin{tikzpicture}[scale=.4, baseline={([yshift=-.8ex]current bounding box.center)}]
				\draw (0,1) rectangle +(1,1);
				\draw (0,2) rectangle +(1,1);
				\draw (1,1) rectangle +(1,1);
				\draw (1,2) rectangle +(1,1);
				\draw (1,3) rectangle +(1,1);
				\draw (2,1) rectangle +(1,1);
				\draw (2,2) rectangle +(1,1);
				\draw (2,3) rectangle +(1,1);
				\draw (2,4) rectangle +(1,1);
				\draw (2,5) rectangle +(1,1);
				\draw (3,3) rectangle +(1,1);
				\draw (3,4) rectangle +(1,1);
				\draw (3,5) rectangle +(1,1);
				\node at (.5,1.5) {$\overline{5}$};
				\node at (.5,2.5) {$2$};
				\node at (1.5,3.5) {$2$};
				\node at (1.5,2.5) {$\overline{4}$};
				\node at (1.5,1.5) {$\overline{3}$};
				\node at (2.5,1.5) {$\overline{2}$};
				\node at (2.5,2.5) {$\overline{4}$};
				\node at (2.5,3.5) {$\overline{5}$};
				\node at (2.5,4.5) {$5$};	
				\node at (2.5,5.5) {$3$};
				\node at (3.5,3.5) {$\overline{2}$};
				\node at (3.5,4.5) {$\overline{4}$};
				\node at (3.5,5.5) {$4$};
			\end{tikzpicture}$};
		\node at (10,0) {$\SYT{0.17in}{}{
				{{4}},
				{{5}},
				{{\overline{5}},{3},{2}},
				{{\overline{4}},{\overline{4}},{\overline{4}},{2}},
				{{\overline{2}},{\overline{2}},{\overline{3}},{\overline{5}}}}$};
	\end{tikzpicture}
	
	Each tableau is obtained from the previous after playing SJDT in two consecutive columns, swapping their heights.
	
	If we compute the right (resp. left) columns  of all last (resp. first) columns of these tableaux, we find the columns of the right (resp. left) key associated to $T$:
	
	$K_+(T)=\YT{0.17in}{}{
		{{4},{4},{4},{4}},
		{{5},{\overline{3}},{\overline{3}},{\overline{3}}},
		{{\overline{3}},{\overline{2}},{\overline{2}}},
		{{\overline{2}}},
		{{\overline{1}}}}$ and
	$K_-(T)=\YT{0.17in}{}{
		{{1},{2},{2},{2}},
		{{2},{\overline{5}},{\overline{5}},{\overline{5}}},
		{{\overline{5}},{\overline{3}},{\overline{3}}},
		{{\overline{4}}},
		{{\overline{3}}}}$.
	
	Note that we have $9$ horizontal slides in our sequence of tableaux, and for each horizontal slide we have to apply the map $\Phi$, or its inverse twice. This means that we are effectively computing the split form of $9$ skew tableaux, even though we only need $3$ tableaux (the first, the third and the last one) to have all column heights in each end of the tableau.
	
	Now we compute both keys using the direct way. In here we only need to compute one split form, and make some calculations on it, and on subtableaux of the split form.
	
	To compute the \emph{right key}, via \emph{direct way}, we need to compute the columns $K_+^1\left(\YT{0.17in}{}{
		{{2},{3},{3},{4}},
		{{4},{\overline{4}},{\overline{4}},{\overline{4}}},
		{{\overline{5}},{\overline{3}},{\overline{2}}},
		{{\overline{4}}},
		{{\overline{3}}}}\right)$,
	$K_+^1\left(\YT{0.17in}{}{
		{{3},{3},{4}},
		{{\overline{4}},{\overline{4}},{\overline{4}}},
		{{\overline{3}},{\overline{2}}}}\right)=K_+^1\left(\YT{0.17in}{}{
		{{3},{4}},
		{{\overline{4}},{\overline{4}}},
		{{\overline{2}}}}\right)$ and $K_+^1\left(\YT{0.17in}{}{
		{{4}},
		{{\overline{4}}} }\right)$.

	
	We start by splitting and matching, and every $\mapsto$ marks when new entries, written in blue, are added to a right column, and we do these until there are no columns left.
	
	$\YT{0.17in}{}{
		{{2},{3},{3},{4}},
		{{4},{4},{\overline{4}},{\overline{4}}},
		{{\overline{5}},{\overline{3}},{\overline{2}}},
		{{\overline{4}}},
		{{\overline{3}}}}\stackrel{split}{\rightarrow}
	\YT{0.17in}{}{
		{{1},{2^a},{2^a},{3},{3},{3},{3},{4}},
		{{2},{4},{\overline{4}^b},{\overline{4}},{\overline{4}},{\overline{4}},{\overline{4}},{\overline{3}}},
		{{\overline{5}},{\overline{5}^b},{\overline{3}^c},{\overline{2}},{\overline{2}},{\overline{2}}},
		{{\overline{4}},{\overline{3}^c}},
		{{\overline{3}},{\overline{1}}}}\mapsto
	\YT{0.17in}{}{
		{{1},{2},{2},{3^a},{3^a},{3},{3},{4}},
		{{2},{4},{\overline{4}},{{\color{blue}5}},{\overline{4}^b},{\overline{4}},{\overline{4}},{\overline{3}}},
		{{\overline{5}},{\overline{5}},{\overline{3}},{\overline{4}^b},{\overline{2}^c},{\overline{2}}},
		{{\overline{4}},{\overline{3}},,{\overline{2}^c}},
		{{\overline{3}},{\overline{1}},,{{\color{blue}\overline{1}}}}}$
	
	$\mapsto
	\YT{0.17in}{}{
		{{1},{2},{2},{3},{3},{3^a},{3^a},{4}},
		{{2},{4},{\overline{4}},{{\color{blue}5}},{\overline{4}},{{\color{blue}5}},{\overline{4}^b},{\overline{3}}},
		{{\overline{5}},{\overline{5}},{\overline{3}},{\overline{4}},{\overline{2}},{\overline{4}^b}},
		{{\overline{4}},{\overline{3}},,{\overline{2}},,{\overline{2}}},
		{{\overline{3}},{\overline{1}},,{{\color{blue}\overline{1}}},,{{\color{blue}\overline{1}}}}}
	\mapsto
	\YT{0.17in}{}{
		{{1},{2},{2},{3},{3},{3},{3},{4}},
		{{2},{4},{\overline{4}},{{\color{blue}5}},{\overline{4}},{{\color{blue}5}},{\overline{4}},{{\color{blue}5}}},
		{{\overline{5}},{\overline{5}},{\overline{3}},{\overline{4}},{\overline{2}},{\overline{4}},,{\overline{3}}},
		{{\overline{4}},{\overline{3}},,{\overline{2}},,{\overline{2}},,{{\color{blue}\overline{2}}}},
		{{\overline{3}},{\overline{1}},,{{\color{blue}\overline{1}}},,{{\color{blue}\overline{1}}},,{{\color{blue}\overline{1}}}}},\qquad K_+^1(T)=\YT{0.17in}{}{
		{{4}},
		{{{\color{blue}5}}},
		{{\overline{3}}},
		{{{\color{blue}\overline{2}}}},
		{{{\color{blue}\overline{1}}}}}$;
	\medskip
	
	$\YT{0.17in}{}{
		{{3},{4}},
		{{\overline{4}},{\overline{4}}},
		{{\overline{2}}}}\stackrel{split}{\rightarrow}
	\YT{0.17in}{}{
		{{3},3^a,3^a,{4}},
		{{\overline{4}},\overline{4}^b,{\overline{4}}^b,\overline{3}},
		{{\overline{2}}, \overline{2}}}\mapsto
	\YT{0.17in}{}{
		{{3},3,3,{4}},
		{{\overline{4}},\overline{4},{\overline{4}},\overline{3}},
		{{\overline{2}}, \overline{2},,\color{blue}\overline{2}}},\qquad
 K_+^1\left(\YT{0.17in}{}{
		{{3},{3},{4}},
		{{\overline{4}},{\overline{4}},{\overline{4}}},
		{{\overline{3}},{\overline{2}}}}\right)=K_+^1\left(\YT{0.17in}{}{
		{{3},{4}},
		{{\overline{4}},{\overline{4}}},
		{{\overline{2}}}}\right)=\YT{0.17in}{}{
		{{4}},
		{\overline{3}},
		{\color{blue}\overline{2}}};
	$
	\medskip
	
	$\YT{0.17in}{}{
		{{4}},
		{{\overline{4}}}}\stackrel{split}{\rightarrow}
	\YT{0.17in}{}{
		{3,{4}},
		{ {\overline{4}},\overline{3}}},\qquad K_+^1(\YT{0.17in}{}{
		{{4}},
		{{\overline{4}}}})=\YT{0.17in}{}{
		{{4}},
		{\overline{3}}}$.

	\medskip
	To compute the \emph{left key}, via \emph{direct way}, we need to compute the columns $K_-^1\left(\YT{0.17in}{}{
		{{2},{3},{3},{4}},
		{{4},{\overline{4}},{\overline{4}},{\overline{4}}},
		{{\overline{5}},{\overline{3}},{\overline{2}}},
		{{\overline{4}}},
		{{\overline{3}}}}\right)$,
	$K_-^1\left(\YT{0.17in}{}{
		{{2},{3}},
		{{4},{\overline{4}}},
		{{\overline{5}},{\overline{3}}},
		{{\overline{4}}},
		{{\overline{3}}}}\right)=K_-^1\left(\YT{0.17in}{}{
		{{2},{3},{3}},
		{{4},{\overline{4}},{\overline{4}}},
		{{\overline{5}},{\overline{3}},{\overline{2}}},
		{{\overline{4}}},
		{{\overline{3}}}}\right)$ and $K_-^1\left(\YT{0.17in}{}{
		{{2}},
		{{4}},
		{{\overline{5}}},
		{{\overline{4}}},
		{{\overline{3}}}}\right)$.
	We start by splitting and matching, and every $\mapsto$ marks when entries are removed from a left column, and we do these until there are no columns left. Recall that this algorithm goes from right to left.
	
	$\YT{0.17in}{}{
		{{2},{3},{3},{4}},
		{{4},{\overline{4}},{\overline{4}},{\overline{4}}},
		{{\overline{5}},{\overline{3}},{\overline{2}}},
		{{\overline{4}}},
		{{\overline{3}}}}\stackrel{split}{\rightarrow}
	\YT{0.17in}{}{
		{{1},{2},{2},{3},{3},{3^a},{3^a},{4}},
		{{2},{4},{\overline{4}},{\overline{4}},{\overline{4}},{\overline{4}^b},{\overline{4}^b},{\overline{3}}},
		{{\overline{5}},{\overline{5}},{\overline{3}},{\overline{2}},{\overline{2}},{\overline{2}}},
		{{\overline{4}},{\overline{3}}},
		{{\overline{3}},{\overline{1}}}}\mapsto \YT{0.17in}{}{
		{{1},{2},{2},{3^a},{3^a},{3},{3},{4}},
		{{2},{4},{\overline{4}},{\overline{4}^b},{\overline{4}^b},{\overline{4}},{\overline{4}},{\overline{3}}},
		{{\overline{5}},{\overline{5}},{\overline{3}},{\overline{2}},,{}},
		{{\overline{4}},{\overline{3}}},
		{{\overline{3}},{\overline{1}}}}$
	
	$\mapsto \YT{0.17in}{}{
		{{1},{2^a},{2^a},{3},{3},{3},{3},{4}},
		{{2},{4},{\overline{4}^b},{\overline{4}},{\overline{4}},{\overline{4}},{\overline{4}},{\overline{3}}},
		{{\overline{5}},{\overline{5}^b},{},{},,{}},
		{{\overline{4}},{\overline{3}}},
		{{\overline{3}},{\overline{1}}}}
	\mapsto \YT{0.17in}{}{
		{{},{2},{2},{3},{3},{3},{3},{4}},
		{{2},{},{\overline{4}},{\overline{4}},{\overline{4}},{\overline{4}},{\overline{4}},{\overline{3}}},
		{{\overline{5}},{\overline{5}},{},{\overline{2}},,{\overline{2}}},
		{{},{}},
		{{},{}}},
\qquad K_-^1(T)=\YT{0.17in}{}{
		{{2}},
		{{\overline{5}}}}.$
	
	In the final step, we are removing $\overline{3}$ from $\ell C_1$, because the entry directly Northeast of it is $\overline{5}$, because the $\overline{3}$ of $rC_1$ has already been slid out.
	
	\medskip
	
	$\YT{0.17in}{}{
		{{2},{3}},
		{{4},{\overline{4}}},
		{{\overline{5}},{\overline{3}}},
		{{\overline{4}}},
		{{\overline{3}}}}\stackrel{split}{\rightarrow}
	\YT{0.17in}{}{
		{1,2^a,{2^a},{3}},
		{2,{4},\overline{4}^b,{\overline{4}}},
		{\overline{5},{\overline{5}^b},{\overline{3}^c},\overline{2}},
		{{\overline{4}}, \overline{3}^c},
		{{\overline{3}}, \overline{1}}}\mapsto
	\YT{0.17in}{}{
		{,2,{2},{3}},
		{2,{},\overline{4},{\overline{4}}},
		{\overline{5},{\overline{5}},{\overline{3}},\overline{2}},
		{{}, \overline{3}},
		{{\overline{3}}, }},\qquad
K_-^1\left(\YT{0.17in}{}{
		{{2},{3}},
		{{4},{\overline{4}}},
		{{\overline{5}},{\overline{3}}},
		{{\overline{4}}},
		{{\overline{3}}}}\right)=K_-^1\left(\YT{0.17in}{}{
		{{2},{3},{3}},
		{{4},{\overline{4}},{\overline{4}}},
		{{\overline{5}},{\overline{3}},{\overline{2}}},
		{{\overline{4}}},
		{{\overline{3}}}}\right)=\YT{0.17in}{}{
		{2},
		{\overline{5}},
		{{\overline{3}}}}.
	$
	
	\medskip
	
	$\YT{0.17in}{}{
		{{2}},
		{{4}},
		{{\overline{5}}},
		{{\overline{4}}},
		{{\overline{3}}}}\stackrel{split}{\rightarrow}
	\YT{0.17in}{}{
		{1,2},
		{2,{4}},
		{\overline{5},{\overline{5}}},
		{{\overline{4}}, \overline{3}},
		{{\overline{3}}, \overline{1}}},\qquad
K_-^1\left(\YT{0.17in}{}{
		{{2}},
		{{4}},
		{{\overline{5}}},
		{{\overline{4}}},
		{{\overline{3}}}}\right)=\YT{0.17in}{}{
		{1},
		{2},
		{\overline{5}},
		{{\overline{4}}},
		{{\overline{3}}}}.
	$
	\section{Virtual symplectic keys}\label{sec:virtual}
In this section  we build on \cite{att} and refer to it and \cite{ba00a} for unexplained notation.
The Dynkin diagram of an irreducible root system is a graph whose nodes are in bijection with  the simple roots and simultaneously with the generators of the corresponding Weyl group $W$ as a Coxeter group.  If $\Delta=\{\alpha_i: i\in I\}$ is the set of simple roots of the given root system, the node  $i\in I$ in the Dynkin diagram corresponds to $\alpha_i\in \Delta$ and to the simple reflection $s_i\in W$. When folding the Dynkin diagram $A_{2n-1}$ through the node $n$ one obtains the Dynkin diagram $C_n$, in particular, the hyperoctahedral group $B_n$ embedded as a subgroup of $\mathfrak{S}_{2n}=\langle s^A_i:1\le i<2n\rangle$. That is, $B_n\simeq B_n^A:=\langle \tilde s_i:=s^A_is^A_{2n-i}, \tilde s_n:=s^A_n, 1\le i<n\rangle$ as a subgroup of $\mathfrak{S}_{2n}$. (See also, \cite{bjorner2006combinatorics}.)
More precisely, our construction is based on  the Dynkin diagram folding (or quotient) of $A_{2n-1}$ 
\begin{align}
C_n\;\; \dynkin[edge length=1cm,labels*={1,2,p,n-1,n}] C{***.**}&\longrightarrow& \text{$A_{2n-1}$ \scriptsize folded}\; \;\dynkin[edge length=1.1cm,fold,labels={1,2,p,n-1,n, n+1, 2n-p, 2n-2, 2n-1}]A{***.***.***}
\nonumber
\end{align}
where the nodes of the $C_n$ Dynkin diagram are mapped into  the node-pairs of the $A_{2n-1}$ Dynkin diagram as follows: $i\mapsto\{i,{ 2n-i}\}$, $1\le i<n$ and $n\mapsto n$. For more details, we refer to \cite{bump2017crystal}.

\subsection{Baker embedding} \label{baker}Let $\Lambda_1,\dots,\Lambda_n$ be the $C_n$ fundamental weights, and $\Lambda^A_1,\dots,$ $\Lambda^A_n,$ $\dots,\Lambda^A_{2n-1}$ the $A_{2n-1}$ fundamental weights. We also write $\alpha_i^A$, $1\le i\le 2n-1$, for the  $A_{2n-1}$ simple roots. For $\lambda=\Lambda_{m_1}+\cdots+\Lambda_{m_k}\in\mathbb{Z}^n$, with $1\le m_1\le\cdots\le m_k\le n$, let
\begin{align}\label{lambdaA}
\lambda^A=\Lambda^A_{2n-m_1}+\Lambda^A_{m_1}+\cdots+\Lambda^A_{2n-m_k}+\Lambda^A_{m_k}, 
\end{align}
be a $\mathbb{Z}^{2n}$ partition with at most $2n-1$ parts, and $\textsf{SSYT}(\lambda^A,2n)$ the $A_{2n-1}$ crystal of semistandard tableaux of shape $\lambda^A$ in the alphabet $[2n]$. We consider
inside  the $A_{2n-1}$ crystal $\textsf{SSYT}(\lambda^A,2n)$ the crystal generated  by acting with lowering operators $f^{E}_{i} := f^{A}_{i}    f^{A}_{2n-i}=   f^{A}_{2n-i}f^{A}_{i} $, $1\le i< n$, and $f^{E}_{n} =(f^{A}_{n})^{2} $ on the highest weight element $K(\lambda^A)$ of $\textsf{SSYT}(\lambda^A, 2n)$,  where $f^{A}_{i}$ denote the $A_{2n-1}$ crystal operators acting on $\textsf{SSYT}(\lambda^A,2n)$. This is an embedding of the crystal  $\textsf{KN}(\lambda,n)$ into the crystal $\textsf{SSYT}(\lambda^A, 2n)$, called Baker embedding \cite{ba00a}, $E:\textsf{KN}(\lambda,n)\hookrightarrow \textsf{SSYT}(\lambda^A, 2n)$,
 \begin{align}\label{def:virtualcrystal}& \textsf{KN}(\lambda,n))=
\{{f}_{i_{1}}^{k_{1}}\cdots{f}%
_{i_{\ell}}^{k_{\ell}}(K({\lambda}))\mid i_1,\dots,i_{\ell}\in [n],\; (k_{1},\ldots,k_{\ell})\in\mathbb{Z}_{\geq
0}^{\ell},\; \ell\ge 0\}\setminus\{0\}, \;\mbox{and}\nonumber\\
&E(\textsf{KN}(\lambda,n))=\nonumber\\
&=\{{f_{i_{1}}^A}^{k_{1}}{f^A}_{2n-i_{1}}^{k_{1}}\cdots{f^A}%
_{i_{\ell}}^{k_{\ell}}{f^A}%
_{2n-i_{\ell}}^{k_{\ell}}(K({\lambda^A}))\mid i_1,\dots,i_{\ell}\in [n],\; (k_{1},\ldots,k_{\ell})\in\mathbb{Z}_{\geq
0}^{\ell},\; \ell\ge 0\}\setminus\{0\}.
\end{align}
The  \emph{virtual crystal} $E(\textsf{K}(\lambda,n))$  models the crystal $\textsf{K}(\lambda,n)$ inside  the $A_{2n-1}$ crystal $\textsf{SSYT}(\lambda^A,2n)$.

It is well known that given two semistandard tableaux $R$ and $S$  the column Schensted  insertion \cite{fulton1997young} $[R\leftarrow S]:=[R\leftarrow a_1\leftarrow a_2\leftarrow\cdots \leftarrow a_t]$ with $w(S)=a_1a_2\cdots a_s$ the  Chinese/Japanese  reading of $S$, defines a crystal isomorphism between the crystal connected component containing $R\otimes S$ and respectively the crystal containing $[R\leftarrow S]$. Thus  we identify $R\otimes S=$ $[R\leftarrow S]$ $=[\emptyset\leftarrow R\leftarrow S]$ \cite{kwon}. In particular $K(\lambda)\otimes K(\mu)=K(\lambda+\mu)$ the highest weight element of $\mathfrak{B}(\lambda+\mu)\subseteq \mathfrak{B}(\lambda)\otimes
\mathfrak{B}(\mu)$ for $\lambda, \mu\in \mathcal{P}_{2n}$.

Let $\psi:\textsf{KN}(\Lambda_i,n)\hookrightarrow \textsf{SSYT}(\Lambda^A_{2n-i}+\Lambda^A_{i},2n)\subseteq \mathfrak{B}(\Lambda^A_i,2n)\otimes \mathfrak{B}(\Lambda^A_{2n-i},2n)$, $1\le i\le n$, be the crystal embedding  explicitly defined in  \cite[Proposition 2.2]{ba00a}. In particular, $\psi(K(\Lambda_i))=K(\Lambda^A_{2n-i}+\Lambda^A_{i})$, $1\le i\le n$. The embedding of $\textsf{KN}(\Lambda_i,n)$ in $\textsf{SSYT}(\Lambda^A_{2n-i}+\Lambda^A_{i},2n)$ is generated in $\textsf{SSYT}(\lambda^A,2n)$ by acting with lowering operators
\begin{align}f^{E}_{i} :=\begin{cases} f^{A}_{i}    f^{A}_{2n-i}=f^{A}_{2n-i}f^{A}_{i}, 1\le i< n,\\
 (f^{A}_{n})^{2}, i=n.
 \end{cases}\label{fE}
 \end{align}
  on the highest weight element $K(\Lambda^A_{2n-i}+\Lambda^A_{i})$ of $\textsf{SSYT}(\Lambda^A_{2n-i}+\Lambda^A_{i}, 2n)$. That is, $\psi(\textsf{KN}(\Lambda_i,n)$ is realized in the crystal connected component of $\mathfrak{B}(\Lambda^A_i,2n)\otimes \mathfrak{B}(\Lambda^A_{2n-i},2n)$ with highest weight element $K(\Lambda^A_{2n-i}+\Lambda^A_{i})$ by acting successively by lowering operators $f^{E}_{i}$ on $K(\Lambda^A_{2n-i}+\Lambda^A_{i})$.

For $\lambda\in \mathcal{P}_n$  and $\lambda^A$ in \eqref{lambdaA}, the Baker virtualization \cite[Proposition 2.3]{ba00a} is the crystal embedding defined by the injective map
\begin{align*}
 {\textsf{E}}: \textsf{KN}(\lambda,n) &\lhook\joinrel\longrightarrow \textsf{SSYT}(\lambda^A,2n)\subseteq \bigotimes_{i=1}^k\mathfrak{B}(\Lambda^A_i+\Lambda^A_{2n-i},2n) \\
 T=C_1C_2\cdots C_k&\mapsto{\textsf{E}}(T)= \psi(C_k)\otimes\cdots\otimes \psi(C_1)\\
 &\qquad\qquad
=[\emptyset\leftarrow w(\psi(C_k)) \leftarrow \cdots \leftarrow w(\psi(C_1))]
 \end{align*}
 where 
 $ w(\psi(C_i))$ is  the Chinese/Japanese  reading of the two column semistandard tableau $\psi(C_i)$.  Then
$K(\lambda^A)=\psi(K(\Lambda_{m_1}))\otimes\cdots\otimes \psi(K(\Lambda_{m_k}))=E(K(\lambda)).$
   Note that according to our conventions $T=C_1\cdots C_k=C_k\otimes\cdots\otimes C_1\in \textsf{KN}(\lambda,n)$ and we have
\begin{align}{\textsf{E}}(C_k\otimes\cdots\otimes C_1)= \psi(C_k)\otimes\cdots\otimes \psi(C_1).\label{tensor}
\end{align}
The crystal $E(\textsf{KN}(\lambda,n))$ is realized by the crystal connected component of  $\bigotimes_{i=1}^k\mathfrak{B}(\Lambda^A_i+\Lambda^A_{2n-i},2n)$ with highest weight element $K(\lambda^A )$ by acting with  lowering operators $f^{E}_{i}$ on $K(\lambda^A)$.

One has,
${\textsf{E}}f_{i}(T) =f^{E}_{i}{\textsf{E}}(T)$, for $T \in   \textsf{KN}(\lambda,n)$, $1\le i\le n$,
and length function
\begin{align}\varphi_i(T)=\begin{cases}\varphi_i^A (E(T))=\varphi_{2n-i}^A(E(T)),& 1\le i<n\\
1/2\varphi_n ^A(E(T)),& i=n.
    \end{cases}\label{length}
    \end{align}
similarly for $\varepsilon_i(T)$ where $\varphi_i^A, \varepsilon^A_i$ denote the length functions in $\textsf{SSYT}(\lambda^A,2n)$.
From \eqref{length}, the following observations follow.

\begin{obs}\label{basics}\begin{enumerate}
\item[(a)]
For $1\le i<n$, \begin{align*}\langle \lambda,\alpha_i\rangle=\lambda_i-\lambda_{i+1}&=\varphi_i (K(\lambda))\nonumber\\ 
    &=\varphi_i^A (K(\lambda^A))=\langle \lambda^A,\alpha_i\rangle=
    \lambda_i^A-\lambda_{i+1}^A\nonumber\\
    &=\varphi_{2n-i}^A(K(\lambda^A))=\langle \lambda^A,\alpha_{2n-i}\rangle=\lambda_{2n-i}^A-\lambda^A_{2n-i+1}.\label{lambdai}
    \end{align*}
    For $i=n$,
 \begin{equation*}\lambda^A_n-\lambda^A_{n+1}=\langle \lambda^A,\alpha_n\rangle=\varphi_n^A (K(\lambda^A))=2\varphi_n(K(\lambda))=2\langle \lambda,e_n\rangle=2\lambda_n=2\langle \lambda,\alpha^\vee_n\rangle.\label{lambdanwhale}
\end{equation*}
Therefore the stabilizer of $\lambda\in \mathcal{P}_n$ under the action of $B_n$  is identified with the stabilizer of $\lambda^A\in \mathcal{P}_{2n-1}$ under the action of $B_n^A$ ($\subseteq \mathfrak{S}_{2n}$), and $B_n^\lambda=B^{A,\lambda^A}_n$.

\item[(b)] For $u\in \mathfrak{S}_{2n}\lambda^A$, and $1\le i<n$,
\begin{equation*}\varphi_{i}^A(K(u))=\langle u,\alpha_i\rangle=\varphi_{i}^A(K(s^A_{2n-i}u))=\langle s^A_{2n-i}u,\alpha_i\rangle,\label{snale}
\end{equation*}
\begin{equation*}
\varphi_{2n-i}^A(K(u))=\langle u,\alpha_{2n-i}\rangle=\varphi_{2n-i}^A(K(s^A_iu))=\langle s^A_iu,\alpha_{2n-i}\rangle.\label{snale1}
\end{equation*}
\begin{align*}u<s^A_i u\Leftrightarrow s^A_{2n-i} u< s^A_is^A_{2n-i}u\\
u<s^A_{2n-i} u\Leftrightarrow s^A_{i} u< s^A_{2n-i}s^A_{i}u
\end{align*}

\end{enumerate}
\end{obs}
{ Let $O_{B_n}(\lambda^A)=\{K(u^A): u^A\in B^A_n\lambda^A\}$ be the $B_n$-orbit of $\lambda^A$,} where $B_n$ is identified with $\langle s^A_is^A_{2n-i}, s^A_n, 1\le i<n\rangle\subseteq \mathfrak{S}_{2n}$, and
  $O_{\mathfrak{S}_{2n}}(\lambda^A)=\{K(u): \mu\in \mathfrak{S}_{2n}\lambda^A\}$ the $\mathfrak{S}_{2n}
  $-orbit of $\lambda^A$.
 Indeed
$$O_{B_n}(\lambda^A)\subseteq O_{\mathfrak{S}_{2n}}(\lambda^A).$$

\begin{obs}\label{maxs} Recall,  for any $1\le i\le 2n-1$, ${f_{i}^A}^{\textrm{max}}(T):={f_{i}^A}^{\varphi_i^A(T)}(T)$  for $T\in \textsf{SSYT}(\lambda^A,2n)$. For $T\in \textsf{KN}(\lambda,n)$, from \eqref{fE}, \eqref{length}, \begin{align*}E({f_{i}}^{\textrm{max}}(T))&=E({f_{i}}^{\varphi_i(T)}(T))={f^A_{2n-i}}^{\varphi_i(T)}{f^A_{i}}^{\varphi_i(T)}(E(T))\\
&={f^A_{2n-i}}^{\varphi^A_{2n-i}(E(T))}{f^A_{i}}^{\varphi^A_i(E(T))}(E(T))={f^A_{i}}^{\textrm{max}}{{f^A}^{\textrm{max}}_{2n-i}}(E(T)), \;\mbox{for $1\le i<n$},
\end{align*}
and
\begin{align*}E({f_{n}}^{\textrm{max}}(T))&=E({f_{n}}^{\varphi_n(T)}(T))={f^A_{n}}^{2\varphi_n(T)}\\
&={f^A_{n}}^{\varphi^A_n(E(T))}(E(T))={f^A_{n}}^{\textrm{max}}(E(T)).
\end{align*}
Therefore, for $T\in \textsf{KN}(\lambda,n)$, we define \begin{align*}{f_i^E}^{\textrm{max}}(E(T))=\begin{cases}{f^A_{i}}^{\textrm{max}}{{f^A}^{\textrm{max}}_{2n-i}}(E(T)),&{\textrm{max}=
 \varphi_i(T)=
\varphi^A_{2n-i}(E(T))=\varphi^A_i(E(T))}, 1\le i<n\\
{f^A_{n}}^{\textrm{max}}(E(T)),&{\textrm{max}=\varphi^A_n(E(T))=2\varphi_n(T)}, i=n,
\end{cases}
\end{align*}
and $E({f_{i}}^{\textrm{max}}(T))={f_i^E}^{\textrm{max}}(E(T)),$ $1\le i\le n$.
\end{obs}

\subsection{ Baker embedding of symplectic keys into $\mathfrak{gl}_{2n}$ keys: virtual keys and virtual Demazure crystals}

We show that the injection $E$ embeds the keys of $\textsf{KN}(\lambda,n)$, that is, the keys in $O(\lambda)$ into the keys in $O_{\mathfrak{S}_{2n}}(\lambda^A)$ restricted to $O_{B_n^A}(\lambda^A)$. More precisely, $E(O(\lambda))=O_{B^A_n}(\lambda^A)$. However the injection $E$ does not need to send a   Demazure crystal of $\textsf{KN}(\lambda,n)$ to a Demazure crystal of $\textsf{SSYT}(\lambda^A,2n)$ parameterized in $B_n^A\lambda^A$, instead $E$ embeds it inside to a such Demazure crystal of $\textsf{SSYT}(\lambda^A,2n)$ . More precisely,
for  $ v\in B_n\lambda$, $E(\mathfrak{B}_v)=\mathfrak{B}_{\tilde v}\cap \textsf{KN}(\lambda,n)$,  where the map $v=s_{i_r}\cdots s_{i_1}\lambda\mapsto \tilde v=\tilde s_{i_r}\cdots \tilde s_{i_1}\lambda^A$ with $s_{i_r}\cdots s_{i_1}\in W^\lambda=B_n^\lambda$ defines a bijection between $B_n\lambda$ and $B_n^A\lambda^A$.
\begin{lema}\label{basics2}
\begin{enumerate}
\item For $1\le j<  n$, the following are equivalent
\begin{itemize}
\item  $\lambda<s_j \lambda$  in $B_n\lambda$.
\item $\lambda^A<s^A_j \lambda^A$ in ${\mathfrak{S}_{2n}}\lambda^A$.
\item $\lambda^A<s^A_{2n-j}\lambda^A$ in ${\mathfrak{S}_{2n}}\lambda^A$.
\item $\lambda^A<s^A_j \lambda^A<s^A_{2n-j}s^A_j \lambda^A  $
\item $\lambda^A<s^A_{2n-j} \lambda^A<s^A_js^A_{2n-j} \lambda^A  $
in ${\mathfrak{S}_{2n}}\lambda^A$.

    \item $\lambda^A<\tilde s_j\lambda^A$  in  $B^A_n\lambda^A$.
    \item $f_j^{\textrm{max}}(K(\lambda))=K(s_j\lambda),\; \textrm{max}=\varphi_j(K(\lambda))=\langle \lambda,\alpha_j\rangle>0$.
\item  ${f_j^E}^{\textrm{max}}(K(\lambda^A))=K(\tilde s_j\lambda^A)$,
   $\textrm{max}=$ $\varphi_j(K(\lambda))=$ $
\varphi^A_{2n-j}(K(\lambda^A))$ $=\varphi^A_j(K(\lambda^A))=$ $\langle \lambda^A,\alpha_j\rangle=$ $\langle \lambda^A,\alpha_{2n-j}\rangle>0$.
\end{itemize}
\item For $j=n$,
\begin{align*} \lambda<s_n\lambda \;\mbox{in $B_n\lambda$}&\Leftrightarrow \lambda^A<\tilde s_n\lambda^A \; \mbox{in ${\mathfrak{S}_{2n}}\lambda^A$}\\
&\Leftrightarrow f_{n}^{\textrm{max}}(K(\lambda))=K( s_n\lambda),\; \textrm{max}=\varphi_n(K(\lambda))
>0\\
&\Leftrightarrow {f_n^E}^{\textrm{max}}(K(\lambda^A))=K(\tilde s_n\lambda^A),\; \textrm{max}=\varphi^A_n(K(\lambda^A))=2\varphi_n(K(\lambda))>0.
\end{align*}
\end{enumerate}
\end{lema}
\begin{proof} $(a)$ Let $1\le j<n$.  From \eqref{length}, Proposition \ref{keystring} and  Remark \ref{basics},
\begin{align*}\lambda<s_j \lambda \; \mbox{ in $B_n\lambda$}&\Leftrightarrow 0<\langle \lambda,\alpha_j\rangle=\varphi_j (K(\lambda))
    =\varphi_j^A (K(\lambda^A))=\langle \lambda^A,\alpha_j\rangle>0\\
    &\Leftrightarrow \lambda^A<s_j \lambda^A
    \end{align*}
    and
    \begin{align*}
    & \lambda^A<s_j \lambda^A\Leftrightarrow 0<\langle \lambda^A,\alpha_j\rangle=\varphi_j^A (K(\lambda^A))=\varphi_{2n-j}^A(K(\lambda^A))=\langle \lambda^A,\alpha_{2n-j}\rangle>0
  \end{align*}
    \begin{align*}
    &\Leftrightarrow \lambda^A<s^A_{2n-j} \lambda^A \\
    &\Leftrightarrow \lambda^A<s^A_{2n-j} \lambda^A<s^A_{2n-j}s^A_j \lambda^A \\
    &\Leftrightarrow \lambda^A<s^A_{j} \lambda^A<s^A_{j}s^A_{2n-j} \lambda^A.
     \end{align*}
     We prove now that $\lambda^A<s^A_{2n-j} \lambda^A<s^A_js^A_{2n-j}\lambda^A \Leftrightarrow \lambda^A<\tilde s_j\lambda^A$. It is enough to prove that $\lambda^A<\tilde s_j\lambda^A\Rightarrow \lambda<s_j \lambda$.
     If $\lambda\nless s_j \lambda$ since $\lambda$ is a partition,
     \begin{align*}\lambda=s_j\lambda&\Leftrightarrow \varphi_j(K(\lambda))=\varphi_j^A(K(\lambda^A))=\varphi_{2n-j}^A(K(\lambda^A))=0\\
     &\Leftrightarrow \langle \lambda^A,\alpha_j\rangle= \langle \lambda^A,\alpha_{2n-j}\rangle=0
     \end{align*}
     \begin{align*}
     &\Leftrightarrow\lambda^A=s_j\lambda^A, \; \lambda^A=s_{2n-j}\lambda^A\Leftrightarrow \lambda^A\nless s_j \lambda^A, \; \lambda^A\nless s_{2n-j} \lambda^A.
      \end{align*}

      Next
      \begin{align*} &\lambda<s_j\lambda\; \mbox{in $B_n\lambda$}\Leftrightarrow\\
 &\Leftrightarrow f_j^{\textrm{max}}(K(\lambda))=K(s_j\lambda),\; {\textrm{max}=\varphi_j(T)>0},\; \mbox{ by Lemma \ref{keystring}}\\
&\Leftrightarrow E(f_j^{\textrm{max}}(K(\lambda))=E(K(s_j\lambda)),\; {\textrm{max}=\varphi_j(T)>0},
\mbox{ since $E$ is an injection}
\end{align*}
\begin{align*}
 &\Leftrightarrow{f^E_{j}}^{\textrm{max}}(K(\lambda^A))={f^A_{j}}^{\textrm{max}}{{f^A}^{\textrm{max}}_{2n-j}}(K(\lambda^A)) = E(K(s_j\lambda)),\\
 &\mbox{ with } {\textrm{max}=\varphi_j(K(\lambda))}=
{\varphi^A_{2n-j}(K(\lambda^A))=\varphi^A_j(K(\lambda^A))>0},\\
&={{f_{2n-j}^A}^{\textrm{max}}}(K(s_j\lambda^A))=K(s_{2n-j}s_j\lambda^A)=K(\tilde s_j\lambda^A)=E(K(s_j\lambda)), \; 1\le j<n,
\end{align*}

$(b)$ Let $j=n$,
\begin{align*}
&\lambda<s_n\lambda \;\mbox{in $B_n\lambda$} \Leftrightarrow f_{n}^{\textrm{max}}(K(\lambda))=K( s_n\lambda),\; \textrm{max}=\varphi_n(K(\lambda))
>0
\end{align*}
\begin{align*}
&\Leftrightarrow {f^E_{n}}^{\textrm{max}}(K(\lambda^A)={f^A_{n}}^{\textrm{max}}(K(\lambda^A)=K(\tilde s_n\lambda^A)=E(K(s_n\lambda)),\\
&\mbox{ with } {\textrm{max}=\varphi^A_n(K(\lambda^A))=2\varphi_n(K(\lambda))>0},\;  j=n,
 \mbox{by Remark \ref{maxs}} \\
 &={f^A_{n}}^{\textrm{max}}(K(\lambda^A))=K(\tilde s_n\lambda^A)=E(K(s_n\lambda)),\;\textrm{max}>0,
  \mbox{ by Lemma \ref{keystring} and Remark \ref{basics}}
\end{align*}
\begin{align*}
 &\Leftrightarrow\lambda^A<\tilde s_n \lambda^A \;\mbox{in $B_n\lambda^A$}. 
\end{align*}
\end{proof}

The  weak Bruhat  graph defined by the vertices of $O(\lambda)=\{K(u): u\in B_n\lambda\}$, in bijection with $B_n^\lambda$, and  with edges the $i$-strings connecting them in $\textsf{KN}(\lambda, n)$ is embedded in $O_{\mathfrak{S}_{2n}}(\lambda^A)$. For $1\le i\le n$ and $u\in B_n\lambda$, there is a non trivial  $i$-chain connecting $K(u)$ to $K(s_iu)$ in $\textsf{KN}(\lambda,n)$ if and only if $u<s_iu$ in $B_n\lambda$, equivalently $f_i^{\textrm{max}}(K(u))= K(s_iu)$ with $\textrm{max}=\varphi_i(K(u))=\left\langle u,\alpha_i^\vee\right\rangle >0$.
More precisely,
\[O(\lambda)=\{{f^{\textrm{max}}_{i_r}}\cdots {f^{\textrm{max}}_{i_1}}(K(\lambda))\mid i_1,\dots, i_r \in [n], ~r\ge 0\}.\]
 We prove that the keys in $\textsf{KN}(\lambda,n)$ are embedded by $E$  into keys in $\textsf{SSYT}(\lambda^A,2n)$. One has $K(\lambda)\in O(\lambda)$ and $E(K(\lambda))=K(\lambda^A)\in O_{B_n}(\lambda^A)$. More generally,  we claim that, $E(O(\lambda))=O_{B_n}(\lambda^A)\hookrightarrow O_{\mathfrak{S}_{2n}}(\lambda^A)$ such that, for $u\in B_n\lambda$,   whenever $\lambda<s_{i_1}\lambda<\cdots<s_{i_r}\cdots s_{i_1}\lambda=u$ for some $s_{i_1},\dots,s_{i_r}\in B_n$, then
 $K(u)$ $\mapsto K(u^A)$ with $\lambda^A<\tilde s_{i_1}\lambda^A<\cdots<\tilde s_{i_r}\cdots \tilde s_{i_1}\lambda^A=u^A$.

\begin{prop}
\begin{enumerate} \label{virtualkeyy}
\item  Let $v\in B_n\lambda$ and  $s_{i_1},\dots,s_{i_r}\in B_n$, $r\ge 0$, such that $s_{i_r}\cdots s_{i_1}\lambda=v$. Then the following are equivalent
\begin{itemize}
\item $\lambda<s_{i_1}\lambda<\cdots<s_{i_r}\cdots s_{i_1}\lambda=v$.
\item $E(
K(s_{i_r}\cdots s_{i_1}\lambda))={
f_{i_r}^E}^{\textrm{max}}\cdots {f^E_{i_1}}^{\textrm{max}}(K(\lambda^A))=K( \tilde s \lambda^A)\in  O_{B_n}(\lambda^A)$, with $\tilde s=\tilde s_{i_r}\cdots \tilde s_{i_1}$ and every $max>0$.
\item $ \lambda^A<\tilde s_{i_1}\lambda^A<\cdots<\tilde s_{i_r}\cdots \tilde s_{i_1}\lambda^A \in B^A_n\lambda^A$.
\end{itemize}
\item $\label{O0}O_{B_n}(\lambda^A)
=\{{f^E_{i_r}}^{\textrm{max}}\cdots {f^E_{i_1}}^{\textrm{max}}(K(\lambda^A))\mid \mbox{  $ i_1,\dots, i_r \in [n]$},\; r\ge 0 \}
=E(O(\lambda)).$

The set $ O_{B_n}(\lambda^A)$ is called the set of \emph{virtual keys} of $\textsf{KN}(\lambda,n)$.
\end{enumerate}
  \end{prop}

\begin{proof} $(a)$ By induction on $r\ge 0$ with base case   $r=0$ and, by  Lemma \ref{basics2},  $r=1$. Let $r>1$ and  suppose the statement true for $r-1$. Consider $ s_{i_{r-1}}\cdots  s_{i_1}\lambda=:u\in B_n\lambda$ and $\tilde s_{i_{r-1}}\cdots \tilde s_{i_1}\lambda^A=:\tilde u\in B_n^A\lambda^A$.
One has
\begin{align*}u<v=s_{i_r}u\;\mbox{ in $B_n\lambda$}\Leftrightarrow K(v)&={f^{\textrm{max}}_{i_r}}K(u), \;\textrm{max}>0,\;\mbox{by
Proposition \ref{keystring}} \\
&\Leftrightarrow E(K(v))=E({f^{\textrm{max}}_{i_r}}K(u)), \;\textrm{max}>0\\
&\Leftrightarrow E(K(v))={f^E_{i_r}}^{\textrm{max}}E(K(u)), \;\textrm{max}>0\\
\end{align*}
\begin{align*}
&\Leftrightarrow E(K(v))={f^E_{i_r}}^{\textrm{max}}{f^E_{i_{r-1}}}^{\textrm{max}}\cdots {f^E_{i_1}}^{\textrm{max}}(K(\lambda^A)), \;\mbox{all $\textrm{max}>0$},
\mbox{by induction}\\
\end{align*}
\begin{align*}
&\Leftrightarrow E(K(v))={f^E_{i_r}}^{\textrm{max}}K( \tilde s_{i_{r-1}}\cdots \tilde s_{i_1} \lambda^A),  \;\mbox{ $\textrm{max}>0$}\\
\end{align*}
\begin{align*}
&\Leftrightarrow E(K(v))=K(\tilde s_{i_{r}} \tilde s_{i_{r-1}}\cdots \tilde s_{i_1} \lambda^A)\\
&\Leftrightarrow \tilde s_{i_{r-1}}\cdots \tilde s_{i_1}\lambda^A < \tilde s_{i_r}\cdots \tilde s_{i_1}\lambda^A \in B^A_n\lambda^A,\;\mbox{by Proposition \ref{keystring}}\\
&\Leftrightarrow \tilde u<\tilde v=\tilde s_{i_r}\tilde u.
\end{align*}
$(b)$ Follows from $(a)$.
\end{proof}

 \begin{obs}$E(\textsf{KN}(\lambda,n))\subseteq \textsf{SSYT}(\lambda^A, 2n)$ and  $f_i^{\textrm{max}}(K(u))= K(s_i\mu)$, with $ u <s_i u$  if and only if
${f^E_i}^{\textrm{max}}(K(u))= K(\tilde s_i u^A)$, with $u^A<\tilde s_i u^A$. This means that there is a $i$-chain of length $\langle u,\alpha_i\rangle>0$ in $\textsf{KN}(\lambda,n)$ connecting $K(u)$ and $K(s_iu)$ if and only if there is a diamond made of two double chains $i,2n-i$ and $  2n-i,i$, each edge of length $\langle u,\alpha_i\rangle>0$,
 connecting $E(K(u))=K(u^A)$ and $\tilde s_i  K(u^A)$ in $\textsf{SSYT}(\lambda^A, 2n)$.
There is  a $n$-chain of length $\langle u,\alpha^\vee_n\rangle>0$  connecting $K(u)$ and $K(s_nu)$ in $\textsf{KN}(\lambda,n)$ if and only if there is a $n,n$-double chain of length $2\langle u,\alpha^\vee_n\rangle>0$ connecting $K(u^A)$ and
 $\tilde s_n K( u^A)=E({f_n}^{\textrm{max}}(K(u)))=(f^{A}_{n})^{\langle u,\alpha^\vee_n\rangle}    (f^{A}_{n})^{\langle u,\alpha^\vee_n\rangle}K(u^A)$ in $\textsf{SSYT}(\lambda^A, 2n)$. See Example \ref{KN(1,0),2}.
 \end{obs}

 Let $s_{i_r}\cdots s_{i_1}$ be a reduced word in  $B_n^\lambda$,   $v=s_{i_r}\cdots s_{i_1}\lambda\in B_n\lambda$,  and $\tilde v=\tilde s_{i_r}\cdots \tilde s_{i_1}\lambda^A\in B_n\lambda^A$, where $\tilde s_{i_r}\cdots \tilde s_{i_1}$ is also a reduced word in $B_n^{A,\lambda^A}$ (and also in $\mathfrak{S}_{2n}^{\lambda^A}$). Then
 \begin{align}
		\mathfrak{B}_v
&=\{f_{i_{r}}^{k_{r}}\cdots f_{i_1}^{k_{1}}(K(\lambda))\mid(k_{r},\ldots,k_{1})\in\mathbb{Z}_{\geq
0}^{r}\}\setminus\{0\}\subseteq \textsf{KN}(\lambda,n)\label{demaz},
	\end{align}
and
\begin{align*}
		E(\mathfrak{B}_v)&
=\{(f^E_{i_{r}})^{k_{r}}\cdots (f^E_{i_1})^{k_{1}}(K(\lambda^A))\mid(k_{r},\ldots,k_{1})\in\mathbb{Z}_{\geq
0}^{\ell}\}\setminus\{0\}\\
&=\{(f^A_{i_{r}})^{k_{r}}(f^A_{2n-i_{r}})^{k_{r}}\cdots (f_{i_1})^{k_{1}}(f_{2n-i_1})^{k_{1}}(K(\lambda^A))\mid(k_{r},\ldots,k_{1})\in\mathbb{Z}_{\geq
0}^{\ell}\}\setminus\{0\}\\
&\subseteq \{(f^A_{i_{r}})^{k_{r}}(f^A_{2n-i_{r}})^{k'_{r}}\cdots (f_{i_1})^{k_{1}}(f_{2n-i_1})^{k'_{1}}(K(\lambda^A))\mid(k_{r},k'_{r}\ldots,k_{1},k'_{1})\in\mathbb{Z}_{\geq
0}^{\ell}\}\setminus\{0\}\\
&=\mathfrak{B}^A_{\tilde v}\subseteq\textsf{SSYT}(\lambda^A, 2n)
\mbox{  where $\tilde v=\tilde s_{i_r}\cdots \tilde s_{i_1}\lambda^A\in B_n\lambda^A$}.
	\end{align*}
For each $v\in B_n\lambda$,   the injection $E$ embeds  the Demazure cystal $\mathfrak{B}_v \subseteq\textsf{KN}(\lambda,n)$ into the Demazure crystal  $\mathfrak{B}^A_{\tilde v} \subseteq\textsf{SSYT}(\lambda^A, 2n)$ with $\tilde v \in B_n\lambda^A$. Henceforth, recalling \eqref{def:virtualcrystal}, we have the following assertion.

\begin{prop}\label{prop:virtualdemazure} With the setup above,
$E(\mathfrak{B}_v)=\mathfrak{B}^A_{\tilde v}\cap E(\textsf{KN}(\lambda,n))$.
\end{prop}
The \emph{virtual Demazure crystals} of $\textsf{KN}(\lambda,n)$ are  $E(\mathfrak{B}_v)=\mathfrak{B}^A_{\tilde v}\cap E(\textsf{KN}(\lambda,n))$, for all $\tilde v \in B_n\lambda^A$. Note $\mathfrak{B}^A_{s_r\cdots s_1\lambda^A}, \mathfrak{B}^A_{s_{2n-r}\cdots s_{2n-1}\lambda^A} \subseteq \mathfrak{B}^A_{\tilde v}$ as Demazure crystals of $\textsf{SSYT}(\lambda^A, 2n)$ whereas this containment does not happen in $E(\textsf{KN}(\lambda,n))$.

\subsection{ {Baker} embedding and crystal dilatation  commute}
Recall \S\ \ref{subsecdilatation} and Proposition \ref{Th_Dila}, (\cite[Proposition 8.3.2]{kash2}).  Given   a positive integer $m$, there is a  unique  embedding of (abstract) crystals $\mathfrak{B}(\lambda)\hookrightarrow \mathfrak{B}(m\lambda)$
such that for any vertex $b\in \mathfrak{B}(\lambda)$ and any path $b={f}_{i_{1}%
}\cdots{f}_{i_{l}}(b_{\lambda})$ in $\mathfrak{B}(\lambda)$
\begin{equation}\theta_m(b)=f_{i_l}^m\cdots f_{i_1}^m(b_{m\lambda}).\label{theta}\end{equation}
In particular, for $\sigma\in W^\lambda$,  $\theta_m(b_{\sigma \lambda})=b_{\sigma m\lambda}$, that is, $\theta_m(O(\lambda))=O(m\lambda)$.
Furthermore, $\varphi_i(\theta_m(b))=m\varphi_i(b)$, $\varepsilon_i(\theta_m(b))=m\varphi_i(b)$ and $\textsf{wt}(\theta_m(b))=m \,\textsf{wt}(b)$, for $i\in I$. This embedding is called {\em $m$-dilatation map} . Since $b_{m\lambda}$ and $b_\lambda^{\otimes m}$ have the same weight $m\lambda$, we may realize this dilatation $\theta_m$ in a canonical way as a crystal embedding of $\mathfrak{B}(\lambda)$ into the connected component of $\mathfrak{B}({\lambda
})^{\otimes m}$ of highest weight $m\lambda$,

\begin{equation}
\theta_{m}:\left\{
\begin{array}
[c]{c}%
\mathfrak{B}(b_{\lambda})\hookrightarrow \mathfrak{B}(b_{\lambda}^{\otimes m})\subset \mathfrak{B}(b_{\lambda
})^{\otimes m}\\
b\longmapsto b_{1}\otimes\cdots\otimes b_{m}%
\end{array}
\right.  %
\end{equation}

When $m$ has enough factors (in general, $m$ is equal to  the least common multiple of the maximal
lengths of the i-chains in $ \mathfrak{B}(\lambda)$), the $m$-dilatation map on $\mathfrak{B}(b_{\lambda})$ transform the vertices into a tensor product of $m$ keys in $O(\lambda)$,

\begin{equation}
\theta_{m}:\left\{
\begin{array}
[c]{c}%
\mathfrak{B}({\lambda})\hookrightarrow \mathfrak{B}(K({\lambda})^{\otimes m})\subset \mathfrak{B}({\lambda
})^{\otimes m}\\
b\longmapsto K(\sigma_1\lambda)\otimes\cdots\otimes K(\sigma_{m}\lambda)%
\end{array}
\right.  \label{embededd}%
\end{equation}
where $\sigma_{1}\ge \ldots\ge \sigma_{m}$ in $W^{\lambda}$, $K^+(b)=K(\sigma_1\lambda)$ and $K^-(b)=K(\sigma_m \lambda)$.
It is important, to recall that the elements $K(\sigma_{1}\lambda)$ and $K(\sigma_{m}\lambda)$ in
$\theta_{m}(T)$ do not then depend on $m$, once we get one $m$ that sends, via $\theta_m$,  all $T\in \mathfrak{B}(\lambda)$ to a tensor product of keys in $\mathfrak{B}(\lambda)$.

That is, the dilated crystal $\theta_m(\mathfrak{B}({\lambda}))$ is  generated in $\mathfrak{B}({m\lambda})$ (equivalently in  $\mathfrak{B}({K(\lambda)}^{\otimes m}$) by acting successively with $f_i^{m}$ on $K(m\lambda)$ (equivalently ${K(\lambda)}^{\otimes m}$) for $i\in I$,
\[\theta_m(\mathfrak{B}(\lambda)=\{f_{i_l}^{m }\cdots f_{i_1}^{m }(K(m\lambda)):({i_l},\dots, {i_1})\in [n]^l, l\ge 0\}\setminus\{0\}\subseteq \mathfrak{B}(m\lambda,n),\]
and the generated vertices in $\mathfrak{B}(K({\lambda})^{\otimes m})\subseteq \mathfrak{B}({K(\lambda
)})^{\otimes m}$ are tensor product of keys of $\mathfrak{B}(\lambda)$ such that  the leftmost and rightmost keys in this factorization are  the right key and the left key respectively of the corresponding vertex in $\mathfrak{B}(\lambda)$. Each $i$-directed edge in $\mathfrak{B}({\lambda})$ corresponds to a $\underbrace{i,\dots,i}_m$-directed edge in $\theta_m(\mathfrak{B}(\lambda)) \subseteq \mathfrak{B}(K({\lambda})^{\otimes m})\simeq \mathfrak{B}(m\lambda)$ and vice-versa. See Figure \ref{dilatation} where the dilated crystal $\theta_6(\textsf{KN}((2,1),2)=\{f_{i_l}^{6 }\cdots f_{i_1}^{6 }(K((12,6))\mid i_1,\dots,i_l\in[1,2]\}\setminus\{0\}\subseteq \textsf{KN}((12,6),2)$ each directed edge is defined by the dilated lowering  operator $f_i^6$ when defined.

\begin{obs} By induction on $l\ge 0$, the base case $l=0$ gives $\theta_m(K(\lambda))=K(\lambda)^{\otimes m}$.
Furthermore, if $f_{i_1}\cdots f_{i_l}(K(\lambda))=T$ in $\textsf{KN}(\lambda,n)$, then

\begin{align}\theta_m(T)&=f_{i_1}^m\cdots f_{i_l}^m(K(\lambda^{\otimes m})=K(\sigma_1\lambda)\otimes\cdots\otimes K(\sigma_{m}\lambda)\nonumber\\
 &=K^+(T)\otimes K(\sigma_2\lambda)\otimes\cdots\otimes K(\sigma_{m-1}\lambda)\otimes K^-(T)\mbox{ in $\mathfrak{B}(K(\lambda^{\otimes m})\simeq \textsf{KN}(m\lambda,n)$},
 \end{align}
with $\sigma_1\ge \sigma_2\ge \cdots\ge \sigma_m$ in $B_n^\lambda$. Then $K^+(T)=K(\sigma_1\lambda)$ and $K^-(T)=K(\sigma_m \lambda)$.
In particular, $\theta_m(K(s_i\lambda))=$ $\theta_m(f_i^{\textrm{max}}(K(\lambda))=$ $f_i^{\textrm{max}\,m}(K(\lambda)^{\otimes m})=$ $
K(s_i\lambda)^{\otimes m}$, where $\textrm{max}=\varphi_i(K(\lambda))$. More generally, from \eqref{tensorphi0}, \eqref{tensorphi1}, $\theta_m(K(\sigma\lambda))=K(\sigma\lambda)^{\otimes m}$ for $\sigma\in B_n^\lambda$.
\end{obs}
Consider the setting of \S \ref{baker} and the following proposition.
\begin{prop}{\em \cite[Corollaire 8.1.6]{kash2}} Let $\lambda$ and $\mu$ two dominant weights. If $V(\nu)$ appears in $V(\lambda)\otimes V(\mu)$  then $V(m\nu)$ appears in $V(m\lambda)\otimes V(m\mu)$ for any positive integer $m$.
\end{prop}

Since $\textsf{SSYT}(\lambda^A,2n)\subseteq \bigotimes_{i=1}^k \textsf{SSYT}(\Lambda^A_i+\Lambda^A_{2n-i},2n)$ then $$\textsf{SSYT}(m\lambda^A,2n)\subseteq \bigotimes_{i=1}^k\textsf{SSYT}(m(\Lambda^A_i+\Lambda^A_{2n-i}),2n)$$ where $\displaystyle m\lambda^A=\sum_{i=1}^k m(\Lambda^A_i+\Lambda^A_{2n-i})=m\sum_{i=1}^k(\Lambda^A_i+\Lambda^A_{2n-i})$ for any integer $m>0$.

In what follows we denote by $\theta_m$  the $m$-dilatation map  either on $\textsf{KN}(m\lambda,n)$ or on $\textsf{SSYT}(\lambda^A,2n)$.
One then has, by definition of $\theta_m$ and $E$,
\begin{align}\textsf{KN}(m\lambda,n)&\underset{\theta_m}\hookrightarrow \textsf{KN}(m\lambda,n)&\nonumber\\
T=f_{i_1}\cdots f_{i_l}(K(\lambda))&\mapsto \theta_m(T)=f^m_{i_1}\cdots f^m_{i_l}(K(m\lambda))&,\nonumber\\
\nonumber\\
\textsf{KN}(m\lambda,n)&\underset{E}\hookrightarrow \textsf{SSYT}(m\lambda^A,2n)&\nonumber\\
\theta_m(T)=f^m_{i_1}\cdots f^m_{i_l}(K(m\lambda))&\mapsto E\theta_m(T)={f_{i_1}^A}^m{(f^A_{2n-i_1})}^m\cdots {f^A_{i_l}}^m{(f^A_{2n-i_l})}^m(K(m\lambda^A))&\label{whatneeded1}
\end{align}

On the other hand,
\begin{align}\textsf{KN}(\lambda,n)&\underset{E}\hookrightarrow\textsf{SSYT}(\lambda^A,2n)&\nonumber\\
T=f_{i_1}\cdots f_{i_l}(K(\lambda))&\mapsto E(T)=f^A_{i_1}f^A_{2n-i_1}\cdots f^A_{i_l}f^A_{2n-i_l}(K(\lambda^A))&
\nonumber \\
\nonumber\\
\textsf{SSYT}(\lambda^A,2n)&\underset{\theta_m}\hookrightarrow \textsf{SSYT}(m\lambda^A,2n)&\nonumber\\
E(T)&\mapsto\theta_mE(T)={f_{i_1}^A}^m{(f^A_{2n-i_1})}^m\cdots {f^A_{i_l}}^m{(f^A_{2n-i_l})}^m(K(m\lambda^A))&,\label{whatneeded12}
\end{align}
\begin{align*}
\mbox{ where  } E(T)=f^A_{i_1}f^A_{2n-i_1}\cdots f^A_{i_l}f^A_{2n-i_l}(K(\lambda^A)) \mbox{ in $E(\textsf{KN}(\lambda,n))$}
\end{align*}
Therefore, from \eqref{whatneeded1} and \eqref{whatneeded12},
\begin{align*}\theta_m(E(T))={f_{i_1}^A}^m{(f^A_{2n-i_1})}^m\cdots {f^A_{i_l}}^m{(f^A_{2n-i_l})}^m(K({\lambda^A}^{\otimes m}))
=E(\theta_m(T)).
\end{align*}

It then follows

\begin{prop}\label{commute}The injections $\theta_m$ and $E$ commute in $\textsf{KN}(\lambda,n)$: $\theta_mE=E\theta_m$.
 \[\vcenter{\hbox{\begin{tikzpicture}[scale=0.7]
\node(tl) at (0,3){{\textsf{KN}$(\lambda,n)$}};
\node(tr) at (5,3){{\textsf{KN}$(m\lambda,n)$}};
\node(bl) at (0,0){{\textsf{SSYT}$(\lambda^A,2n)$}};
\node(br) at (5,0){{\textsf{SSYT}$(m\lambda^A, 2n)$}};
\draw[->](tl)--node[above]{${\theta_m}$}(tr);
\draw[->](bl)--node[below]{$\theta_m$}(br);

\draw[->](tl)--node[left]{$E$}(bl);
\draw[->](tr)--node[left]{$E$}(br);
\end{tikzpicture}}}
\]
\end{prop}

\subsection{Baker embedding and right and left key maps commute}
We show that the composition of the Baker embedding $E$ with the $C_n$ right (respect. left) key map $K^+$ (respect. $K^-$) is equal to the composition of the  $A_{2n-1}$ right (respect. left) key map with the Baker embedding $E$. That is, that for all $T\in \textsf{KN}(\lambda,n)$, $E(K^+(T))=K^+(E(T))$ and  $E(K^-(T))=K^-(E(T))$. Henceforth,  $K_+(T)=E^{-1}(K_+(E(T)))$ and
$K^-(T)=E^{-1}(K^-(E(T)))$ where $E^{-1}$ is the reverse of Baker embedding using the type $A_{2n-1}$ reverse column Schensted insertion \cite{agl}.

\begin{teo} \label{Main}
\begin{enumerate}
\item [(a)]
Let $m$ be a positive  integer such that, for each $T\in \textsf{KN}(\lambda,n)$,  the $m$-dilatation map $\theta_m$
 on $\textsf{KN}(\lambda,n)$  gives $\theta_m(T)=K(\sigma_1\lambda)\otimes\cdots\otimes K(\sigma_{m}\lambda)$ for some $\sigma_1\ge$ $\cdots$ $\ge \sigma_m$  in $B_n^\lambda$.
 Then  \begin{equation}\label{mymainresult} \theta_m(E(T))=E(\theta_m(T))=E(K(\sigma_1\lambda))\otimes\cdots\otimes E(K(\sigma_{m}\lambda))=K(\tilde\sigma_1\lambda^A)\otimes\cdots\otimes K(\tilde\sigma_m\lambda^A),
\end{equation}
where $E(K(\sigma_i\lambda))=K(\tilde\sigma_i\lambda^A)\in O_{B^A_n}(\lambda^A)$, $1\le i\le m$.
\item[(b)] $E(K^+(T))=K^+(E(T))$ and  $E(K^-(T))=K^-(E(T))$ for all $T\in \textsf{KN}(\lambda,n)$.

\end{enumerate}
\end{teo}
\begin{proof}(a) Let $r\ge 0$, $T=f_{i_r}\cdots f_{i_1}(K(\lambda))$ and $\theta_m(T)=K(\sigma_1\lambda)\otimes\cdots\otimes K(\sigma_m\lambda)$ with $\sigma_1\ge\cdots\ge\sigma_m$ in $B_n^\lambda$. That is, by the assumption of our statement, the $m$-dilation map $\theta_m$ on $\textsf{KN}({\lambda},n)$ satisfies
\begin{equation}
\theta_{m}:\left\{
\begin{array}
[c]{c}%
\textsf{KN}({\lambda},n)\hookrightarrow \mathfrak{B}(K({\lambda})^{\otimes m})\subset \textsf{KN}({\lambda,n
})^{\otimes m}\\
T\longmapsto K(\sigma_1\lambda)\otimes\cdots\otimes K(\sigma_{m}\lambda)%
\end{array}\label{assumpt}
\right.  %
\end{equation}
where $\sigma_{1}\ge \cdots\ge \sigma_{m}$ in $B_
{n}^{\lambda}$.

 The proof of \eqref{mymainresult} is by induction on $r\ge 0$. For $r=0$,
\begin{align}&\theta_m(E(K(\lambda)))=\theta_m(K(\lambda^A))=K(\lambda^A)^{\otimes m},  \mbox{ by the proprety of $m$-dilatation on keys, Theorem \ref{Th_Dila}, (a)}\nonumber\\
&=E(\theta_m(K(\lambda)))=E(K(\lambda)^{\otimes m}), \mbox{ by Proposition \ref{commute}}\nonumber\\
&\Rightarrow \theta_m(E(K(\lambda)))=E(K(\lambda)^{\otimes m})=K(\lambda^A)^{\otimes m}=E(K(\lambda))^{\otimes m}.\label{base}
\end{align}

Assume by induction on $r\ge 0$, with base case $r=0$ \eqref{base}, that \eqref{mymainresult} holds
\begin{align}\theta_m(E(T))=E(\theta_m(T))&=E(K(\sigma_1\lambda)\otimes\cdots\otimes K(\sigma_{m}\lambda))\nonumber\\
&=E(K(\sigma_1\lambda))\otimes\cdots\otimes E(K(\sigma_{m}\lambda)) \mbox{ by induction}\nonumber\\
&=K(\tilde\sigma_1\lambda^A)\otimes\cdots\otimes K(\tilde\sigma_{m}\lambda^A)),\mbox{ by Proposition \ref{virtualkeyy}}\label{whatneeded2}\\
&\mbox{ with $\tilde\sigma_1\ge\cdots\ge\tilde\sigma_m$ in $B_n^{A,\lambda^A}$.}\nonumber
 \end{align}
Let $f_j(T)\neq 0$, for some $1\le j\le n$. 
Then,  by definition of the $m$-dilatation map $\theta_m$ on $\textsf{KN}({\lambda},n)$, \eqref{assumpt},   there exists
\begin{equation}1\le q_1<\cdots< q_t\le m,\label{sequence}
\end{equation}
with $\langle\sigma_{q_i}\lambda,\alpha_j^\vee\rangle=\varphi_j(K(\sigma_{q_i}\lambda))>0$, $1\le i\le t$,  $$\displaystyle \sum_{1\le i\le t} \langle\sigma_{q_i}\lambda,\alpha_j^\vee\rangle=m, \mbox{  and } \sigma_1\ge\cdots\ge s_j\sigma_{q_1}\ge \cdots\ge s_j\sigma_{q_t}\ge \cdots\ge \sigma_m \mbox{ in } B_n^\lambda,$$
such that
\begin{align*} \theta_m(f_j(T))&=f_j^m\theta_m(T)=f_j^m(K(\sigma_1\lambda)\otimes\cdots\otimes K(\sigma_m\lambda))\\
&=K(\sigma_1\lambda)\otimes\cdots\otimes K(s_j\sigma_{q_1}\lambda)\otimes\cdots\otimes K(s_j\sigma_{q_t}\lambda)\otimes\cdots\otimes K(\sigma_m\lambda).
\end{align*}
Therefore, considering the $m$-dilatation map $\theta_m$ on $\textsf{SSYT}(\lambda^A,2n)$,
\begin{align}\theta_mE(f_j (T))&=\theta_mf_j^EE(T)=(f_j^E)^m\theta_mE(T) \mbox{ by \eqref{theta}}\nonumber\\
&=(f_j^E)^m(K(\tilde\sigma_1\lambda^A)\otimes\cdots\otimes K(\tilde\sigma_{m}\lambda^A)),  \mbox{ by induction on $r\ge 0$, \eqref{whatneeded2}}\nonumber\\
&=\begin{cases}(f^A_j)^m(f^A_{2n-j})^m(K(\tilde\sigma_1\lambda^A)\otimes\cdots\otimes K(\tilde\sigma_{m}\lambda^A)), \; 1\le j<n\\
(f_n^A)^2(K(\tilde\sigma_1\lambda^A)\otimes\cdots\otimes K(\tilde\sigma_{m}\lambda^A)), \; j=n.
\end{cases}\label{computetensor}
\end{align}

For $\sigma\in B_n$ ($\tilde\sigma\in B^A_n$) and $1\le i\le n$, one has by \eqref{lengths}
\begin{align}\label{weneedit}\langle\sigma\lambda,\alpha_i^\vee\rangle&=\langle \textsf{wt}(K(\sigma\lambda)),\alpha_i^\vee\rangle
&=\begin{cases}\langle \sigma\lambda,\alpha_i\rangle=\varphi_i(K(\sigma\lambda))-\varepsilon_i(K(\sigma\lambda)),\;1\le i<n\\
\langle\sigma\lambda,e_n\rangle=\varphi_n(K(\sigma\lambda))-\varepsilon_n(K(\sigma\lambda)),\; i=n,
\end{cases}
\end{align}
by \eqref{length}
\begin{align*}
&=\begin{cases}\varphi^A_i(EK(\sigma\lambda))-\varepsilon^A_i(EK(\sigma\lambda))=\varphi^A_{2n-i}(EK(\sigma\lambda))-\varepsilon^A_{2n-i}(EK(\sigma\lambda)),\;1\le i<n\\
1/2[\varphi^A_n(EK(\sigma\lambda))-\varepsilon^A_n(EK(\sigma\lambda))],\; i=n,
\end{cases}\\
&\mbox{ by Proposition \ref{virtualkeyy}, }\\
&=\begin{cases}\varphi^A_i(K(\tilde\sigma\lambda^A))-\varepsilon^A_i(K(\tilde\sigma\lambda^A))=\varphi^A_{2n-i}(K(\tilde\sigma\lambda^A))
-\varepsilon^A_{2n-i}(K(\tilde\sigma\lambda^A)),\; 1\le i<n\\
1/2[\varphi^A_n(K(\tilde\sigma\lambda^A))-\varepsilon^A_n(K(\tilde\sigma\lambda^A))],\; i=n,
\end{cases}\\
\mbox{ by \eqref{lengths}}\nonumber\\
&=\begin{cases}\langle \tilde\sigma\lambda^A,\alpha^A_i\rangle=\langle\tilde\sigma\lambda^A,\alpha^A_{2n-i}\rangle,\; 1\le i<n\\
1/2\langle\tilde\sigma\lambda^A,\alpha^A_n\rangle,\; i=n.
\end{cases}
\end{align*}

To compute $ (f_j^E)^m(\theta_m E(T))$ \eqref{computetensor}, we need to apply the tensor product rule in \eqref{tensorphi0} and \eqref{tensorphi1}, and from \eqref{weneedit}, one has
\begin{align}\varphi_j^A(\theta_m E(T))&=\begin{cases}\max\{\varphi_j^A(K(\tilde\sigma_k\lambda^A))+\sum_{k<u\le m}\langle\tilde\sigma_u\lambda^A,\alpha_j^A\rangle\}=\\
\max\{\varphi_{2n-j}^A(K(\tilde\sigma_k\lambda^A))+\sum_{k<u\le m}\langle\tilde\sigma_u\lambda^A,\alpha_{2n-j}^A\rangle\},\;1\le j<n,\\
    \max\{\varphi_{n}^A(K(\tilde\sigma_k\lambda^A))+\sum_{k<u\le m}\langle\tilde\sigma_u\lambda^A,\alpha_{n}^A\rangle\},\;j=n,\\
\end{cases},\nonumber\\
\mbox{ by \eqref{weneedit}}\nonumber\\
&=\begin{cases}\max\{\varphi_j(K(\sigma_k\lambda))+\sum_{k<u\le m}\langle\sigma_u\lambda,\alpha_j\rangle\}=\varphi_j(\theta_m(T)), \; 1\le j<n,\\
2\max\{\varphi_{n}(K(\sigma_k\lambda))+\sum_{k<u\le m}\langle\sigma_u\lambda,e_n\rangle\}=2\varphi_n(\theta_m(T)),\;j=n,\\
\end{cases}\label{alsoneeded}
\end{align}

Therefore from the tensor product rule the sequence $1\le q_1<q_2<\cdots<q_t\le m$ in \eqref{sequence} exists, satisfying
$$\displaystyle \sum_{1\le i\le t} \langle\tilde\sigma_{q_i}\lambda^A,\alpha^A_j\rangle=\begin{cases}m, \: 1\le j<n,\\
2m,\; j=n,\end{cases}$$  and $$ \tilde\sigma_1\ge\cdots\ge \tilde s_j\tilde \sigma_{q_1}\ge \cdots\ge \tilde s_j\tilde\sigma_{q_t}\ge \cdots\ge \tilde\sigma_m \mbox{ in } B_n^{A,\lambda^A},$$ such that
\begin{align*}&(f_{2n-j}^A)^m(f^A_j)^m(K(\tilde\sigma_1\lambda^A)\otimes\cdots\otimes K(\tilde\sigma_{m}\lambda^A))=\\
=&f_{2n-j}^A)^m(K(\tilde\sigma_1\lambda^A)\otimes\cdots\otimes K( s_j^A\tilde \sigma_{q_1}\lambda^A)\otimes\cdots\otimes K( s_j^A\tilde \sigma_{q_t}\lambda^A))\otimes\cdots\otimes K(\tilde\sigma_m\lambda^A))\\
=&K(\tilde\sigma_1\lambda^A)\otimes\cdots\otimes K(s^A_{2n-j} s_j^A\tilde \sigma_{q_1}\lambda^A)\otimes\cdots\otimes K(s^A_{2n-j} s_j^A\tilde \sigma_{q_t}\lambda^A))\otimes\cdots\otimes K(\tilde\sigma_m\lambda^A)\\
=&EK(\sigma_1\lambda)\otimes\cdots\otimes EK( s_j \sigma_{q_1}\lambda)\otimes\cdots\otimes EK( s_j \sigma_{q_t}\lambda))\otimes\cdots\otimes EK(\sigma_m\lambda) \mbox{ by Proposition \ref{virtualkeyy}}
\end{align*}

and

\begin{align*}&(f^A_n)^{2m}(K(\tilde\sigma_1\lambda^A)\otimes\cdots\otimes K(\tilde\sigma_{m}\lambda^A))=\\
=&K(\tilde\sigma_1\lambda^A)\otimes\cdots\otimes K( s_n^A\tilde \sigma_{q_1}\lambda^A)\otimes\cdots\otimes K( s_n^A\tilde \sigma_{q_t}\lambda^A))\otimes\cdots\otimes K(\tilde\sigma_m\lambda^A)\\
=&EK(\sigma_1\lambda)\otimes\cdots\otimes EK( s_n \sigma_{q_1}\lambda)\otimes\cdots\otimes EK( s_n \sigma_{q_t}\lambda))\otimes\cdots\otimes EK(\sigma_m\lambda).
\end{align*}
Thus
\begin{align*}
\theta_mE(f_j (T))&=
E\theta_m(f_j (T))\\
&=\begin{cases}E(K(\sigma_1\lambda)\otimes\cdots\otimes K( s_j \sigma_{q_1}\lambda)\otimes\cdots\otimes K( s_j \sigma_{q_t}\lambda))\otimes\cdots\otimes K(\sigma_m\lambda)),\; 1\le j<n,\\
EK(\sigma_1\lambda)\otimes\cdots\otimes K( s_n \sigma_{q_1}\lambda)\otimes\cdots\otimes K( s_n \sigma_{q_t}\lambda))\otimes\cdots\otimes K(\sigma_m\lambda),\; j=n.
\end{cases}\\
&=\begin{cases}EK(\sigma_1\lambda)\otimes\cdots\otimes EK( s_j \sigma_{q_1}\lambda)\otimes\cdots\otimes EK( s_j \sigma_{q_t}\lambda))\otimes\cdots\otimes EK(\sigma_m\lambda),\; 1\le j<n,\\
EK(\sigma_1\lambda)\otimes\cdots\otimes EK( s_n \sigma_{q_1}\lambda)\otimes\cdots\otimes EK( s_n \sigma_{q_t}\lambda))\otimes\cdots\otimes EK(\sigma_m\lambda),\; j=n.
\end{cases}
\end{align*}

Thereby $K^+(E(f_j(T))=EK^+(f_j(T))$ and $K^-(E(f_j(T))=EK^-(f_j(T))$ from which $(b)$ follows.
\end{proof}

\begin{ex}\label{KN(1,0),2} Let $T\in \textsf{SSYT}(\lambda^A,2n)$. If $K^+(T)\in O_{B_n}(\lambda^A)$  we may not conclude that $T\in E(\textsf{KN}(\lambda,n))$.
For example, below, on the left one has the crystal $\textsf{KN}((1,0),2)=O(B_2)$, in the middle the virtual crystal $E \textsf{KN}((1,0),2)=O(B_2^A)$ and on the right its embedding on the crystal $\textsf{SSYT}(\lambda^A=(2,1,1), 4)$.
$$\begin{array}{ccccccccc}
\includegraphics[height=10cm]{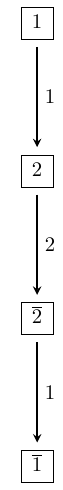}&&&\includegraphics[height=10cm]{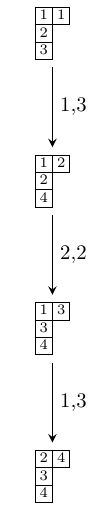}&&&\includegraphics[height=12cm]{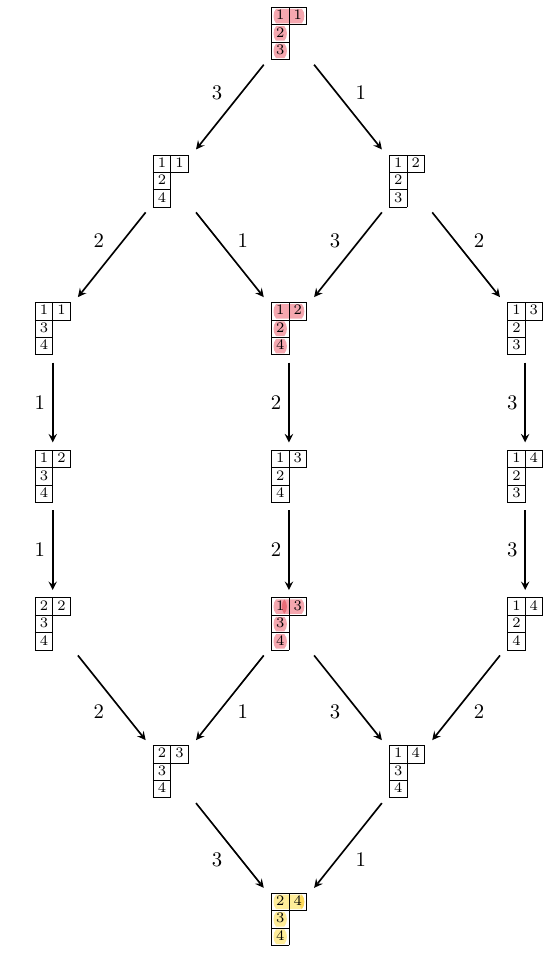}
\end{array}$$
In the crystal $\textsf{SSYT}(\lambda^A=(2,1,1), 4)$ the tableaux \[S=
\YT{0.19in}{}{
	    	{{1},{3}},
			{{2}},
			{{4}}},\quad
T=\YT{0.19in}{}{
	    	{{1},3},
			{{3}},
			{{4}}}=E(\overline 2)\in E \textsf{KN}((1,0),2)
			\] have the same right key $K^+(S)=K^+(T)=K^-(T)=T\in O(B_2^{A})$ while $S\notin E\textsf{KN}((1,0),2)$ $=O(B_2^A).$ See also
\cite[Fig. 5.3]{bump2017crystal}.

\end{ex}

\begin{obs}\label{virtualLS}
If $m$ has enough factors (in general at least equal to the least comom multiple of the $i$-string lengths of $\textsf{KN}(\lambda,n)$), $\theta_m$  defines a bijection between the Kashiwara crystal $\textsf{KN}(\lambda,n)$ and the crystal $\textbf{B}(\lambda)$ of L-S paths of shape $\lambda$ as $B_n$-paths where  each vertex of $\theta_m(\textsf{KN}(\lambda,n))$, $K(\tau_0\lambda)^{\otimes x_0}\otimes (K(\tau_1\lambda))^{\otimes x_1}\ge \cdots\otimes (K(\tau_r\lambda))^{\otimes x_r}$ such that $\tau_0>\cdots>\tau_r$ in the Bruhat order in $W/W_\lambda$ and $x_0+x_1+\cdots+x_r=m$, provides the L-S path in $\textbf{B}(\lambda)$ with parameters $(\tau_0>\tau_1>\cdots>\tau_r;0<a_1=x_0/m<a_2=a_1+x_1/m<\cdots<a_{r}=a_{r-1}+x_{r-}/m <1=a_r+x_r/m)$. See Figure \ref{dilatation}.

If we consider $M>0 $  a multiple of $m$ and with enough factors (in general at least equal  to the least common multiple of the $i$-string lengths of $\textsf{SSYT}(\lambda^A,2n)$) and calculate $\theta_M(\textsf{SSYT}(\lambda^A,2n))$ then  this dilatation of $\textsf{SSYT}(\lambda^A,2n))$ yields bijectively the set $\mathbf{B}(\lambda^A)$ of L-S paths, as $\mathfrak{S}_{2n}$-paths. As a crystal $\mathbf{B}(\lambda^A)$ is isomorphic to
the Kashiwara crystal $\textsf{SSYT}(\lambda^A,2n)$ containing $E\textsf{KN}(\lambda,n)$.
 On the other hand, $\theta_M(E\textsf{KN}(\lambda,n))=E\theta_M(\textsf{KN}(\lambda,n))$, and both $\theta_m(\textsf{KN}(\lambda,n))$ and $\theta_M(\textsf{KN}(\lambda,n))$ provide the same crystal $\textbf{B}(\lambda)$ of L-S paths of shape $\lambda$, as $B_n$-paths. Henceforth,
 the crystal $\textbf{B}(\lambda)$ of L-S paths of shape $\lambda$, as $B_n$-paths, is embedded in the crystal $\mathbf{B}(\lambda^A)$ of L-S paths, as $\mathfrak{S}_{2n}$-paths, where the parameters $\tau_0>\cdots>\tau_r$ are sent to $\tilde\tau_0>\cdots>\tilde\tau_r$ and the new rational parameters depend on the $M$-folder tensor products in $\theta_M(\textsf{SSYT}(\lambda^A,2n))$.
Thanks to Theorem \ref{Main}, $E$  embeds $\textbf{B}(\lambda)$ into the crystal $\mathbf{B}(\lambda^A)$ of L-S paths as $\mathfrak{S}_{2n}$-paths, since for $T\in \textsf{KN}(\lambda,n)$,
$$E\theta_M(T)=E((K(\tau_0\lambda))^{\otimes x'_0}\otimes (K(\tau_1\lambda))^{\otimes x'_1}\otimes \cdots\otimes (K(\tau_r\lambda))^{\otimes x'_r})$$ $$= K(\tilde\tau_0\lambda)^{\otimes x'_0}\otimes (K(\tilde\tau_1\lambda))^{\otimes x'_1}\otimes \cdots\otimes (K(\tilde\tau_r\lambda))^{\otimes x'_r}=\theta_M(E(T))\in\theta_M(\textsf{KN}(\lambda,n)),$$
where $\tilde\tau_i\in B_n^A(\lambda^A)$.
\end{obs}

\subsection{Virtual symplectic right and left keys examples}\label{sec:virtualkeys}
\subsubsection{Baker virtualization of the symplectic Kashiwra crystal} Consider $n=5$, and the KN tableau $T$ of shape $\lambda=\Lambda_4+\Lambda_3=(2,2,2,1,0)$ in $\textsf{KN}(\Lambda_4+\Lambda_3,5)$,
\[
T=\YT{0.19in}{}{
	    	{{1},{2}},
			{{3},{\overline{5}}},
			{{\overline{4}},{\overline{3}}},
			{{\overline{3}}}},\quad { \textsf{wt}(T)=(1,1,-1,-1,-1)}.
\]
Labelling the columns of $T$ from left to right as $C_1$ and $C_2$, we obtain $E(T)$ with shape
$\lambda^A=\Lambda^A_7+\Lambda^A_3+\Lambda^A_6+\Lambda^A_4$ where for convenience we consider the alphabet $[2n]$ represented by $\{1<2<\cdots<n<\bar n<\cdots<\bar 1\}$,:

\begin{align}\label{ET}
\psi(C_2)=\YT{0.19in}{}{
	    	{{1},{2}},
			{{2},{\overline{5}}},
			{{4},{\overline{3}}},
			{{\overline{5}}},
			{{\overline{4}}},
			{{\overline{3}}},
			{{\overline{1}}}},
\psi(C_1)=\YT{0.19in}{}{
	    	{{1},{1}},
			{{2},{3}},
			{{5},{\overline{4}}},
			{{\overline{5}}, {\overline{2}} },
			{{\overline{4}}},
			{{\overline{3}}}} \Rightarrow
{E}(T)=[\emptyset\leftarrow w(\psi(C_2)) \leftarrow w(\psi(C_1))]=
\YT{0.19in}{}{
	    	{{1},{1},{1},{2}},
			{{2},{2},{4},{\overline{5}}},
			{{3},{\overline{5}},{\overline{4}},{\overline{3}}},
			{{5},{\overline{4}},{\overline{1}}},
			{{\overline{5}},{\overline{3}}},
			{ {\overline{4}},{\overline{2}} },
			{{\overline{3}}}}.
\end{align}	

We get $E(T)\in \textsf{SSYT}(\lambda^A, 10)$ and one may apply to $E(T)$ any type $A$ procedure  to compute its right and left keys  in $\textsf{SSYT}(\lambda^A, 10)$. For example, JDT procedure gives

\begin{align} \label{virtualkeys}K_+(E(T))=\YT{0.19in}{}{
	    	{{1},           {2},            {2},            {2}},
			{{2},           {4},           {\overline{5}}, {\overline{5}}},
			{{4},           {\overline{5}},{\overline{3}}, {\overline{3}}},
			{{\overline{5}},{\overline{4}},{\overline{1}}},
			{{\overline{4}},{\overline{3}}},
			{\overline{3}, {\overline{1}} },
			{\overline{1}}}\qquad
 K_-(E(T))=\YT{0.19in}{}{
	    	{{1},           {1},            {1},            {1}},
			{{2},           {2},           {2}, 2},
			{3,           {{5}},{\overline{4}}, {\overline{4}}},
			{{{5}},{\overline{5}},{\overline{3}}},
			{{\overline{5}},{\overline{4}}},
			{\overline{4}, {\overline{3}} },
			{\overline{3}}}
\end{align}

Using $Q_\lambda$, which is uniquely determined by $\lambda=\Lambda_4+\Lambda_3$ \cite[Proposition 1, Corollary 1]{att}, to perform the reverse column Schensted insertion on  $K_+(E(T))$ and $ K_-(E(T))$ respectively provides the
image under $\psi$ of two pairs of KN columns $C'_1$, $C'_2$ and $C"_1$,$C"_2$ respectively. Applying $\psi^{-1}$ to each column results in:

\[
Q_\lambda=\YT{0.19in}{}{
	    	{{1},{4},{11},{15}},
			{{2},{5},{12},{16}},
			{{3},{6},{13},{17}},
			{{7},{14},{18}},
			{{8},{19}},
			{{9},{20} },
			{{10}}} \Rightarrow
\psi(C'_1)=\YT{0.19in}{}{
	    	{2,2},
			{4,\overline 5},
			{\overline 5, \overline 3},
			{\overline 4,\overline 1},
			{{\overline 3}},
			{\overline 1 }},\quad
\psi(C'_2)=\YT{0.19in}{}{
	    	{1,2},
			{2,\overline 5},
			{4, \overline 3},
			{\overline 5},
			{{\overline 4}},
			{\overline 3 },
{\overline 1 }}\overset{Theorem \;\ref{Main}}\Longrightarrow K_+(T)=C'_1C'_2=
\YT{0.19in}{}{
	    	{2,2},
			{\overline 5,\overline 5},
			{\overline 3, \overline 3},
{\overline 1 }}\\
\]
\[\Rightarrow K_+(T)=K(-1,2,-2,0,-2)=K(w\lambda), \mbox{ with $w=[2\,\overline 5\,\overline 3\,\overline 1\, 4]\in B_5$}\]  and $T \in \mathfrak{B}_w(\lambda) \subseteq \textsf{KN}(\lambda,5).$ 

Therefore  $\mathfrak{B}_w(\lambda)$ $\subseteq$ $\textsf{KN}(\lambda,5)$ is  the set KN tableaux in $\textsf{KN}(\lambda,5)$ \emph{standard} (recall  the definition in Section \S \ref{standardlaklitt}) in the symplectic Schubert variety $X_w\subseteq Sp(2n,\mathbb{C})/B$ where $B$ is a Borel subgroup of $Sp(2n,\mathbb{C})$.
Note $K^+(E(T))=E(K^+(T))=K(\tilde w\lambda^A)=K(1,4,0,2,0,4,2,4,0,3)$, the $A_{2n-1}$ key tableau on the LHS of \eqref{virtualkeys}, where $\tilde w=2\,\overline 5\,\overline 3\,\overline 1\, 4\,\overline 4\,1\,3\,5\,\overline 2 $ $\in \mathfrak{S}_{10}$. Therefore, recalling Proposition \ref{prop:virtualdemazure}, $E(T )\in E(\mathfrak{B}_w(\lambda))=$ $\mathfrak{B}_{\tilde w}(\lambda^A) \cap E(\textsf{KN}(\lambda,5))\subseteq $ $ \textsf{SSYT}(\lambda^A,10)$,  and the  SSYT tableaux in $\mathfrak{B}_{\tilde w}(\lambda^A)\cap {E}\textsf{KN}(\lambda,5)$ are \emph{standard} in the Schubert variety $X_{\tilde w}\subseteq Gl(n,\mathbb{C})/{\tilde B}$, $\tilde B$ a Borel subgroup of $Gl(n,\mathbb{C})$ such that  $B=\tilde B\cap Sp(2n,\mathbb{C})$. (The Borel subgroups of $Sp(2n,\mathbb{C})$ are obtained by intersecting the Borel subgroups of $Gl(2n,\mathbb{C})$ with $Sp(2n,\mathbb{C})$. All  Borel subgroups of $G$ are conjugate to each other.)
That is, $T\in \textsf{KN}(\lambda,n)$ is  \emph{standard} on the Schubert variety $X_w\subseteq Sp(2n,\mathbb{C})/B$ if and only if $E(T)$ is \emph{standard} on the Schubert variety $X_{\tilde w}\subseteq Gl(n,\mathbb{C})/{\tilde B}$.

 And
\[Q_\lambda\Rightarrow\psi(C"_1)=\YT{0.19in}{}{
	    	{1,1},
			{2,2},
			{ 5, \overline 4},
			{\overline 5,\overline 3},
			{{\overline 4}},
			{\overline 3 }},\quad
\psi(C"_2)=\YT{0.19in}{}{
	    	{1,2},
			{2,2},
			{3, \overline 4},
			{\overline 5},
			{{\overline 5}},
			{\overline 4 },
{\overline 3 }}\overset{Theorem \;\ref{Main}}\Longrightarrow K^-(T)=C"_1C"_2=
\YT{0.19in}{}{
	    	{1,1},
			{2, 2},
			{\overline 4, \overline 4},
{\overline 3 }}
\]
\[\Rightarrow  K^-(T)=K(2,2,-1,-2,0)=K(\sigma\lambda), \mbox{ with $\sigma=[1\,2\,\overline 4\,\overline 3\, 5]\in B_5$ and $T \in \mathfrak{B}^\sigma(\lambda)$}.\]
Therefore  $\mathfrak{B}^\sigma(\lambda)$ $\subseteq$ $\textsf{KN}(\lambda,5)$ is  the set KN tableaux in $\textsf{KN}(\lambda,5)$ \emph{standard} on the opposite symplectic Schubert variety $X^\sigma\subseteq Sp(2n,\mathbb{C})/B$.

Again $K^-(E(T))=E(K^-(T))=K(\tilde \sigma\lambda^A)=K(1,4,0,2,0,4,2,4,0,3)$, the $A_{2n-1}$ key tableau on the RHS of \eqref{virtualkeys}, where $\tilde \sigma=2\,\overline 5\,\overline 3\,\overline 1\, 4\,\overline 4\,1\,3\,5\,\overline 2 $ $\in \mathfrak{S}_{10}$. Therefore, recalling Proposition \ref{prop:virtualdemazure}, $E(T )\in E(\mathfrak{B}_\sigma(\lambda))=$ $\mathfrak{B}_{\tilde \sigma}(\lambda^A) \cap E(\textsf{KN}(\lambda,5))\subseteq $ $ \textsf{SSYT}(\lambda^A,10)$,  and the  SSYT tableaux in $\mathfrak{B}_{\tilde \sigma}(\lambda^A)\cap {E}\textsf{KN}(\lambda,5)$ are \emph{standard} on the Schubert variety $X_{\tilde \sigma}\subseteq Gl(n,\mathbb{C})/{\tilde B}$, $\tilde B$ a Borel subgroup of $Gl(n,\mathbb{C})$ such that  $B=\tilde B\cap Sp(2n,\mathbb{C})$.
That is, $T\in \textsf{KN}(\lambda,n)$ is  \emph{standard} on the Schubert variety $X_\sigma\subseteq Sp(2n,\mathbb{C})/B$ if and only if $E(T)$ is \emph{standard} on the Schubert variety $X_{\tilde \sigma}\subseteq Gl(n,\mathbb{C})/{\tilde B}$.

Since  $T \in \mathfrak{B}_w(\lambda)$ and $T \in \mathfrak{B}^\sigma(\lambda)$, we conclude that $T$ is \emph{standard} on  the Richardson $X_w\cap X^\sigma $  in full flag variety $Sp(2n,\mathbb{C})/B$. More generally $\mathfrak{B}_w(\lambda)\cap\mathfrak{B}^\sigma(\lambda)$ $\subseteq$ $\textsf{KN}(\lambda,5)$ is  the set KN tableaux in $\textsf{KN}(\lambda,5)$ \emph{standard} on the  symplectic Richardson variety $X_w\cap X^\sigma \subseteq Sp(2n,\mathbb{C})/B$.
\medskip

 Applying the SJDT procedure to $T$ (or the direct way) provides
 \[T=\YT{0.19in}{}{
	    	{{1},{2}},
			{{3},{\overline{5}}},
			{{\overline{4}},{\overline{3}}},
			{\overline{3}}}\longrightarrow
\YT{0.19in}{}{
	    	{{1},{2}},
			{{2},\overline{5}},
			{\overline{4},\overline{3}},
			{,\overline{2}}},\quad
r\YT{0.19in}{}{
	    	{{2}},
			{\overline{5}},
			{\overline{3}},
			{\overline{2}}}=
\YT{0.19in}{}{
	    	{{2}},
			{\overline{5}},
			{\overline{3}},
			{\overline{1}}},\quad
\ell\YT{0.19in}{}{
	    	{{1}},
			{{3}},
			{\overline{4}},
			{\overline{3}}}=
\YT{0.19in}{}{
	    	{{1}},
			{{2}},
			{\overline{4}},
			{\overline{3}}}
 \]
and we get the same output as the procedure by virtualization.
\subsubsection{  Baker virtualization of the  crystal of Lakshmibai-Seshadri paths as $B_n$ paths}For the dilatation of size $m=6$, equal to the least common multiple of the maximal $i$-string lengths of $\textsf{KN}(\lambda,3)$, $\lambda=\Lambda_2+\Lambda_1=(2,1)$,
 we  give in Figure \ref{laks} the 
   crystal $\textbf{B}(\Lambda_)$ of L-S paths of shape $\lambda=(2,1)$ as  $B_2$-paths isomorphic to $\textsf{KN}(\lambda,3)$. To embed it
      in the crystal $\textbf{B}(\lambda^A)$ of L-S paths of shape $\lambda^A=\Lambda^A_3+\Lambda^A_1+2\Lambda^A_2$ as $\mathfrak{S}_{4}$-paths, we may follow Remark \ref{virtualLS}. Embed, by $E$, $\textsf{KN}(\lambda,3)$ in $\textsf{SSYT}(\lambda^A,6)$ and choose $M$, with enough factors, in general     the least common multiple of $m$ and the $i$-string lengths of $\textsf{SSYT}(\lambda^A,6)$, to apply the dilatation of size $M$ to $\textsf{SSYT}(\lambda^A,6)$, or dilate by $M$, as just defined, $\textsf{KN}(\lambda,3)$ and embed it into $\textsf{SSYT}(M\lambda^A,6)$ by $E$.

\section{Final remarks} \label{sec:final}Recently  a new method to compute keys was provided in \cite{krv23}. Calculations on KN tableaux support evidence that this type A new method can be adapted to type KN tableaux using the split form.

	\section{Acknowledgements}
This work was partially supported by the Centre for Mathematics of the University of Coimbra (funded by the Portuguese Government through FCT/MCTES, DOI 10.54499/UIDB/00324/2020). The second author (JMS)  was  supported by FCT, through the grant PD/BD/142954/2018, under POCH funds, co-financed by the European Social Fund and Portuguese National Funds from MCTES.
This
work has been partially conducted when the second author (JMS) was under  PhD supervision  by the first author (OA), and a member of the Centre for Mathematics of
 University of Coimbra.

 The first author is grateful to Thomas Gobet,  C\'edric Lecouvey, Cristian Lenart, Jake Levinson, Peter Littelmann, and Travis Scrimshaw for enlightening discussions. She also thanks to Inês Rodrigues for a Latex drawing.

	\printbibliography


	
\end{document}